\documentclass[11pt,twoside]{article}
\usepackage{amssymb}
\usepackage{amssymb,amsmath,amsthm,amsfonts,mathrsfs,hyperref}
\usepackage{times}
\usepackage{enumerate}
\usepackage{cite,titletoc}
\usepackage{graphicx}
\usepackage{float}
\usepackage{epstopdf}
\usepackage{subcaption}

\allowdisplaybreaks

\newcommand{\ave}[1]{\langle #1\rangle}

\newcommand{\R}{\mathbb{R}}

\def\f{\frac}

\def\d{d}

\def\Xint#1{\mathchoice
{\XXint\displaystyle\textstyle{#1}}%
{\XXint\textstyle\scriptstyle{#1}}%
{\XXint\scriptstyle\scriptscriptstyle{#1}}%
{\XXint\scriptscriptstyle\scriptscriptstyle{#1}}%
\!\int}
\def\XXint#1#2#3{{\setbox0=\hbox{$#1{#2#3}{\int}$}
\vcenter{\hbox{$#2#3$}}\kern-.5\wd0}}

\def\dashint{\Xint-}

\newcommand{\esssup}[0]{\operatornamewithlimits{ess\,sup}}
\newcommand{\essinf}[0]{\operatornamewithlimits{ess\,inf}}

\hypersetup{colorlinks=true,linkcolor=blue,citecolor=red,urlcolor=cyan}

\theoremstyle{plain}
\newtheorem{thm}[equation]{Theorem}
\newtheorem{lem}[equation]{Lemma}
\newtheorem{prop}[equation]{Proposition}
\newtheorem{cor}[equation]{Corollary}

\theoremstyle{definition}
\newtheorem{defn}[equation]{Definition}

\theoremstyle{remark}
\newtheorem{rem}[equation]{Remark}

\numberwithin{equation}{section}

\title{Properties and applications of partial multiple weights for fractional integrals}

\author{Wang Dinghuai$^1$\quad Yin Huicheng$^{2,}$\footnote{Wang Dinghuai (\texttt{Wangdh1990@126.com}) is supported by NSFC (No.~12101010), and
Yin Huicheng
    (\texttt{huicheng@} \texttt{nju.edu.cn}, \texttt{05407@njnu.edu.cn}) is supported by the NSFC
    (No.~12331007) and by the National key research and development program of China (No. 2020YFA0713803). }
    \vspace{0.5cm}\\
\small
1.  School of Mathematics and Statistics, Anhui Normal University, Wuhu 241002, China.\\
\small
2. School of Mathematical Sciences, Nanjing Normal University, Nanjing 210023, China.\\
}
\vspace{0.5cm}

\begin{document}

\date{}
\maketitle
\thispagestyle{empty}

\begin{abstract}
In this paper, through the introduction of partial multiple weights,
we firstly study the related Rubio de Francia extrapolation theorem within the framework of partial Muckenhoupt classes
and further obtain the corresponding extrapolation theorem for two types of off-diagonal estimates.
Secondly, we establish some weighted estimates for fractional integrals associated with partial Muckenhoupt weights.
As applications, several basic inequalities (including the Fefferman-Phong inequality, the degenerate Poincar\'{e} inequality
and the Caffarelli-Kohn-Nirenberg inequality) related to partial Muckenhoupt weights are derived.
Meanwhile, our results can give the characterization of the commutators of fractional integrals,
which yields a partial answer to an open question proposed by D. Cruz-Uribe in the paper [D. Cruz-Uribe, Two weight inequalities
for fractional integral operators and commutators, World Scientific Publishing Co. Pte. Ltd., Hackensack, NJ, 2017, 25-85].

\vskip 0.2 true cm


\qquad  \quad diagonal estimate, off-diagonal estimate, commutator

\vskip 0.2 true cm
\noindent
\end{abstract}

\tableofcontents

\section{Introduction}
\subsection{Open question and main results}\label{1.1-C}

For $x\in\R^d$ ($d\ge 1$) and $0<\alpha<d$, the fractional integral $I_{\alpha}$ of a measurable function $f$ is defined by
$$I_{\alpha}(f)(x): = \int_{\R^d} \frac{f(y)}{|x-y|^{d-\alpha}} \, dy.$$
Let $[b, I_{\alpha}]:=bI_{\alpha}(\cdot)-I_{\alpha}(b\cdot)$ be the commutator of fractional integral $I_{\alpha}$
with the measurable function $b(x)$. In addition, for a cube $Q\subset \R^d$
and a continuous convex function $\Phi:[0,\infty)\rightarrow [0,\infty)$
which is strictly increasing and fulfills $\Phi(0)=0$ and $\frac {\Phi(t)}{t}\rightarrow \infty$ as $t\rightarrow \infty$,
we define the normalized Luxemburg norm of
$f$ on $Q$ as
$$\|f\|_{\Phi,Q}:=\inf \big\{\lambda>0: \frac{1}{|Q|}\int_{Q}\Phi\big(\frac{|f(x)|}{\lambda}\big)dx\leq 1\big\}.$$
A function $f\in\text{BMO}$ (see \cite{JN1961} and \cite{FS1972}) means
$\| f \|_{\rm BMO}: = \sup_{Q} \frac{1}{|Q|} \int_Q |f(x)-\ave{f}_{Q}| \, dx < \infty,$
where the supremum is taken over all cubes $Q$ in $\R^d$ and $\ave{f}_{Q}:=\f{1}{|Q|}\int_Qf(x)dx$.
In \cite{CM2012} (see also \cite{CU2016}), the authors showed that if $b \in \text{BMO}$, $0<\alpha<d,$ $1 < p \leq q < \infty$,
and $(u, \sigma)$ is a pair of weights satisfying
\begin{equation}\label{AB}
\sup_{Q} |Q|^{\frac{\alpha}{d} + \frac{1}{q} - \frac{1}{p}} \|u^{\frac{1}{q}}\|_{A,Q} \|\sigma^{\frac{1}{p'}}\|_{B,Q} < \infty,
\end{equation}
\begin{equation}\label{CD}
\sup_{Q} |Q|^{\frac{\alpha}{d} + \frac{1}{q} - \frac{1}{p}} \|u^{\frac{1}{q}}\|_{C,Q} \|\sigma^{\frac{1}{p'}}\|_{D,Q} < \infty,
\end{equation}
where $\f{1}{p'}+\f{1}{p}=1$, $A(t) = t^{q} {\ln}^{2q - 1 + \delta}(e + t)$, $B(t) = t^{p'} {\ln}^{p' - 1 + \delta}(e + t),$
$C(t) = t^{q} {\ln}^{q - 1 + \delta}(e + t)$ and $D(t) = t^{p'} {\ln}^{2p' - 1 + \delta}(e + t)$ with $\delta>0$ being any fixed constant,
then the strong $(p,q)$ estimate $[b, I_{\alpha}](\cdot \sigma): L^{p}(\sigma)\rightarrow L^{q}(u)$ holds.
In particular, at the end of the paper \cite{CU2016}, the author proposed the following open question and
pointed out that ``nothing is known about this question but it merits further investigation"

{\bf Open question:} Can anything be said about $b$ if there exists a pair of weights $(u, \sigma)$ (or
perhaps a family of such pairs) such that $[b,I_{\alpha}](\cdot \sigma): L^{p}(\sigma) \rightarrow L^{q}(u)$?
\qquad \qquad \qquad \qquad \qquad \qquad \qquad {\bf (OQ)}

Many previous results have indicated that the BMO space is perhaps an appropriate condition for studying the boundedness of commutators (see \cite{C2016, C1982, CRW1976, GLW2020,W2023}). However, so far the open question {\bf (OQ)} remains unsolved.
In this paper, by introducing some kinds of partial multiple weights and deriving new Rubio de Francia extrapolation theorems (see \cite{RDF}) ,
we will investigate  {\bf (OQ)}. Before stating main results, we next list or introduce some necessary notations
and conceptions for reader's convenience.

{\bf $\bullet$} For $0< p<\infty,$ the Lebesgue space $L^{p}(\R^d)$
and the weak Lebesgue space $L^{p,\infty}(\R^d)$
are defined by $\{f: \|f\|_{L^{p}(\R^d)}:=\big(\int_{\R^d}|f(x)|^{p}dx\big)^{\frac{1}{p}}<\infty\}$
and
$\{f: \|f\|_{L^{p,\infty}(\R^d)}:=\sup_{t>0}t|\{x\in \R^d: |f(x)|>t\}|^{\frac{1}{p}}<\infty\},$ respectively.
In addition, for nonnegative measurable function $\sigma(x)$, $L^{p}(\sigma)=\{f: \|f\|_{L^{p}(\sigma)}:=\big(\int_{\R^d}|f(x)|^{p}\sigma(x)dx\big)^{\frac{1}{p}}<\infty\}$.

{\bf $\bullet$} For $0< p<\infty$ and $0< q\leq \infty$, the Lorentz space $L^{p,q}(\R^d)$ is defined by
$\{f: \|f\|_{L^{p,q}(\R^d)}:=\big\|t^{1-\frac1q}|\{x\in \R^d: |f(x)|>t\}|^{\frac{1}{p}}\big\|_{L^{q}(0,\infty)}<\infty\}.$

{\bf $\bullet$} For the measurable set $E, 0<|E|<\infty,$ define
$\ave{f}_{E}:=\frac{1}{|E|}\int_{E}f(x)dx.$
When $1\leq s< p\leq \infty$, $f\in M^{p}_{s}(\R^d)$ (Morrey space) means
$\|f\|_{M^{p}_{s}(\R^d)}:=\sup_{Q}|Q|^{\frac{1}{p}}\ave{|f|^{s}}^{\frac1s}_{Q}<\infty$
for any cube $Q\subset\R^d$.

{\bf $\bullet$} A non-negative locally integrable function
$w\in A_p = A_p(\mathbb{R}^{d})$ ($1 < p < \infty$) means that for any cube $Q\subset\mathbb{R}^{d}$,
$[w]_{A_p} := \sup_{Q}\, \ave{w}_Q  \ave{w^{1-p'}}^{p-1}_Q
< \infty$ holds. If $[w]_{A_\infty}:=\sup_Q \, \ave{w}_Q \exp\big( \ave{\log w^{-1}}_Q  \big)<\infty$
or $[w]_{A_1}:= \sup_Q \, \ave{w}_Q\esssup_Q w^{-1}<\infty$,
then one says $w\in A_\infty(\mathbb{R}^{d})$ or $w\in A_1(\mathbb{R}^{d})$, respectively.

{\bf $\bullet$}  For $0<q\leq \infty$ and $\vec{p}=(p_1, \ldots,  p_m)$ with $1 \le p_1, \ldots, p_{m} \le \infty$,
we call $\vec{w}=(w_1, \ldots, w_m)\in A_{\vec p, q} = A_{\vec p,q}(\R^{d})$ if
$0<w_i<\infty, \text{a.e., $i = 1, \ldots, m$}$, and
\begin{equation}\label{YH-1}
[\vec{w}]_{A_{\vec p, q}}
:=\sup_Q \, \ave{w^q}_Q^{\frac 1 q} \prod_{i=1}^{m} \ave{w_i^{-p_i'}}_Q^{\frac 1{p_i'}}< \infty\quad\text{for $w := \prod_{i=1}^{m} w_i.$}
\end{equation}
If $p_i = 1$ or $q = \infty$, then
$\ave{w_i^{-p_i'}}_Q^{\frac 1{p_i'}}$ or $\ave{w^q}_Q^{\frac 1 q}$ in \eqref{YH-1}
will be replaced by $\esssup_Q w_i^{-1}$ or $\esssup_Q w$, respectively.

{\bf $\bullet$} For $\vec{r}=(r_{1},\ldots,r_{m+1})$ with $1\leq r_{1},\ldots,r_{m+1}\leq \infty$,
$0<q\leq \infty$ and $\vec{p}=(p_1, \ldots,  p_m)$ with $1 \le p_1, \ldots, p_{m} \le \infty$,
we say $\vec w\in A_{(\vec p, q), \vec r}$ if
\begin{equation}\label{YH-2}
[\vec w]_{A_{(\vec p, q), \vec r}}:=\sup_{Q}\big\langle w^{\frac{q r'_{m+1}}{r'_{m+1}-q }} \big\rangle^{\frac{1}{q }-\frac{1}{r'_{m+1}}}_{Q}
\prod_{i=1}^{m} \big\langle w_{i}^{\frac{p_{i}r_{i}}{r_{i}-p_{i}}} \big\rangle^{\frac{1}{r_{i}}-\frac{1}{p_{i}}}_{Q}<\infty.
\end{equation}
When $r_{m+1}= 1$, the term corresponding to $w$ in \eqref{YH-2} will be replaced
by $\ave{w^q}_{Q}^\frac{1}{q}$. Analogously, when $p_i = r_i$, the term corresponding to $w_i$ is
replaced by $\esssup_Q w_i^{-1}$. Obviously, $A_{(\vec p,q), (1,...,1)}=A_{\vec p, q}$.

It is pointed out that the definition of $A_{\vec{p},q}$ essentially comes from the standard $m$-linear definition
in \cite{M2009}, with the key difference being that the exponents $p_i$ are allowed to be infinite. In addition, for the number $q$ with $1/q=1/p:=\sum_{i}1/p_{i}$, $A_{\vec{p},q}=A_{\vec p,p}$ has been provided in \cite{LMV2021}.

Next we introduce the partial multiple weight which will be repeatedly used later.
\begin{defn}\label{defn:Apqu}
Given $0<q\leq \infty$, $\vec{p}=(p_1, \ldots,  p_m)$ with $1 \le p_1, \ldots, p_{m} \le \infty$,
and $\vec{w}=(w_1, \ldots, w_m)$.
We call $\vec{w}\in A^{u}_{\vec p, q}$ if $(\vec w,u)\in A_{(\vec p,\infty),q},$
meanwhile $[\vec{w}]_{A^{u}_{\vec{p}, q}} := [(\vec{w},u)]_{A_{(\vec{p},\infty), q}}$ is defined.
For $\vec{r}=(r_{1},\ldots,r_{m+1})$ with $1\leq r_{1},\ldots,r_{m+1}\leq \infty$, define
$[\vec{w}]_{A^{u}_{(\vec{p}, q),\vec{r}}} := [(\vec{w},u)]_{A_{((\vec{p},\infty), q),\vec r}}$.
\end{defn}

Note that if $u\in A_1$, then $A^{u}_{\vec p,q}\subset A_{\vec p,q}$. Indeed, for $\vec w\in A_{\vec p,q}^{u}$ with $u\in A_{1}$,
it holds that
\begin{equation*}
\begin{aligned}
\ave{w}^{\frac{1}{q}}_{Q}&=\big\langle u^{-q}\cdot \big(uw\big)^{q}\big\rangle^{\frac{1}{q}}_{Q}\leq \|u^{-1}\|_{L^{\infty}(Q)}\big\langle \big(uw\big)^{q}\big\rangle^{\frac{1}{q}}_{Q}.
\end{aligned}
\end{equation*}
This yields $[w]_{A_{\vec{p},q}}\leq [w]_{A^{u}_{\vec{p},q}}$. In addition, $A^{u}_{\vec{p},q}=A_{\vec{p},q}$ for $u\equiv 1$.

\vskip 0.3 true cm
Our main result can stated as follows.
\begin{thm}\label{main}
Let $0\leq \beta<\alpha<d$, $1 < p\leq q< \infty$, $\frac \beta d=\frac{1}{p}-\frac{1}{q}$ and $1\leq s<\frac{d}{\alpha-\beta}$.
Then $b\in {\rm BMO}$ if and only if $\|[b,I_{\alpha}](f)\|_{L^{q}(u^{q}w^{q})}\lesssim \|u\|_{M^{\frac{d}{\alpha-\beta}}_{s}(\R^d)}\|f\|_{L^{p}(w^{p})}$ for any $u \in A_{1} \bigcap M^{\frac{d}{\alpha-\beta}}_{s}(\R^d)$, $w\in A_{p,q}^{u}$ and $f\in L^{p}(w^{p})$.
\end{thm}

\begin{rem}\label{HC-1}
When $u \in A_{1}\bigcap M^{\frac{d}{\alpha-\beta}}_{s}(\R^d)$ and $w\in A_{p,q}^{u}$, Proposition \ref{cha:Apq} below implies that
both $w^{-p'}$ and $u^{q}w^{q}$ satisfy the reverse H\"{o}lder inequality.
This leads to that $(w^{-p'}, u^{q}w^{q})$ satisfies \eqref{AB} and \eqref{CD}.
\end{rem}

\begin{rem}\label{HC-2}
In order to establish the boundedness of commutators of fractional integral operators,
it requires to derive two-weight estimates for fractional integral and for maximal operators in general.
To this end, the testing condition introduced in \cite{S1984,S1988} or
the ``$A_p$ bump" condition proposed in \cite{N1983} and \cite{P1994} are usually utilized. In this paper, 
by introducing a new approach that  the fractional integrals can be controlled 
through the bilinear fractional maximal functions, we give a direct proof for the weights in the 
corresponding partial Muckenhoupt class.
\end{rem}

\begin{rem}\label{HC-3}
In terms of \cite[Corollary 2.4]{CR2018}, $b\in {\rm BMO}$ is a necessary condition for the boundedness of $[b,I_{\alpha}]$ from $L^{p}(v^{p})$  to $L^{q}(v^{q})$, where $0<\alpha<d$, $1<p<q<\infty$, $\frac \alpha d=\frac{1}{p}-\frac{1}{q}$ and $v\in A_{p,q}$. The
proof can be outlined as: for any cube $Q$, there exist a cube $Q'$ and some function sequence $(f_{j}, g_{j})_{j=1}^{\infty}$ such that
\begin{equation}\label{rem-necessity-eq1}
|b(x)-b_{Q'}|\chi_{Q}(x)=|Q|^{-\frac{\alpha}{d}}\sum_{j}a_{j}g_{j}(x)[b,I_{\alpha}](f_{j})(x),
\end{equation}
where $\sum_{j}|a_{j}|<\infty$, $\mu(Q')\approx \mu(Q)$ for any $\mu\in A_{\infty}$,
$\|f_j\|_{L^{p}(w^{p})}\lesssim w^{p}|Q|^{\frac{1}{p}}$ with $supp~f_j \subset Q'$,
and $\|g_j\|_{L^{q'}(w^{-q'})}\lesssim w^{-q'}|Q|^{\frac{1}{q'}}$ with~ $supp~g_j \subset Q$.
Therefore, it holds that
\begin{align*}
&\frac{1}{|Q|}\int_{Q}|b(x)-b_{Q}|dx\lesssim \frac{2}{|Q|}\int_{Q}|b(x)-b_{Q'}|dx \\
&\lesssim |Q|^{-\frac{\alpha}{d}}\sum_{j}|a_{j}|\|g_{j}\|_{L^{q'}(w^{-q'})}\|[b,I_{\alpha}](f_{j})\|_{L^{q}(w^{q})}\\
&\lesssim |Q|^{-\frac{\alpha}{d}}\sum_{j}|a_{j}|\|g_{j}\|_{L^{q'}(w^{-q'})}\|f\|_{L^{p}(w^{p})}\lesssim 1.
\end{align*}
Then $b\in {\rm BMO}$.

Note that for $u \in A_{1} \cap M^{\frac{d}{\alpha-\beta}}_{s}(\R^d)$ and $w\in A_{p,q}^{u}$,
applying \eqref{rem-necessity-eq1} to obtain
\begin{align*}
\frac{1}{|Q|}\int_{Q}|b(x)-b_{Q}|dx &\lesssim |Q|^{-\frac{\alpha}{d}}\sum_{j}|a_{j}|\|g_{j}\|_{L^{q'}(w^{-q'})}\|[b,I_{\alpha}](f_{j})\|_{L^{q}(u^{q}w^{q})}\|u^{-1}\|_{L^{\infty}(Q)}\\
&\lesssim |Q|^{-\frac{\alpha}{d}}\sum_{j}|a_{j}|\|g_{j}\|_{L^{q'}(w^{-q'})}\|f\|_{L^{p}(w^{p})}\|u^{-1}\|_{L^{\infty}(Q)}.
\end{align*}
This yields
\begin{equation}\label{rem-necessity-eq2}
|Q|^{\frac{\alpha-\beta}{d}}\Big(\frac{1}{|Q|}\int_{Q}u(x)dx\Big)\Big(\frac{1}{|Q|}\int_{Q}|b(x)-b_{Q}|dx\Big)\lesssim 1.
\end{equation}
Obviously, $b\in {\rm BMO}$ can not be derived from \eqref{rem-necessity-eq2}.
Therefore, we have to look for other approach to show the necessity part in Theorem \ref{main}.
\end{rem}

In order to prove the necessity of Theorem \ref{main}, we will establish the following extrapolation theorem for the  off-diagonal estimate.

\begin{thm}\label{D-OD}
Let $0\leq \beta <\alpha<d$, $1< p_1\leq q_1<\infty$, $\frac{\beta}{d}=\frac{1}{p_1}-\frac{1}{q_1}$
and $1\leq s<\frac{d}{\alpha-\beta}$. If it holds that for any $u\in A_{1}\bigcap M^{\frac{d}{\alpha-\beta}}_{s}(\R^d)$
and $w\in A_{p_1,q_1}^{u}$,
$\|g\|_{L^{q_1}(u^{q_1}w^{q_1})}\lesssim \|u\|_{M^{\frac{d}{\alpha-\beta}}_{s}(\R^d)}\|f\|_{L^{p_1}(w^{p_1})}$ is true when $(g,f)\in\mathcal{F}$,
where $\mathcal F$ is a collection of 2-tuples of non-negative functions,
then we have that for all weight $v\in A_{p_2,q_2}$ with $1<p_2<q_2<\infty$ and $\frac 1{p_2} -\frac 1{q_2} =\frac \alpha d$,
$$\|g\|_{L^{q_2}(v^{q_2})}\lesssim \|f\|_{L^{p_2}(v^{p_2})} \quad \text{when $(g,f)\in\mathcal{F}$.}$$
\end{thm}

It is pointed out that the proof of Theorem \ref{D-OD} is based on the following three new Rubio de Francia extrapolation theorems
for partial Muckenhoupt weights. By the way, the classical Rubio de Francia extrapolation theorem (see \cite{RDF}) implies that if an operator $T$ is
bounded on $L^{p_0}(w)$ for some fixed $p_0$
($1 \leq p_0 < \infty$) and for any $w\in A_{p_0}$, then $T$ is also bounded on $L^p(w)$ for all $1 < p < \infty$ and $w\in A_p$.

\begin{thm}\label{ex-Apqu}
Let $u\in A_{1}$, $1\leq p_{0}<\infty$ and $0<q_{0}, t_{0}<\infty$. Assume that for any $w\in A^{u}_{p_{0},t_{0}}$,
\begin{equation*}
\begin{aligned}
\Big(\int_{\R^d}g(x)^{q_{0}}w(x)^{q_{0}}dx\Big)^{\frac 1{q_{0}}  }
\lesssim \Big(\int_{\R^d}f(x)^{p_{0}}w(x)^{p_{0}}dx\Big)^{\frac 1{p_{0}}}\quad\text{holds when $(g,f)\in\mathcal{F}$,}
\end{aligned}
\end{equation*}
where $\mathcal F$ is a collection of 2-tuples of non-negative functions.
Then for all $1<p<\infty$ and $0<q,t<\infty$ with
$\frac1{p}-\frac1{p_{0}}=\frac1{q}-\frac1{q_{0}}=\frac{1}{t}-\frac{1}{t_{0}},$
and for all $v\in A_{p,t}^{u}$, we have
\begin{equation*}
\begin{aligned}
\Big(\int_{\R^d}g(x)^{q}v(x)^{q}dx\Big)^{\frac 1q }\lesssim \Big(\int_{\R^d}f(x)^{p}v(x)^{p}dx\Big)^{\frac 1p }
\quad \text{when $(g,f)\in\mathcal{F}$.}
\end{aligned}
\end{equation*}
\end{thm}

In addition, we have the following extrapolation theorems for two types of off-diagonal estimates.

\begin{thm}\label{D-OD-2}
Let $0\leq \beta_1 <\beta_2<\alpha<d$, $1\leq p\leq q_{1}<q_2<\frac d{\beta_2-\beta_1}$, $\frac{1}{p}=\frac{\beta_1}{d}+\frac{1}{q_1}=\frac{\beta_2}{d}+\frac{1}{q_2}$, $1\leq s_{1}<\frac{d}{\alpha-\beta_{1}}$, $1\leq s_{2}<\frac{d}{\alpha-\beta_{2}}$ and $\frac1{s_1}>\frac1{s_2}+\frac d{\beta_2-\beta_1}$. Assume that for any $u_{1}\in A_{1}\bigcap M^{\frac{d}{\alpha-\beta_{1}}}_{s_{1}}(\R^d)$ and $w_1\in A_{p_{1},q_{1}}^{u_{1}}$,
$\|g\|_{L^{q_{1}}(u_1^{q_{1}}w_{1}^{q_{1}})}\lesssim \|u_1\|_{M^{\frac{d}{\alpha-\beta_{1}}}_{s_{1}}(\R^d)}\|f\|_{L^{p}(w_1^{p})}$ holds when $(g,f)\in\mathcal{F}$,
where $\mathcal F$ is a collection of 2-tuples of non-negative functions.
Then for all weight $u_2\in A_{1}\bigcap M^{\frac{d}{\alpha-\beta_{2}}}_{s_2}(\R^d)$ and $w_{2}\in A^{u_2}_{p,q_2}$, we have
$\|g\|_{L^{q_2}(u_2^{q_2}w_2^{q_2})}\lesssim \|u_2\|_{M^{\frac{d}{\alpha-\beta_{2}}}_{s_{2}}(\R^d)}\|f\|_{L^{p}(w_{2}^{p})}$ when $(g,f)\in\mathcal{F}$.
\end{thm}
Especially, for the case of $\beta_{2}=\alpha$ in Theorem \ref{D-OD-2}, one has
\begin{thm}\label{D-OD-3}
Let $0\leq \beta<\alpha<d$, $1\leq p\leq q_{1}<q_2<\frac d{\alpha-\beta}$, $\frac{1}{p}=\frac{\beta}{d}+\frac{1}{q_1}=\frac{\alpha}{d}+\frac{1}{q_2}$ and $1\leq s<\frac{d}{\alpha-\beta}$. Assume that for any $u\in A_{1}\bigcap M^{\frac{d}{\alpha-\beta}}_{s}(\R^d)$
and $w\in A_{p,q_{1}}^{u}$,
$\|g\|_{L^{q_{1}}(u^{q_{1}}w^{q_{1}})}\lesssim \|u\|_{M^{\frac{d}{\alpha-\beta}}_{s}(\R^d)}\|f\|_{L^{p}(w^{p})}$ holds when $(g,f)\in\mathcal{F}$.
Then for all weight $v\in A_{p,q_2}$, one has
$\|g\|_{L^{q_2}(v^{q_2})}\lesssim \|f\|_{L^{p}(v^{p})}$ for $(g,f)\in\mathcal{F}$.
\end{thm}

For the sufficiency proof of Theorem \ref{main}, the following multilinear extrapolation theorem is crucial for
estimating the bound of the product functions.

\begin{thm}\label{ex-Appqu}
Let $ \mathcal F$ be a collection of $(m+1)$-tuples of non-negative functions and $u\in A_{1}$. 	Assume that
 there are a vector $\vec{r}=(r_1,\dots,r_{m+1})$ with $1\le r_1,\dots,r_{m+1}\leq \infty$, some exponents $0<q<\infty$
 and $\vec p=(p_1,\dots, p_m)$ with $1\le p_1,\dots, p_m\leq \infty$, and $\vec{r}\preceq (\vec{p}, q)$ such that for any given
 $\vec w=(w_1,\dots, w_m) \in A^u_{(\vec p, q), \vec{r}}$, the inequality
$\|f\|_{L^{q}(w^{q})} \lesssim \prod_{i=1}^m\|f_i\|_{L^{p_i}(w_i^{p_{i}})}$ holds when $(f,f_1,\dots, f_m)\in \mathcal F$,
where $\frac1p:=\frac1{p_1}+\dots+\frac1{p_m}>\max\{0,\frac{1-q}{q}\}$,
$f:=\prod_{i=1}^m f_i$, $w:=\prod_{i=1}^m w_i$, $\vec{r}\preceq (\vec{p}, q)$ means $r_1\le p_1, ..., r_m\le p_m, r_{m+1}'>q$.
Then for all $0<q^*<\infty$, $\vec p^*=(p^*_1, \ldots, p^*_m)$ with $1<p^*_1,\dots, p^*_m\leq \infty$, and $\vec{r}\prec(\vec p^*,q^*)$,
and for all weight
$\vec v=(v_1,\dots, v_m) \in A^u_{(\vec p^*, q^*), \vec{r}}$, we have that  for $(f,f_1,\dots,f_m)\in \mathcal F$,
\begin{equation*}\label{extrapol:C}
\|f\|_{L^{q^*}(v^{q^*})}\lesssim\prod_{i=1}^m \|f_i\|_{L^{p^*_i}(v_i^{p^*_{i}})},
\end{equation*}
where $\frac{1}{p^*}:=\frac1{p^*_1}+\ldots+\frac1{p^*_m}>0$, $\frac{1}{p}-\frac{1}{q}=\frac{1}{p^*}-\frac{1}{q^*}$ and $v:=\prod_{i=1}^m v_i$.	
\end{thm}
\begin{rem}
When $p_i=\infty$ in Theorem \ref{ex-Appqu}, $L^{p_i}(w_i^{p_i})=L^{\infty}(w_i)$ holds
for a given locally integrable and almost everywhere positive function $w_i$.
\end{rem}
As a direct corollary of Theorem \ref{ex-Appqu}, by taking $u\equiv 1$, one has the
following multivariable Rubio de Francia extrapolation theorem.

\begin{cor}\label{ex-Appq}
Under the conditions of Theorem \ref{ex-Appqu} with $u\equiv 1$ and the same restrictions on $q^*$, $\vec p^*=(p^*_1, \ldots, p^*_m)$,
$\vec{r}=(r_1,\dots,r_{m+1})$ with $\vec{r}\prec(\vec p^*,q^*)$ and $\vec{s}=(s_1,\dots, s_m)$ with $\vec{r}\prec(\vec{s},q^*)$,
then  for $\vec v=(v_1,\dots, v_m) \in A_{(\vec p^*, q^*), \vec{r}}$,
the following inequalities hold
\begin{equation*}\label{extrapol:C}
\|f\|_{L^{q^*}(v^{q^*})} \lesssim \prod_{i=1}^m \|f_i\|_{L^{p^*_i}(v_i^{p^*_{i}})}\quad\text{when $(f,f_1,\dots,f_m)\in \mathcal F$}.
\end{equation*}
\end{cor}

\begin{rem}
When \( m = 1 \) and $r=1$, Theorem \ref{ex-Appqu} is a special case of Theorem \ref{ex-Apqu} with \( q = t \).
In addition, motivated by \cite{D2011} and \cite{LMPT2010}, there is an interesting problem to study the
sharp bound of $\big(L^{p_{1}}(w_{1}^{p_{1}})\times\cdots\times L^{p_{m}}(w_{m}^{p_{m}}), L^{q}(w^{q})\big)$
for the operator in $A_{\vec p,t}$ (or $A_{\vec p,t}^u$).
\end{rem}

\begin{rem}
Under Theorem \ref{ex-Apqu}, the proof of Theorem \ref{ex-Appqu} is only a minor modification of the corresponding
proof in \cite[Theorem 1.1]{LMO2020}
(or \cite[Theorem 1.3]{ALM2023}). However, for the completeness and convenience of readers, we still provide a detailed proof
in the paper.
\end{rem}

\subsection{Sketch on the proofs of Theorems \ref{ex-Apqu}-\ref{ex-Appqu}}\label{1.3-CC}

In order to prove Theorem \ref{ex-Apqu}, it requires to establish
some basic factorization results related to partial multiple weights, which
are presented in Proposition \ref{FAC-1}. Meanwhile,  a modified maximal function (see \eqref{def-Mur}) is introduced  and
a suitable $A^{u}_{1,q}$ weight is constructed (see Lemma \ref{ex-lem2}) by us.

\vspace{0.2cm}

To prove Theorem \ref{D-OD-2} (or Theorem \ref{D-OD-3}), we first utilize the embedding
theorem and duality method in Lorentz spaces to establish
\begin{equation}\label{WD}
\|g\|^{q_1}_{L^{q_2}(u_2^{q_2}w_2^{q_2})}=\sup_{\|U^{q_1}\|_{L^{(\frac{q_2}{q_1})',r'}(\R^d)}\neq 0} \Big|\int_{\R^d}g(x)^{q_1}u_2^{q_1}w_2(x)^{q_1}U(x)^{q_1} dx\Big|\|U^{q_1}\|^{-1}_{L^{(\frac{q_2}{q_1})',r'}(\R^d)},
\end{equation}
where $U^{q_1}\in L^{(\frac{q_2}{q_1})',r'}(\R^d)\subset L^{(\frac{q_2}{q_1})',\infty}(\R^d)$ and further $U\in L^{\frac{d}{\beta_2-\beta_1},\infty}(\R^d)$.
On the other hand, for any $w_2\in A^{u_2}_{p,q_2}$ and $U\in L^{\frac{d}{\beta_2-\beta_1},\infty}(\R^d)$,
it follows from the crucial result in Lemma \ref{D-OD-2-lem} that there exists a function $u_1\in A_{1}\bigcap M^{\frac{d}{\alpha-\beta_1}}_{s_{1}}(\R^{d})$ such that $w_2\in A^{u_1}_{p,q_1}$. Consequently, the boundedness of $(g,f)\in \mathcal{F}$ in $(L^{p}(w_1^p), L^{q_1}(u_1^{q_1}w_1^{q_1}))$ for $w_1\in A^{u_1}_{p,q_1}$ implies $g \in L^{q_2}(u_2^{q_2}w_2^{q_2})$ by \eqref{WD},
and then the proof of Theorem \ref{D-OD-2} is completed.

\vspace{0.2cm}

The proof of Theorem \ref{ex-Appqu} comes from the following two steps:
at first, a restricted version of  Theorem \ref{ex-Appqu}
is derived when all the related exponents except one are fixed.
In this process, we establish a new characterization of the weighted class $A^u_{\vec p, q, \vec r}$
(see  Lemma \ref{ex-lem-main}), which allows us to show the
multivariable extrapolation result from a one-variable extrapolation case.
Subsequently, by iterating the result in the first step,
the corresponding  conclusion on the case of $\vec p$ can
be improved into the more generic case of $\vec p^*$.

\subsection{Sketch on the proof of Theorem \ref{D-OD}}\label{1.3-C}

The proof of Theorem \ref{D-OD} can be taken as follows

{\bf $\bullet$} By applying Theorem \ref{ex-Apqu} with parameters $q = t$ and $q_0 = t_0$,
one can reduce the general case to the special situation of $p_{1}=p_{2}=:p$ in Theorem \ref{D-OD}.

{\bf $\bullet$} When $\alpha\leq \frac{d+\beta}{2}$, this implies $q_2<\frac{d}{\alpha-\beta}$ in Theorem \ref{D-OD-3}.
Then it follows from Theorem \ref{D-OD-3} that Theorem \ref{D-OD} holds directly.

{\bf $\bullet$} When $\alpha>\frac{d+\beta}{2}$, one can select an increasing sequence $\beta_k=\beta+(\frac{3}{2}\alpha-\frac d2-\beta)(1-\frac 1{2^{k-1}})$ ($k\in\Bbb N$)
to fulfill $\alpha\leq \frac{d+\beta_{k_0}}{2}$ for some  $\beta_{k_0}$.
Obviously, we have
\begin{equation}\label{WY-1}
0\leq \beta_k <\beta_{k+1}<\alpha<d, \lim_{k\rightarrow \infty}\beta_k=\frac32\alpha-\frac d2.
\end{equation}
In addition, a direct calculation yields
$2\beta_{k+1}-\beta_k=\frac{3}{2}\alpha-\frac d2<d$.
Together with $\frac{1}{p}=\frac{\beta_k}{d}+\frac{1}{q_k}=\frac{\beta_{k+1}}{d}+\frac{1}{q_{k+1}}$, one has
\begin{equation}\label{WY-2}
1< p\leq q_{k}<q_{k+1}<\frac d{\beta_{k+1}-\beta_k}.
\end{equation}
Let $s_{k}$ be defined as
$\frac{1}{s_k}=\frac{1}{s}+\frac{\beta_k-\beta}{d}+\big(1-\frac{1}{2^{k-1}}\big)\big(\frac{\alpha-\beta}{d}-\frac{1}{s}\big)$.
Then
\begin{equation}\label{WY-3}
1\leq s_{k}<\frac{d}{\alpha-\beta_k}, 1\leq s_{k+1}<\frac{d}{\alpha-\beta_{k+1}}, \frac1{s_k}>\frac1{s_{k+1}}+\frac {\beta_{k+1}-\beta_k}d.
\end{equation}
In terms of \eqref{WY-1}, \eqref{WY-2} and \eqref{WY-3}, Theorem \ref{D-OD-2} can be applied to prove
Theorem \ref{D-OD} when $\beta_1, \beta_2, s_1$ and $s_2$ are replaced by $\beta_k, \beta_{k+1}, s_k$ and $s_{k+1}$, respectively.
On the other hand, due to $\frac32\alpha-\frac d2>2\alpha-d$ and $\alpha=\frac{d+2\alpha-d}{2}< \frac{d + \beta_{k_0}}{2}$ for some $\beta_{k_0}$,
then Theorem \ref{D-OD} follows from Theorems \ref{D-OD-2} and \ref{D-OD-3}.

\subsection{Partial multiple weighted estimates for fractional operators}\label{1.4-C}

It is well known that the fractional integrals have a large number of applications in partial differential equations.
An exposition about the properties of operator $I_{\alpha}$ can be found in the books \cite{Stein1970} and
\cite{G2004}. The weighted inequalities of fractional integrals have also been extensively studied
(see \cite{MW1974}, \cite{CF1974}, \cite{FM1974}, \cite{P1990}, \cite{ST1991}, \cite{S1984,S1988},
\cite{N1983} and \cite{P1994}). In particular, such a weighted inequality with suitable weight function $U$
can be derived (also see \cite{FP1983})
$$\|fU\|_{L^{p}(\R^d)}\leq C\|\nabla f\|_{L^{p}(\R^d)}.$$

Based on Subsections \ref{1.3-CC}-\ref{1.3-C}, we can establish some weighted estimates for the fractional integrals 
in relation to partial multiple weights.
\begin{thm}\label{Lp-Lq-Ia}
Let $0 \leq \beta < \alpha < d$, $1<r<s<\infty$ and $1< p \leq q < \infty$ such that $\frac{1}{p} - \frac{1}{q} = \frac{\beta}{d}$.
For $U \in M^{\frac{d}{\alpha-\beta}}_{s}(\R^d)$, $u=M(U^{r})^{1/r}$ with $M(U^{r})$ being the Hardy-Littlewood maximal function
of $U^{r}$ and $w\in A_{p,q}^{u}$, it holds that
$\|I_{\alpha}(f)U\|_{L^{q}(w^{q})}\lesssim \|U\|_{M^{\frac{d}{\alpha-\beta}}_{s}(\R^d)}\|f\|_{L^{p}(w^p)}.$
\end{thm}

\begin{rem}
By the definition of $u$, it is easy to know $u\in A_{1}\bigcap M^{\frac{d}{\alpha-\beta}}_s(\R^d)$. Therefore, for $0 \leq \beta < \alpha < d$, $1<p\leq q<\infty$, $\frac{\beta}{d}=\frac 1p-\frac 1q$, $u\in A_{1}\bigcap M^{\frac{d}{\alpha-\beta}}_s(\R^d)$ and $w\in A_{p,q}^u$,
it  follows from Theorem \ref{Lp-Lq-Ia} that
$\|I_{\alpha}(f)\|_{L^{q}(u^{q}w^{q})}\lesssim \|u\|_{M^{\frac{d}{\alpha-\beta}}_{s}(\R^d)}\|f\|_{L^{p}(w^p)}.$
\end{rem}

\subsection{Sketch on the proof of Theorem \ref{main}}\label{1.2-C}

At first, we now illustrate the proof on the sufficient condition in Theorem \ref{main}
(namely, for $b\in {\rm BMO}$, $[b,I_{\alpha}]$ is bounded from $L^{p}(w^{p})$ to $L^{q}(u^{q}w^{q})$ for any $w\in A_{p,q}^{u}$).
For $0\leq \alpha< md$,  the $m$-linear fractional maximal function is defined by
$$\mathcal{M}_{\alpha,m}(\vec{f})(x)=\sup_{Q\ni x}|Q|^{\frac{\alpha}{d}}\prod_{i=1}^{m}\ave{f_{i}}_{Q}.$$
Denote by $\mathcal{M}_{0,m}=\mathcal{M}_{m}$ and $\mathcal{M}_{\alpha,1}=M_{\alpha}$.
Note that for $u\in A_{1}\bigcap L^{\frac{d}{\alpha-\beta},1}(\R^d)$ with $0\leq \beta< \alpha$,
\begin{equation}\label{Ma-Mb}
\begin{aligned}
M_{\alpha}(f)(x)\leq \sup_{Q\ni x}|Q|^{\frac{\beta}{d}}\ave{f}_{Q}\ave{u^{-1}}_{Q}\cdot |Q|^{\frac{\alpha-\beta}{d}}\ave{u}_{Q}\leq \|u\|_{L^{\frac{d}{\alpha-\beta},1}(\R^d)}\mathcal{M}_{\beta,2}(f,u^{-1})(x).
\end{aligned}
\end{equation}
Due to \eqref{Ma-Mb}, in order to show the sufficiency in Theorem \ref{main},
it suffices to estimate the bound of the mapping $\big(L^{p}(w^{p})\times L^{\infty}(u), L^{q}(u^{q}w^{q})\big)$ for
$\mathcal{M}_{\beta,2}(f,u^{-1})$, where $L^{\infty}(u)=\{g: gu\in L^{\infty}(\R^d)\}$.

Note that when $1< p_{1}, p_{2}< \infty$ with $\frac{\beta}{d}=\frac{1}{p_{1}}+\frac{1}{p_{2}}-\frac{1}{q}$, the
following inequality
\begin{equation}\label{Mb2}
\begin{aligned}
\|\mathcal{M}_{\beta,2}(f_{1},f_{2})\|_{L^{q}(w_{1}^{q}w_{2}^{q})}\lesssim \|f_{1}\|_{L^{p_{1}}(w_{1}^{p_{1}})}\|f_{2}\|_{L^{p_{2}}(w_{1}^{p_{2}})}
\end{aligned}
\end{equation}
has been proved by \cite{LOPTT2009} for $\beta=0$ and \cite{M2009} for $\beta>0$.
For $p_{2}=\infty$, we will apply the multilinear extrapolation theorem (Corollary \ref{ex-Appq})
to show \eqref{Mb2}.

Combining \eqref{Ma-Mb} and \eqref{Mb2} yields the partial weighted estimate for the fractional operators in Theorem \ref{ex-Appqu}.
Then by the arguments similar to those in \cite{CF1974} together with Theorem \ref{Lp-Lq-Ia}, we can derive the
corresponding weighted estimate for $[b, I_{\alpha}]$
in Theorem \ref{main}, where $b\in {\rm BMO}$.

Secondly, we explain how to prove the necessity in Theorem \ref{main} (namely,  if $\|[b,I_{\alpha}](f)\|_{L^{q}(u^{q}w^{q})}\lesssim \|u\|_{M^{\frac{d}{\alpha-\beta}}_s(\R^d)}\|f\|_{L^{p}(w^{p})}$ for any $u\in A_1\cap M^{\frac{d}{\alpha-\beta}}_s(\R^d)$, $w\in A_{p,q}^{u}$ and $f\in L^{p}(w^{p})$, then $b\in {\rm BMO}$ holds).
Note that  the extrapolation Theorem \ref{D-OD} together with \cite[Corollary 2.4]{CR2018}
will derive the  necessity  estimate in Theorem \ref{main} directly.

The structure of the paper is as follows. In Section \ref{pro}, we will outline the basic properties of partial multiple weights,
which will play essential roles in the subsequent proofs on the weighted estimates and extrapolation theorems.
The proofs of our main results on extrapolation (including Theorems \ref{D-OD}-
\ref{ex-Appqu}) are presented
in Section \ref{extra}. Section \ref{estimates} contains the proof of Theorem \ref{Lp-Lq-Ia},
meanwhile the weighted estimates for fractional integrals in Lebesgue spaces and their applications are given.
In Section \ref{Thm1.1}, based on the results in Sections \ref{extra}-\ref{estimates},
the characterization of the commutators is established and then Theorem \ref{main} is shown.

\vskip 0.2 true cm

Throughout this paper, the following notations are used.

\vskip 0.1 true cm

{\bf $\bullet$} $A\lesssim B$ means $A\le CB$ with some generic positive constant $C$
and  $A\approx B$ means $A\lesssim B\lesssim A$.

{\bf $\bullet$} \text{$w(E):=\int_{E}w(x)dx$ and $\ave{f}_{E,w}:=\frac{1}{w(E)}\int_{E}f(x)dw=\frac{1}{w(E)}\int_{E}f(x)w(x)dx$}
for the locally integrable and almost everywhere positive function $w$.

\section{Properties of partial multiple weights}\label{pro}
Motivated by the properties of classical Muckenhoupt weights, we
start to investigate the related properties for the partial multiple weights.

\subsection{Characterization of the partial multiple weights}

Note that the author in \cite{M2009} established the following properties of the weights in the class $A_{\vec p, q}$,
which is a variant version of Theorem 3.6 in \cite{LOPTT2009}.

\begin{lem}(\cite[Theorem 3.4]{M2009})
If $1<p_{1},\ldots,p_{m}<\infty$ and $\vec w\in A_{\vec p,q}$, then
$\big(\prod_{i=1}^{m}w_{i}\big)^{q}\in A_{mq}$ and $w_{i}^{-p_{i}}\in A_{mp'_{i}}$.
\end{lem}

Next, we establish a more generalized result as follows.
\begin{prop}\label{cha:Apq}
Assume $0<q\leq \infty$, $\vec p=(p_1, \ldots,  p_m )$ with $1 \le p_1, \ldots, p_m  \le \infty$, $\vec{w}=(w_1, \ldots, w_m )$
and $w = \prod_{i=1}^m w_i$. Then
$\big[w_i^{-p_i'}\big]_{A_{mp_i'+p_i'/q-p_i'/p }} \le [\vec{w}]_{A_{\vec p, q}}^{p_i'}$ ($i = 1, \ldots, m$)
and $[w^q]_{A_{mq+1-q/p}} \le  [\vec{w}]_{A_{\vec p, q}}^{q}$.

When $p_i = 1$, the corresponding estimate above is interpreted as $[w_i^{\frac{1}{m+\frac 1q -\frac 1p }}]_{A_1} \le [\vec{w}]_{A_{\vec p, q}}^{\frac{1}{m+\frac 1q -\frac 1p }}$; when $q=\infty$,
one has $[w^{-\frac{1}{m-\frac 1p }}]_{A_1} \le [\vec{w}]_{A_{\vec p, q}}^{\frac{1}{m-\frac 1p }}$.

Conversely, it holds that
$[\vec{w}]_{A_{\vec p, q}} \le [w^q]_{A_{mq+1-q/p}}^{\frac{1}{q}} \prod_{i=1}^m \Big[w_i^{-p_i'}\Big]_{A_{mp_i'
+p_i'/q-p_i'/p}}^{\frac{1}{p_i'}}.
$
\end{prop}

\begin{proof}
For any fixed $j \in \{1, \ldots, m\}$, we next show
$$[w_j^{-p_j'}]_{A_{mp_j'+p_j'/q-p_j'/p }} \le [\vec{w}]_{A_{\vec p, q}}^{p_j'}.$$
Note that
\begin{equation}\label{cha:Apq:eq1}
\frac{1}{p} + \sum_{i \ne j} \frac{1}{p_i'} = m - 1 + \frac{1}{p_j}.
\end{equation}
Define $s_j$ as
\begin{equation}\label{cha:Apq:eq1-sj}
\frac{1}{s_j} = \frac{1}{m - 1 + \frac{1}{p_j}+\frac 1q-\frac 1p}\cdot \frac 1q,
\end{equation}
and for $i \ne j$, set
\begin{equation}\label{cha:Apq:eq1-si}
\frac{1}{s_i} = \frac{1}{m - 1 + \frac{1}{p_j}+\frac 1q-\frac 1p} \cdot\frac{1}{p_i'}.
\end{equation}
Collecting \eqref{cha:Apq:eq1}, \eqref{cha:Apq:eq1-sj} and \eqref{cha:Apq:eq1-si} yields $\sum_i \frac{1}{s_i} = 1$.
Due to
\begin{equation}\label{cha:Apq:eq2}
\big[w_j^{-p_j'}\big]_{A_{mp_j'+p_j'/q-p_j'/p}} = \sup_Q \, \big\langle w_j^{-p_j'}\big\rangle_Q     \big\langle w_j^{\frac{q}{s_j}}\big\rangle_Q^{\frac{s_j}{q}},
\end{equation}
then it follows from
$\frac{p_j'}{mp_j'+\frac{p_j'}{q}- \frac{p_j'}{p}-1}= \frac{1}{m - 1 + \frac{1}{p_j}+\frac 1q- \frac 1p}$
and the H\"older's inequality that
$$
\big\langle w_j^{-p_j'}\big\rangle_Q^{\frac 1{p_j'}}   \big\langle w_j^{\frac{q}{s_j}}\big\rangle_Q^{\frac{s_j}{q}} = \big\langle w_j^{-p_j'}\big\rangle_Q^{\frac 1{p_j'}}
\big\langle w^{\frac{q}{s_j}} \prod_{i \ne j } w_i^{-\frac{q}{s_j}}\big\rangle_Q^{\frac{s_j}{q}} \le \ave{w^q}_Q^{\frac{1}{q}}\prod_{i=1}^{m} \ave{ w_i^{-p_i'} }_Q^{\frac{1}{p_i'} } \le [\vec{w}]_{A_{\vec p,q}}.
$$
When $p_j=1$, by
$\esssup_Q w_j^{-1} \big\langle w_j^{\frac{1}{m+\frac 1q -\frac 1p }}\big\rangle_Q^{m+\frac 1q -\frac 1p }\le [\vec{w}]_{A_{\vec p,q}},$
one has
$\big[w_j^{\frac {1}{m+\frac 1q -\frac 1p }}\big]_{A_1}^{m+\frac 1q -\frac 1p } \le [\vec{w}]_{A_{\vec p, q}}.$

We next estimate $[w^q]_{A_{mq+1-q/p}}$. Note that
\begin{equation}\label{cha:Apq:eq4}
[w^q]_{A_{mq+1-q/p}} = \sup_Q \, \ave{w^q}_Q \big\langle w^{-\frac{1}{m-\frac 1p }}\big\rangle_Q^{mq-\frac qp}.
\end{equation}
Define $s_i$ by $\frac{1}{m-\frac 1p } \cdot s_i=p_i'$.
By $\sum_i \frac{1}{s_i} = 1$,
it follows from H\"older's inequality that
\begin{align*}
\ave{w^q}_Q \big\langle w^{-\frac{1}{m-\frac 1p }}\big\rangle_Q^{mq-\frac qp} \le \ave{w^q}_Q \prod_{i=1}^m \big\langle  w_i^{-p_i'}\big\rangle_Q^{(\frac{1}{m-\frac 1p })\frac{1}{p_i'}(mq-\frac qp)}
= \big[ \ave{w^q}_Q^{\frac{1}{q}} \prod_{i=1}^{m} \ave{  w_i^{-p_i'}}_Q^{\frac{1}{p_i'}} \big]^{q} \le [\vec{w}]_{A_{\vec p, q}}^{q}.
\end{align*}
In the case of $q=\infty$,  one has
\begin{multline*}
\big[w^{-\frac{1}{m-\frac 1p }}\big]_{A_1}^{m-\frac 1p } = \sup_Q \big\langle w^{-\frac{1}{m-\frac 1p }} \big \rangle_Q^{m-\frac 1p } \esssup_Q w
\le \sup_Q\big[ \prod_{i=1}^{m} \langle w_i^{-p_i'} \rangle^{\frac 1p _i'}_Q \big]\esssup_Q w = [\vec{w}]_{A_{\vec p, q}}.
\end{multline*}
To establish the boundedness of $ [\vec{w}]_{A_{\vec p, q}}$, we first show
\begin{equation}\label{cha:Apq:eq3}
1 \le \big\langle w^{-\frac{1}{m-\frac 1p }} \big\rangle_Q^{m-\frac{1}{p}} \prod_{i=1}^{m} \big\langle w_i^{\frac{1}{m - 1 + \frac{1}{p_i}+\frac 1q -\frac 1p}} \big\rangle_Q^{m-1+\frac{1}{p_i}+\frac 1q -\frac 1p}.
\end{equation}
In fact, for $\alpha_{1},\alpha_{2}>0$,  one has that by H\"{o}lder's inequality
$$1\leq \ave{w^{-\alpha_{1}}}^{\frac1{\alpha_1}}\ave{w^{\alpha_2}}^{\frac1{\alpha_2}}.$$
Setting  $\alpha_1 = \frac{1}{m-\frac 1p }$ and $\alpha_2 = \frac{1}{n(m-1)+\frac{m}{q}-\frac{m-1}{p}}$ yields
$$
1 \le \ave{ w^{-\frac{1}{m-\frac 1p }} }_Q^{m-\frac{1}{p}} \big\langle w^{ \frac{1}{m(m-1)+\frac{m}{q}-\frac{m-1}{p}} }\big\rangle_Q^{m(m-1)+\frac{m}{q}-\frac{m-1}{p}}.
$$
In addition, define $r_i$ from $\frac{1}{m(m-1)+\frac{m}{q}-\frac{m-1}{p}}\cdot r_i = \frac{1}{m - 1 + \frac{1}{p_i}+\frac 1q-\frac 1p}$.
By H\"older's inequality, we arrive at
$$
\big\langle w^{ \frac{1}{m(m-1)+\frac{m}{q}-\frac{m-1}{p}} }\big\rangle_Q^{m(m-1)+\frac{m}{q}-\frac{m-1}{p}} \le \prod_{i=1}^m \big\langle w_i^{\frac{1}{m - 1 + \frac{1}{p_i}+\frac 1q-\frac 1p} }\big\rangle_Q^{ m - 1 + \frac{1}{p_i}+\frac 1q-\frac 1p},
$$
and then \eqref{cha:Apq:eq3} is proved.

It follows from \eqref{cha:Apq:eq3} that
\begin{align*}
[\vec{w}]_{A_{\vec p, q}}&\le \sup_Q \Big[ \ave{w^q}_Q  \big\langle w^{-\frac{1}{m-\frac 1p}} \big\rangle_Q^{m-\frac 1p} \big]^{\frac 1 q}
 \prod_{i=1}^m \big[ \ave{w_i^{-p_i'}}_Q \big \langle  w_i^{\frac{p_i'}{mp_i'+p_i'/q-p_i'/p-1}}\big \rangle_Q^{mp_i'+p_i'/q-p_i'/p- 1} \big]^{\frac 1{p_i'}}\\
& \le [w^q]_{A_{mq+1-q/p}}^{\frac{1}{q}} \prod_{i=1}^m \big[w_i^{-p_i'}\big]_{A_{mp_j'+p_j'/q-p_j'/p}}^{\frac{1}{p_i'}},
\end{align*}
where \eqref{cha:Apq:eq2} and \eqref{cha:Apq:eq4} are used.
\end{proof}

The following conclusion for the duality of multilinear weights can be easily obtained.
\begin{prop}\label{dua:Apq}
Set $\vec p=(p_1, \ldots,  p_m )$ with $1 < p_1, \ldots, p_m  < \infty$ and $1 \leq q\leq \infty$.
Let $\vec{w}=(w_1, \ldots, w_m )\in A_{\vec p, q}$ with $w = \prod_{i=1}^m w_i$ and define
$\vec w^{\, i}= (w_1, \ldots, w_{i-1}, w^{-1}, w_{i+1}, \ldots, w_m )$,
$\vec p^{\,i}= (p_1, \ldots, p_{i-1}, q', p_{i+1}, \ldots, p_m )$.
Then $[\vec{w}^{\,i}]_{A^{u}_{\vec p^{\, i}, p_i'}} =
[\vec{w}]_{A_{\vec p, q}}.$
\end{prop}

\begin{proof}
Without loss of generality and for convenience,  $i=1$ is taken.
Due to $w^{-1} \prod_{i=2}^m w_i = w_1^{-1}$, then $[\vec{w}^{\,i}]_{A_{\vec p^{\, i},p_i'}} = \ave{w_1^{-p_1'}}_Q^{\frac{1}{p_1'}} \ave{w^{q}}_Q^{\frac{1}{q}} \prod_{i=2}^m \ave{ w_i^{-p_i'}}_Q^{\frac{1}{p_i'}} = [\vec{w}]_{A_{\vec p, q}}.$
Therefore,  Proposition \ref{dua:Apq} is proved.
\end{proof}

\medskip

According to the definition of $A^{u}_{\vec{p}, q}$ and the above propositions, we have directly
\begin{cor}\label{cor-Apqu}
Let $0<q\leq \infty$, $\vec p=(p_1, \ldots,  p_m )$ with $1 \le p_1, \ldots, p_m  \le \infty$, $\vec{w}=(w_1, \ldots, w_m )$ 
and $w = \prod_{i=1}^m w_i$.
Then
\begin{align*}
&[u^{-1}]_{A_{m+1+ 1/q -1/p  }} \le [\vec{w}]_{A^{u}_{\vec p, q}},\\
&[w_i^{-p_i'}]_{A_{(m+1 )p_i'+p_i'/q-p_i'/p }} \le [\vec{w}]_{A^{u}_{\vec p, q}}^{p_i'}, \qquad i = 1, \ldots, m,\\
&[u^{q}w^q]_{A_{(m+1 )q+1-q/p}} \le  [\vec{w}]_{A^{u}_{\vec p, q}}^{q}.
\end{align*}
When $p_i = 1$,  the corresponding estimate above is interpreted as $[w_i^{\frac{1}{m+1 +\frac 1q -\frac 1p }}]_{A_1} \le [\vec{w}]_{A_{\vec p, q}}^{\frac{1}{m+1 +\frac 1q -\frac 1p }}$; when $q=\infty$,
$[w^{-\frac{1}{m+1 -\frac 1p }}]_{A_1} \le [\vec{w}]_{A^u_{\vec p, q}}^{\frac{1}{m+1 -\frac 1p }}$ holds.

Conversely,
$[\vec{w}]_{A^{u}_{\vec p, q}} \le [u^{q}w^q]_{A_{(m+1 )q+1-q/p}}^{\frac{1}{q}} [u^{-1}]_{A_{m+1+1/q -1/p  }}\prod_{i=1}^m [w_i^{-p_i'}]_{A_{(m+1 )p_i'+p_i'/q-p_i'/p}}^{\frac{1}{p_i'}}$.
\end{cor}

\begin{cor}
Assume $\vec p=(p_1, \ldots,  p_m )$ with $1 < p_1, \ldots, p_m  < \infty$ and $1 \leq q\leq \infty$.
Let $\vec{w}=(w_1, \ldots, w_m )\in A^{u}_{\vec p, q}$ with $w = \prod_{i=1}^m w_i$ and define
$\vec w^{\, i}= (w_1, \ldots, w_{i-1}, u^{-1}w^{-1}, w_{i+1}, \ldots, w_m )$,
$\vec p^{\,i}= (p_1, \ldots, p_{i-1}, q', p_{i+1}, \ldots, p_m )$.
Then $[\vec{w}^{\,i}]_{A^{u}_{\vec p^{\, i}, p_i'}} =
[\vec{w}]_{A^{u}_{\vec p, q}}.$
\end{cor}

\subsection{Power weights in partial Muckenhoupt weights}
If there exist two constants $r>1$ and $C>0$ such that for any cube $Q\subset \R^d$,
\begin{equation}\label{RHr}
\ave{w^{r}}_{Q}^{\frac{1}{r}}\leq C\ave{w}_{Q},
\end{equation}
then $ w $ is said to fulfill the reverse H\"{o}lder condition of order $r$
and is written by $ w \in RH_{r}$.
It follows from H\"{o}lder inequality that $ w \in RH_{r}$ implies $ w \in RH_{s}$ for $s<r$.
In addition, if $w\in RH_{r}$ with $r>1$, then $ w \in RH_{r+\epsilon}$ for some $\epsilon>0$.
Denote by
$r_{ w }=\sup\{r>1: w \in RH_{r}\}$, which is called the {\it critical index} of $ w $ for the reverse H\"{o}lder condition.

We next investigate that for which scope of real number $b$,  $|x|^b\in A_{p, q}^{u}$ holds when $u\in A_{1}$.

\begin{prop}\label{power1}
Assume $u\in A_{1}$, $1<p<\infty$ and $0<q<r_{u}$. Then $|x|^{b} \in A_{p,q}^{u}$ if $-\frac dq+\frac d{r_{u}}<b<\frac d{p'}.$
\end{prop}
\begin{proof}
It suffices to show  that for any cube $Q=Q(x_{Q},r_{Q})$,
\begin{equation}\label{power1-1}
\ave{|\cdot|^{bq}u^{q}}_{Q}^{\frac 1q }\ave{|\cdot|^{-bp'}}^{\frac 1{p'}}_{Q}\|u^{-1}\|_{L^{\infty}(Q)}\lesssim 1.
\end{equation}
For $0\leq b<\frac d{p'}$, it is easy to know $|x|^{-bp'}\in A_{1}\subset A_{\tilde{p}}$ with $\tilde{p}=\frac{p'}{qs'}+1$
and $1<s<\frac {r_{u}}q$. This yields
\begin{equation}\label{power1-2}
\ave{|\cdot|^{-bp'}}^{\frac 1{p'}}_{Q} \ave{|\cdot|^{bqs'}}_{Q}^{\frac1{qs'}}\lesssim 1.
\end{equation}
In addition, it follows from H\"{o}lder inequality and \eqref{RHr} that
\begin{equation}\label{power1-3}
\begin{aligned}
\ave{|\cdot|^{bq}u^{q}}_{Q}^{\frac 1q }&\leq \ave{u^{qs}}_{Q}^{\frac1{qs}}  \ave{|\cdot|^{bqs'}}^{\frac1{qs'}}_{Q}\lesssim \ave{u}_{Q}\ave{|\cdot|^{bqs'}}^{\frac1{qs'}}_{Q}.
\end{aligned}
\end{equation}
Collecting \eqref{power1-2} and \eqref{power1-3} derives \eqref{power1-1} for $0\leq b<\frac d{p'}$.

For $-\frac dq + \frac d{r_{u}}<b< 0$, we have $0<\frac{dq}{r_{u}}-bq<d$ and
\begin{equation*}
\begin{aligned}
\ave{|\cdot|^{bq}u^{q}}_{Q}&\leq \langle |\cdot|^{bq\cdot \frac{b-\frac d{r_{u}}}{b}}\rangle^{\frac{b}{b-\frac d{r_{u}} }}
\langle u^{q\cdot \frac{d-br_{u}}{d}}\rangle^{\frac{d}{d-br_{u}}}
\lesssim \ave{u}_{Q}^{q}\langle |\cdot|^{bq- \frac{dq}{r_{u}} } \rangle^{\frac{b}{b-\frac d{r_{u}}  }}
\lesssim \ave{u}_{Q}^{q}\ave{|\cdot|^{bq}}_{Q}.
\end{aligned}
\end{equation*}
Then \eqref{power1-1} is shown by $|x|^{bq}\in A_{1}$.
\end{proof}

In addition, the proof on Proposition \ref{power1} can be suitably modified to obtain the following result for
$A_{p, q}^{u}$ with $p=1$.
\begin{prop}\label{power2}
Let $u\in A_{1}$. For $-\frac d{q}+\frac {d}{r_{u}}<b\leq 0$, $|x|^{b}\in  A_{1, q}^{u}$ holds.
\end{prop}

We point out that the conditions $q<r_{u}$ and $-\frac dq+\frac d{r_{u}}<b<\frac d{p'}$
in Proposition \ref{power1} are optimal, which can be illustrated as follows.
\begin{prop}\label{power3}
Let $1<p<\infty$ and $u=|x-x_{A}|^{-a}$ with $0<a<d$ and $x_{A}\in \mathbb{R}^{d}\backslash \{0\}$. Then the conditions $q<r_{u}$ and $-\frac dq+\frac d{r_{u}}<b<\frac d{p'}$ are necessary for $|x|^{b}\in  A_{p,q}^{u}$.
\end{prop}
\begin{proof}
If $q\geq r_{u}=\frac da$ and $Q=Q(x_{A},\frac{|x_{A}|}{2})$, then
$\ave{|\cdot-x_{A}|^{-aq}|\cdot|^{bq}}_{Q}\approx \big(\frac{|x_{A}|}{2}\big)^{bq}\ave{|\cdot-x_{A}|^{-aq}}_{Q}=\infty.$
On the other hand, it is obvious that \eqref{power1-1} does not hold when $x_{Q}=0$ and $b\geq \frac d{p'}$.

We now use the contradiction method to show that the condition $-\frac dq+\frac d{r_{u}}<b$ is necessary
when $|x|^{b}\in  A_{p,q}^{u}$. If $b\leq-\frac dq+  \frac d{r_{u}}=-\frac dq+a<0$,  then
\begin{equation}\label{power1-4}
\begin{aligned}
&\ave{|\cdot|^{-bp'}}_{Q}\thickapprox r_{Q}^{-bp'} \qquad \text{for any } Q=Q(0,r_{Q}).
\end{aligned}
\end{equation}
Note that for $r_{Q}\geq 5|x_{A}|$ and $Q=Q(0,r_{Q})$, one has
\begin{equation}\label{power1-5}
\begin{aligned}
&\ave{u}_{Q}=\frac{1}{|Q|}\int_{Q}|x-x_{A}|^{-a}dx\lesssim \frac{1}{|Q|}\int_{B(x_{A},\sqrt{d}r_{Q})}|x-x_{A}|^{-a}dx\lesssim r_{Q}^{-a}.
\end{aligned}
\end{equation}
Then for $b<-\frac dq+a$,
\begin{equation}\label{power1-6}
\begin{aligned}
&\int_{Q(0,r_{Q})}|x-x_{A}|^{-aq}|x|^{bq}dx\geq \int_{B(0,r_{Q})}|x-x_{A}|^{-aq}|x|^{bq}dx\\
&\geq\int_{B(0,\frac{1}{2}|x_{A}|)}|x-x_{A}|^{-aq}|x|^{bq}dx \gtrsim |x_{A}|^{-aq}\int_{B(0,\frac{1}{2}|x_{A}|) }|x|^{bq}dx\gtrsim |x_{A}|^{-aq+bq+d}.
\end{aligned}
\end{equation}
In addition, for $b=-\frac dq+a$, we have that by $|x|\leq |x-x_{A}|+|x_{A}|\leq \frac{3}{2}|x-x_{A}|$ for any $x\in B(0,r_{Q})\backslash B(0,2|x_{A}|)$,
\begin{equation}\label{power1-7}
\begin{aligned}
&\int_{B(0,r_{Q})}|x-x_{A}|^{-aq}|x|^{bq}dx\geq \int_{B(0,r_{Q})\backslash B(0,2|x_{A}|)}|x-x_{A}|^{-aq}|x|^{bq}dx\\
&\geq \big(\frac{3}{2}\big)^{-aq}\int_{B(0,r_{Q})\backslash B(0,2|x_{A}|) }|x|^{-aq+bq}dx\gtrsim \log \frac{r_{Q}}{2|x_{A}|}.
\end{aligned}
\end{equation}
Hence, it follows from \eqref{power1-4}-\eqref{power1-7} that as $r_{Q}\rightarrow +\infty$,
\begin{equation*}
\begin{aligned}
&\frac{1}{\ave{u}_{Q}^{q}|Q|}\int_{Q}|x-x_{A}|^{-aq}|x|^{bq}dx\bigg(\frac{1}{|Q|}\int_{Q}|x|^{-bp'}dx\bigg)^{\frac{q}{p'}}\\
&\gtrsim \left\{
\begin{array}{ll}\vspace{1ex}
\big(\frac{r_{Q}}{|x_{A}|}\big)^{aq-bq-d}\rightarrow +\infty, & b<-\frac dq+a,\\
\log \frac{r_{Q}}{2|x_{A}|}\rightarrow +\infty, &b=-\frac dq+a.
\end{array}
\right.
\end{aligned}
\end{equation*}
Therefore, Proposition \ref{power3} is proved.
\end{proof}

\begin{rem}\label{remark-power}
If $u$ and $w$ are singular at the same point, then the condition $q<r_{u}$ is not necessary in Proposition \ref{power3}.
Indeed, choosing $u(x)=|x|^{-a}$ with $0<a<d$, then $|x|^{b}\in A^{u}_{p,q}$ holds if and only if
\begin{equation*}
\begin{aligned}
\left\{
\begin{array}{ll}\vspace{1ex}
-\frac dq+a<b<\frac d{p'}, & 1<p<\infty,\\
-\frac dq+a<b\leq 0, &p=1.
\end{array}
\right.
\end{aligned}
\end{equation*}
Hence, the condition $q<r_{u}=\frac{d}{a}$ can be removed in Proposition \ref{power3}.
\end{rem}

\subsection{Factorization for partial Muckenhoupt weights}

The standard factorization theorem for \(A_p\) weight means that a weight \(w\in A_p\) holds if and only if \(w = u\mu^{1-p}\)
for some \(u, \mu \in A_1\) (see \cite{J1980}). We now give two factorization results which will be used later.

\begin{prop}\label{FAC-1}
Let $u\in A_{1}$, $1\leq p,p_{0}<\infty$, $0<r, r_{0}<\infty$ with $\frac 1p -\frac 1{p_{0}}=\frac 1r-\frac 1{r_{0}}$ and $\gamma=\frac 1r+\frac 1{p'}$. Then it holds that

(i) when $1\leq p<p_{0}<\infty$, $0<r<r_{0}<\infty$, $\gamma\leq 1$, $w\in A_{p,r}^{u}$ and $u^{-1}\mu^{\gamma}\in A^{u}_{1,1/\gamma}$,
one has $\mu^{\gamma\frac{r-r_{0}}{r_{0}}}w^{\frac{r}{r_{0}}}\in A^{u}_{p_{0},r_{0}}$ and
$[\mu^{\gamma\frac{r-r_{0}}{r_{0}}}w^{\frac{r}{r_{0}}}]_{A^{u}_{p_{0},r_{0}}}\leq [u^{-1}\mu^{\gamma}]_{A_{1,1/\gamma}^{u}}^{\frac{r_{0}-r}{r_{0}}} [w]_{A^{u}_{p,r}}^{\frac{r}{r_{0}}}.$

(ii) when $1\leq p<p_{0}<\infty$, $0<r<r_{0}<\infty$, $\gamma> 1$, $w\in A_{p,r}^{u}$ and $u^{-1}\mu\in A^{u}_{1,1}$, one has $\mu^{\gamma\frac{r-r_{0}}{r_{0}}}w^{\frac{r}{r_{0}}}\in A^{u}_{p_{0},r_{0}}$ and
$[\mu^{\gamma\frac{r-r_{0}}{r_{0}}}w^{\frac{r}{r_{0}}}]_{A^{u}_{p_{0},r_{0}}}\leq [u^{-1}\mu]_{A_{1,1}^{u}}^{\gamma\frac{r_{0}-r}{r_{0}}} [w]_{A^{u}_{p,r}}^{\frac{r}{r_{0}}}.$

(iii) when $1\leq p_{0}<p<\infty$, $0<r_{0}<r<\infty$, $\gamma\leq 1$, $w\in A_{p,r}^{u}$ and $u^{-1}\mu^{\gamma}\in A^{u}_{1,1/\gamma}$,
we have $\mu^{\frac{\gamma(r-r_{0})}{r\gamma-1}}w^{\frac{r(r_{0}\gamma-1)}{r\gamma-1}}\in A^{u}_{p_{0},r_{0}}$ and
$\big[\mu^{\frac{\gamma(r-r_{0})}{(r\gamma-1)r_{0}}}w^{\frac{r(r_{0}\gamma-1)}{(r\gamma-1)r_{0}}}\big]_{A^{u}_{p_{0},r_{0}}}\leq  [u^{-1}\mu^{\gamma}]^{\frac{p_{0}-p}{(p-1)p_{0}}}_{A_{1,1/\gamma}} [w]_{A^{u}_{p,r}}^{\frac{p'}{p'_{0}}}.$

(iv) when $1\leq p_{0}<p<\infty$, $0<r_{0}<r<\infty$, $\gamma> 1$, $w\in A_{p,r}^{u}$ and $u^{-1}\mu\in A^{u}_{1,1}$, we have $\mu^{\frac{\gamma(r-r_{0})}{r\gamma-1}}w^{\frac{r(r_{0}\gamma-1)}{r\gamma-1}}\in A^{u}_{p_{0},r_{0}}$ and
$\big[\mu^{\frac{\gamma(r-r_{0})}{r\gamma-1}}w^{\frac{r(r_{0}\gamma-1)}{r\gamma-1}}\big]_{A^{u}_{p_{0},r_{0}}}\leq  [u^{-1}\mu]^{\frac{p_{0}-p}{(p-1)p_{0}}}_{A_{1,1}} [w]_{A^{u}_{p,r}}^{\frac{p'}{p'_{0}}}.$
\end{prop}

\begin{proof}
{\bf Case 1. $1\leq p<p_{0}<\infty$, $0<r<r_{0}<\infty$}

\vskip 0.1 true cm

Using H\"{o}lder's inequality with exponents $\frac{r_{0}}{\gamma(r_{0}-r)p'_{0}}$ and $\frac{p'r_{0}}{p'_{0}r}$,
and noting $r_{0}/p'_{0}=r_{0}\gamma-1$ and $r/p'=r\gamma-1$, one obtains
\begin{equation}\label{fac-1}
\begin{aligned}
&\big\langle \big(\mu^{\gamma\frac{r-r_{0}}{r_{0}}}w^{\frac{r}{r_{0}}}\big)^{-p'_{0}} \big\rangle^{\frac 1{p'_{0}}  }_{Q}=\big\langle \mu^{\gamma\frac{(r_{0}-r)p'_{0}}{r_{0}}}w^{-\frac{p'_{0}r}{r_{0}}} \big\rangle^{  \frac 1{p'_{0}}  }_{Q}
\leq \langle \mu\rangle^{\gamma\frac{r_{0}-r}{r_{0}}}_{Q}\ave{w^{-p'}}^{-\frac{r}{r_{0}p'}}_{Q}.
\end{aligned}
\end{equation}

$(i)$ For $\gamma\leq 1,$  by the definition of $A^{u}_{1,1/\gamma}$ and $u^{-1}\mu^{\gamma}\in A^{u}_{1,1/\gamma}$,
we have that for any cube $Q\subset\R^d$ and $x\in Q$,
\begin{equation*}
u(x)\mu(x)^{-\gamma}\leq [u^{-1}\mu^{\gamma}]_{A_{1,1/\gamma}^{u}} \ave{\mu}_{Q}^{-\gamma}\|u^{-1}\|^{-1}_{L^{\infty}(Q)}.
\end{equation*}
Thus,
\begin{equation}\label{fac-2}
\begin{aligned}
\big\langle u^{r_{0}}\mu^{\gamma(r-r_{0})}w^{r} \big\rangle^{1/r_{0}}_{Q}
&=\big\langle \big(u^{-1}\mu^{\gamma}\big)^{r-r_{0}}u^{r}w^{r} \big\rangle^{1/r_{0}}_{Q}\\
&\leq [u^{-1}\mu^{\gamma}]_{A_{1,1/\gamma}^{u}}^{\frac{r_{0}-r}{r_{0}}} \langle \mu\rangle^{-\gamma\frac{r_{0}-r}{r_{0}}}_{Q}\ave{u^{r}w^{r}}^{1/r_{0}}_{Q}\|u^{-1}\|_{L^{\infty}(Q)}^{\frac{r-r_{0}}{r_{0}}}.
\end{aligned}
\end{equation}
Multiplying \eqref{fac-1} by \eqref{fac-2} and $\|u^{-1}\|_{L^{\infty}(Q)}$ yields
$[\mu^{\gamma\frac{r-r_{0}}{r_{0}}}w^{\frac{r}{r_{0}}}]_{A^{u}_{p_{0},r_{0}}}\leq [u^{-1}\mu^{\gamma}]_{A_{1,1/\gamma}^{u}}^{\frac{r_{0}-r}{r_{0}}} [w]_{A^{u}_{p,r}}^{\frac{r}{r_{0}}}.$

$(ii)$ For $\gamma>1$ and $u^{-1}\mu\in A^{u}_{1,1}$, it holds that
\begin{equation*}
u(x)\mu(x)^{-1}\leq [u^{-1}\mu]_{A_{1,1}^{u}} \ave{\mu}_{Q}^{-1}\|u^{-1}\|^{-1}_{L^{\infty}(Q)}
\end{equation*}
and
\begin{equation}\label{fac-3}
\begin{aligned}
&\big\langle u^{r_{0}}\mu^{\gamma(r-r_{0})}w^{r} \big\rangle^{1/r_{0}}_{Q}
=\big\langle u^{(\gamma-1)(r-r_{0})}\big(u^{-\gamma}\mu^{\gamma}\big)^{r-r_{0}}u^{r}w^{r} \big\rangle^{1/r_{0}}_{Q}\\
&\leq [u^{-1}\mu]_{A_{1,1}^{u}}^{\frac{\gamma(r_{0}-r)}{r_{0}}} \|u^{-1}\|_{L^{\infty}(Q)}^{(\gamma-1)(r_{0}-r)}
\langle \mu\rangle^{-\gamma\frac{r_{0}-r}{r_{0}}}_{Q}\ave{u^{r}w^{r}}^{1/r_{0}}_{Q}\|u^{-1}\|_{L^{\infty}(Q)}^{\gamma\frac{r-r_{0}}{r_{0}}}\\
&\leq [u^{-1}\mu]_{A_{1,1}^{u}}^{\frac{\gamma(r_{0}-r)}{r_{0}}}
\langle \mu\rangle^{-\gamma\frac{r_{0}-r}{r_{0}}}_{Q}\ave{u^{r}w^{r}}^{1/r_{0}}_{Q}\|u^{-1}\|_{L^{\infty}(Q)}^{\frac{r-r_{0}}{r_{0}}}.
\end{aligned}
\end{equation}
Multiplying \eqref{fac-1} by \eqref{fac-3} and $\|u^{-1}\|_{L^{\infty}(Q)}$ yields
$\big[\mu^{\gamma\frac{r-r_{0}}{r_{0}}}w^{\frac{r}{r_{0}}}\big]_{A^{u}_{p_{0},r_{0}}}
\leq [u^{-1}\mu]_{A_{1,1}^{u}}^{\gamma\frac{r_{0}-r}{r_{0}}} [w]_{A^{u}_{p,r}}^{\frac{r}{r_{0}}}$.

\vskip 0.1 true cm

{\bf Case 2. $1\leq p_{0}<p<\infty$, $0<r_{0}<r<\infty$}

\vskip 0.1 true cm

For any cube $Q$, we have
$\ave{\mu}_{Q}\|\mu^{-1}\|_{L^{\infty}(Q)}\leq \ave{v}_{Q}\|uv^{-\gamma}\|^{1/\gamma}_{L^{\infty}(Q)}\|u^{-1}\|^{1/\gamma}_{L^{\infty}(Q)}$.
Together with  $\mu\in A_{1}$ and $[\mu]_{A_{1}}\leq [u^{-1}\mu^{\gamma}]^{1/\gamma}_{A^{u}_{1,1/\gamma}}$, this yields
\begin{equation}\label{fac-4}
\begin{aligned}
&\big\langle \big(\mu^{\frac{\gamma(r-r_{0})}{(r\gamma-1)r_{0}}}w^{\frac{r(r_{0}\gamma-1)}{(r\gamma-1)r_{0}}}\big)^{-p'_{0}}
\big\rangle^{\frac 1p '_{0}}_{Q}=\big\langle \mu^{\frac{\gamma(r_{0}-r)p'_{0}}{(r\gamma-1)r_{0}}}w^{-p'} \big\rangle^{\frac 1p '_{0}}_{Q}\\
&\leq [\mu]^{\frac{\gamma(r-r_{0})}{(r\gamma-1)r_{0}}}_{A_{1}}\langle \mu\rangle^{\frac{\gamma(r_{0}-r)}{(r\gamma-1)r_{0}}}_{Q}\ave{w^{-p'}}^{-\frac{1}{p'_{0}}}_{Q}
\leq [u^{-1}\mu^{\gamma}]^{\frac{r-r_{0}}{(r\gamma-1)r_{0}}}_{A_{1,1/\gamma}}\langle \mu\rangle^{\frac{\gamma(r_{0}-r)}{(r\gamma-1)r_{0}}}_{Q}\ave{w^{-p'}}^{-\frac{1}{p'_{0}}}_{Q}.
\end{aligned}
\end{equation}

$(iii)$ For $\gamma\leq 1,$ by the definition of $A^{u}_{1,1/\gamma}$ and $u^{-1}\mu^{\gamma}\in A^{u}_{1,1/\gamma}$,
we obtain that for any cube $Q\subset\R^d$ and $x\in Q$,
$u(x)\mu(x)^{-\gamma}\leq [u^{-1}\mu^{\gamma}]_{A_{1,1/\gamma}^{u}} \ave{\mu}_{Q}^{-\gamma}\|u^{-1}\|^{-1}_{L^{\infty}(Q)}$.
Thus,
\begin{equation}\label{fac-5}
\begin{aligned}
&\big\langle u^{r_{0}}\mu^{\frac{\gamma(r-r_{0})}{r\gamma-1}}w^{\frac{r(r_{0}\gamma-1)}{r\gamma-1}} \big\rangle^{1/r_{0}}_{Q}
=\big\langle \big(u^{-1}\mu^{\gamma}\big)^{\frac{r_{0}-r}{r\gamma-1}}(uw)^{\frac{r(r_{0}\gamma-1)}{r\gamma-1}} \big\rangle^{1/r_{0}}_{Q}\\
&\leq [u^{-1}\mu^{\gamma}]_{A_{1,1/\gamma}^{u}}^{\frac{r-r_{0}}{r_{0}(r\gamma-1)}}
\langle \mu\rangle^{-\frac{\gamma (r_{0}-r)}{r_{0}(r\gamma-1)}}_{Q}
\ave{u^{r}w^{r}}^{\frac{p'}{p'_{0}r}}_{Q}\|u^{-1}\|_{L^{\infty}(Q)}^{\frac{r-r_{0}}{r_{0}(r\gamma-1)}}.
\end{aligned}
\end{equation}
Multiplying \eqref{fac-4} by \eqref{fac-5} and $\|u^{-1}\|_{L^{\infty}(Q)}$ yields
$[\mu^{\gamma\frac{r-r_{0}}{r_{0}}}w^{\frac{r}{r_{0}}}]_{A^{u}_{p_{0},r_{0}}}
\leq [u^{-1}\mu^{\gamma}]_{A_{1,1/\gamma}^{u}}^{\frac{r_{0}-r}{r_{0}}} [w]_{A^{u}_{p,r}}^{\frac{r}{r_{0}}}.$

$(iv)$ It follows from $\gamma>1$ and $u^{-1}\mu\in A^{u}_{1,1}$ that
$u(x)\mu(x)^{-1}\leq [u^{-1}\mu]_{A_{1,1}^{u}} \ave{v}_{Q}^{-1}\|u^{-1}\|^{-1}_{L^{\infty}(Q)}$
and
\begin{equation}\label{fac-6}
\begin{aligned}
&\big\langle u^{r_{0}}\mu^{\frac{\gamma(r-r_{0})}{r\gamma-1}}w^{\frac{r(r_{0}\gamma-1)}{r\gamma-1}} \big\rangle^{1/r_{0}}_{Q}
=\big\langle u^{\frac{(\gamma-1)(r_{0}-r)}{r\gamma-1}}\big(u^{-\gamma}\mu^{\gamma}\big)^{\frac{r_{0}-r}{r\gamma-1}}(uw)^{\frac{r(r_{0}\gamma-1)}{r\gamma-1}} \big\rangle^{1/r_{0}}_{Q}\\
&\leq [u^{-1}\mu]_{A_{1,1}^{u}}^{\frac{\gamma(r-r_{0})}{r_{0}(r\gamma-1)}}
\langle \mu\rangle^{-\frac{\gamma (r_{0}-r)}{r_{0}(r\gamma-1)}}_{Q}
\ave{u^{r}w^{r}}^{\frac{p'}{p'_{0}r}}_{Q}\|u^{-1}\|_{L^{\infty}(Q)}^{\frac{r-r_{0}}{r_{0}(r\gamma-1)}}.
\end{aligned}
\end{equation}
Applying  \eqref{fac-4} and \eqref{fac-6} to get
$\big[\mu^{\gamma\frac{r-r_{0}}{r_{0}}}w^{\frac{r}{r_{0}}}\big]_{A^{u}_{p_{0},r_{0}}}
\leq [u^{-1}\mu]_{A_{1,1}^{u}}^{\frac{r_{0}-r}{r_{0}}} [w]_{A^{u}_{p,r}}^{\frac{r}{r_{0}}}$.
\end{proof}

\subsection{The reverse H\"{o}lder property}

Note that for $w\in A_{\infty}$, there exist two constants $C$ and $\epsilon>0$ depending only on the space dimension $d$ and
$[w]_{A_{\infty}}$ such that for any cube $Q$,
$\ave{w^{1+\epsilon}}^{\frac{1}{1+\epsilon}}_{Q}\leq C \ave{w}_{Q}$.

For $\vec w\in A_{\vec p,q}^{u}$, Corollary \ref{cor-Apqu} implies that $u^{q}w^{q}, w_{i}^{-p_{i}'}\in A_{\infty}$
for $i=1,\ldots, m$.
Then there exist some constants $\epsilon_{0},\epsilon_{1},\cdots \epsilon_{m}>0$ such that for $w\in A_{\vec p,q}^{u}$ and any cube $Q$,
\begin{equation}\label{RHI-0}
\big\langle(uw)^{q(1+\epsilon_{0})}\big\rangle^{\frac{1}{1+\epsilon_{0}}}_{Q}\lesssim \ave{u^{q}w^{q}}_{Q}
\end{equation}
and
\begin{equation}\label{RHI-i}
\big\langle w_{i}^{-p'_{i}(1+\epsilon_{i})}\big\rangle^{\frac{1}{p'_{i}(1+\epsilon_{i})}}_{Q}
\lesssim \ave{w^{-p'_{i}}}^{\frac{1}{p'_{i}}}_{Q}, \qquad i=1,\ldots, m.
\end{equation}

From \eqref{RHI-0} and \eqref{RHI-i}, one can immediately derive the following Proposition \ref{Apqu-q+}.
\begin{prop}\label{Apqu-q+}
Let $u\in A_{1}$, $0<q< \infty$, $1< p_{1},\ldots, p_{m}< \infty$ and $\vec w=(w_{1},\ldots,w_{m})\in A^{u}_{\vec p,q}.$
Then

(i) there exists a constant $\epsilon>0$ such that
$\vec w\in A^{u}_{\vec p,q(1+\epsilon)}\subset A^{u}_{(p_{1}(1+\epsilon),\ldots,p_{m}(1+\epsilon)),q(1+\epsilon)}.$

(ii) there exists a constant $\epsilon$ with \(0 < \epsilon < \min\{p_1 - 1,\cdots,p_{m}-1\}\) such that
$\vec w\in A^{u}_{(p_{1}(1-\epsilon),\ldots,p_{m}(1-\epsilon)),q}$

$\subset A^{u}_{(p_{1}(1-\epsilon),\ldots,p_{m}(1-\epsilon)),q(1-\epsilon)}.$
\end{prop}

It is well known that if $w \in A_{p}$, then $w^{\eta} \in A_{p}$ for $0 < \eta < 1$.
Note that there are some minor differences for the weight in $A_{p,q}^u$.

\begin{prop}\label{aeta}
Let $u\in A_{1}$, $1\leq p_{1},\ldots,p_{m}\leq \infty$, $0<q<\infty$ and $0<\eta<1$. When $q\leq r_{u}$, for $\vec{w}
=(w_{1},\ldots,w_{m})\in  A_{\vec p, q}^{u}$, we have $\vec{w}^{\eta}=(w_{i}^{\eta},\cdots w_{n}^{\eta})\in A_{\vec p,q}^{u}$
and $[\vec w^{\eta}]_{A^{u}_{\vec p,q}}\lesssim [\vec w]^{\eta}_{A^{u}_{\vec p,q}}$.
Conversely, when $q> r_{u}$, there exist $u\in A_{1}$ and $\vec{w}\in  A_{\vec p, q}^{u}$ such that  $\vec{w}^{\eta}\notin A_{\vec p,q}^{u}$.
\end{prop}
\begin{proof}
{\bf Case 1. $q\leq r_{u}$}
\vskip 0.1 true cm

Given $ w \in  A_{\vec p, q}^{u}$, let $\epsilon_{0}>0$ be as in \eqref{RHI-0}. Then,
by $0<(1-\eta)\cdot \frac{1+\epsilon_{0}}{1+\epsilon_{0}-\eta}<1$, we arrive at
\begin{align*}
\ave{u^{q}w^{q\eta}}^{\frac{1}{q}}_{Q}&\leq \big\langle u^{q(1-\eta)\cdot \frac{1+\epsilon_{0}}{1+\epsilon_{0}-\eta}}\big\rangle_{Q}^{\frac{1+\epsilon_{0}-\eta}{q(1+\epsilon_{0})}}
\big\langle(uw)^{q(1+\epsilon_{0})}\big\rangle ^{\frac{\eta}{q(1+\epsilon_{0})}}_{Q}
\lesssim \ave{u}_{Q}^{1-\eta}\ave{u^{q}w^{q}}^{\frac{\eta}{q}}_{Q}.
\end{align*}
Therefore,
\begin{equation*}
\begin{aligned}
\big[\vec w^{\eta}\big]_{A^{u}_{\vec p,q}}
&=\sup_{Q}\ave{u^{q}w^{q\eta}}^{\frac{1}{q}}_{Q}\|u^{-1}\|_{L^{\infty}(Q)}
\prod_{i=1}^{m}\ave{w_{i}^{-p'_{i}\eta}}_{Q}^{\frac{1}{p'_{i}}}\\
&\lesssim \sup_{Q} \ave{u}_{Q}^{1-\eta}\ave{u^{q}w^{q}}^{\frac{\eta}{q}}_{Q}\|u^{-1}\|_{L^{\infty}(Q)}
\prod_{i=1}^{m}\ave{w_{i}^{-p'_{i}}}_{Q}^{\frac{\eta}{p'_{i}}}\lesssim [\vec w]^{\eta}_{A^{u}_{\vec p,q}}.
\end{aligned}
\end{equation*}

{\bf Case 2. $q> r_{u}$}
\vskip 0.1 true cm

Without loss of generality, only $m = 1$ and $\vec{p}=p>1$ are considered.
Taking $u(x)=|x|^{-a}$, where $0<a=d/r_{u}<\min\{d,d/q\}$. If $w(x)=|x|^{-d/q+a+\epsilon}$ with $\epsilon\in (0, \frac{(1-\eta)(-d/q+a)}{1+\eta})$ being small enough such that $0<-d/q+a<-d/q+a+\epsilon<d/p',$
then $w\in A^{u}_{p,q}$ holds by Remark \ref{remark-power}. On the other hand, if $0<\epsilon<\frac{(1-\eta)(-d/q+a)}{1+\eta}$, then
$(-d/q+a+\epsilon)\eta<-d/q+a$
and $w^{\eta}\notin A_{p,q}^{u}$.
\end{proof}

\begin{prop}\label{aeta}
Let $u\in A_{1}$, $1\leq p_{1},\ldots,p_{m}\leq \infty$ and $0<q<\infty$. For $\vec w=(w_{1},\ldots,w_{n})\in A^{u}_{\vec p,q},$
there exists a constant $\epsilon>0$ such that $\vec w^{1+\epsilon}=(w_{1}^{1+\epsilon},\ldots,w_{m}^{1+\epsilon})\in A^{u}_{\vec p,q}$
and $[\vec w^{1+\epsilon}]_{A^{u}_{\vec p,q}}\lesssim [\vec w]^{1+\epsilon}_{A^{u}_{\vec p,q}}$.
\end{prop}
\begin{proof}
For $ w \in  A_{\vec p, q}^{u}$, let $\epsilon_{0}, \ldots,\epsilon_{m}>0$ be as in \eqref{RHI-0} and \eqref{RHI-i}. Taking $\epsilon=\min\{\epsilon_{0},\epsilon_{1},\ldots, \epsilon_{m}\}$. Then
\begin{align*}
\big\langle u^{q}w^{q(1+\epsilon)}\big\rangle^{\frac{1}{q}}_{Q}=\big\langle u^{-q\epsilon}(uw)^{q(1+\epsilon)}\big\rangle^{\frac{1}{q}}_{Q}&\leq \|u^{-1}\|_{L^{\infty}(Q)}^{\epsilon}\langle u^{q}w^{q}\rangle ^{\frac{1+\epsilon}{q}}_{Q}
\end{align*}
and
\begin{align*}
\big\langle w_i^{-p'_{i}(1+\epsilon)}\big\rangle^{\frac{1}{p'_{i}(1+\epsilon)}}_{Q}\lesssim \ave{w_i^{-p'_{i}}}^{\frac{1}{p'_{i}}}_{Q}, 
\qquad i=1,\ldots, m.
\end{align*}
This yields $\vec w^{1+\epsilon}\in A^{u}_{\vec p,q}$
and $[\vec w^{1+\epsilon}]_{A^{u}_{\vec p,q}}\lesssim [\vec w]^{1+\epsilon}_{A^{u}_{\vec p,q}}$.
\end{proof}

We end this subsection with the following consequence of the reverse H\"{o}lder property, which provides a characterization 
of partial multiple weights.
\begin{prop}\label{c2}
Let $u\in A_{1}$, $1\leq p_{1},\ldots,p_{m}\leq \infty$, $0<q<\infty$ and $0\leq a <1$.
For $ \vec w=(w_{1},\ldots, w_{m}) \in A_{\vec p, q}^{u^{a}}$, there is a
constant $b=b(n,\vec p, q, a ,[ w ]_{A_{\vec p, q}^{u^{a}}})$ with $b\in (a,1]$
such that $ \vec w \in A_{\vec p, q}^{\tilde{ a }}$. That is,
\begin{align}\label{YH-3}
A_{\vec p,q}^{u^{a}}=\bigcup_{b\in (a ,1]}A_{\vec p,q}^{u^{b}}.
\end{align}
Especially, for $a=0$, $A_{\vec p,q}=\bigcup_{b\in (0,1]}A_{\vec p,q}^{u^{b}}.$
\end{prop}
\begin{proof}
We first show that for $0\leq a_{1}< a _{2}\leq 1$,
\begin{equation}\label{Apq-a1a2}
A_{\vec p, q}^{u^{a_{2}}}\subset A_{\vec p, q}^{u^{a_{1}}} \qquad \text {and} \qquad [\vec w]_{A_{\vec p, q}^{u^{a_{1}}}}\leq [\vec w]_{A_{\vec p, q}^{u^{a_{2}}}}.
\end{equation}
In fact, for any cube $Q$ and $u\in A_{1}$, one has
\begin{equation*}
\begin{aligned}
\big\langle \big(u^{a_{1}}w\big)^{q}\big\rangle^{\frac{1}{q}}_{Q}&=\big\langle u^{(a_{1}-a_{2})q}\cdot \big(u^{a_{2}}w\big)^{q}\big\rangle^{\frac{1}{q}}_{Q}\leq \|u^{-a_{1}}\|^{-1}_{L^{\infty}(Q)}\|u^{-a_{2}}\|_{L^{\infty}(Q)}\big\langle \big(u^{a_{2}}w\big)^{q}\big\rangle^{\frac{1}{q}}_{Q},
\end{aligned}
\end{equation*}
which yields \eqref{Apq-a1a2}.

Given $ w \in A_{\vec p, q}^{u^{a}}$, let $\epsilon_{0}>0$ be as in \eqref{RHI-0} and
choose $b=a+\frac{\epsilon_{0}}{q(1+\epsilon_{0})}$. Then
\begin{equation*}
\begin{aligned}
\big\langle \big(u^{b}w\big)^{q}\big\rangle^{\frac{1}{q}}_{Q}&=\big\langle u^{(b-a)q}\cdot \big(u^{a}w\big)^{q}\big\rangle^{\frac{1}{q}}_{Q}
\leq \big\langle \big(u^{a}w\big)^{q(1+\epsilon_{0})}\big\rangle^{\frac{1}{q(1+\epsilon_{0})}}_{Q}
\big\langle u^{\frac{(1+\epsilon_{0})(b-a)q}{\epsilon_{0}}}\big\rangle_{Q}^{\frac{\epsilon_{0}}{q(1+\epsilon_{0})}}\lesssim \big\langle \big(u^{a}w\big)^{q}\big\rangle^{\frac{1}{q}}_{Q} \ave{u}_{Q}^{b-a}.
\end{aligned}
\end{equation*}
Due to $\ave{u}_{Q}^{b-a}\|u^{-b}\|^{-1}_{L^{\infty}(Q)}\lesssim \|u^{-a}\|_{L^{\infty}(Q)}$,
then $[\vec{w}]_{A^{u^{b}}_{\vec{p},q}}\lesssim [\vec{w}]_{A^{u^{a}}_{\vec{p},q}}$. This means that \eqref{YH-3}
holds  by \eqref{Apq-a1a2}.
\end{proof}

\subsection{The relation between $A^{u}_{\vec p,q}$ and $A_{\vec p,q}$.}

Note that $A^u_{\vec{p}, q} \subset A_{\vec{p}, q}$ holds when $u \in A_{1}$.
We next study the following reverse embedding.

\begin{prop}\label{Apq-Appu}
Let $u\in A_{1}$, $1\leq p_{1},\ldots,p_{m}\leq \infty$ and $0<q<r_{u}$. Then $A_{\vec p,q}\subset A_{\vec p,\tilde{q}}^{u}$
for $\frac{1}{\tilde{q}}=\frac{1}{q}+\frac{1}{r_{u}}$.
\end{prop}
\begin{proof}
For any $\vec w=(w_{1},\ldots,w_{m})\in A_{\vec p,q}$, we have $w^{q}\in A_{\infty}$.
In addition, there is a constant $\epsilon>0$ such that $\vec w\in A_{\vec p,q(1+\epsilon)}$
and $[\vec w]_{A_{\vec p,q(1+\epsilon)}}\lesssim [\vec w]_{A_{\vec p,q}}$. By H\"{o}lder inequality
with exponent $\frac{q(1+\epsilon)}{\tilde{q}}$ and $(\frac{q(1+\epsilon)}{\tilde{q}})'$,
one has that due to $\tilde{q}(\frac{q(1+\epsilon)}{\tilde{q}})'<r_{u}$,
\begin{align*}
[\vec w]_{A^{u}_{\vec p,\tilde{q}}}&=\sup_{Q}\ave{u^{\tilde{q}}w^{\tilde{q}}}^{\frac{1}{\tilde{q}}}_{Q} \big(\prod_{i=1}^m\ave{w_i^{-p_i'}}^{\frac{1}{p_i'}}_{Q}\big) \big\|u^{-1}\big\|_{L^{\infty}(Q)}\\
&\leq \sup_{Q}\ave{w^{q(1+\epsilon)}}^{\frac{1}{q(1+\epsilon)}}_{Q}
\ave{u^{\tilde{q}(\frac{q(1+\epsilon)}{\tilde{q}})'}}^{\frac{1}{(\frac{q(1+\epsilon)}{\tilde{q}})'\tilde{q}}}_{Q}
\big(\prod_{i=1}^m\ave{w_i^{-p_i'}}^{\frac{1}{p_i'}}_{Q}\big) \big\|u^{-1}\big\|_{L^{\infty}(Q)}\\
&\lesssim \sup_{Q}\ave{w^{q(1+\epsilon)}}^{\frac{1}{q(1+\epsilon))}}_{Q}
\big(\prod_{i=1}^m\ave{w_i^{-p_i'}}^{\frac{1}{p_i'}}_{Q}\big)
\ave{u}_{Q} \big\|u^{-1}\big\|_{L^{\infty}(Q)}=[\vec w]_{A_{\vec p, q(1+\epsilon)}^{u}}.
\end{align*}
This, together with  Proposition \ref{Apqu-q+}, derives Proposition \ref{Apq-Appu}.
\end{proof}

We further analyze the connection between $A^{u}_{\vec p,q}$ and $A_{\vec p,q}$, as described below.
\begin{prop}\label{Apq-Apqu}
Assume $u\in A_{1}$, $0<q\leq \infty$ and $1\leq p\leq \infty$.

(i) If $w\in A_{p, q}^{u}$, then $uw\in A_{p, q}.$

(ii) If $w\in A_{p,q}$ with $1<p<r_{u}$, then $u^{-1}w\in A_{\frac{pr_{u}}{r_{u}-p},q}^{u}$. Moreover,
the index $\frac{pr_{u}}{r_{u}-p}$ in $A_{\frac{pr_{u}}{r_{u}-p},q}^{u}$ is sharp.

(iii) If $w\in A_{1,q}$, then $u^{-1}w\in A_{r, q}^{u}$ for any $r\in (r'_{u},\infty)$.
In addition, there exists a function $w$ such that $w\in A_{1,q}$ and $u^{-1}w\notin A_{r'_{u},q}^{u}$.
\end{prop}
\begin{proof}
$(i)$ For $p>1$, by the definition of $A_{p, q}$ and $u\in A_{1}$, we have
\begin{align*}
[uw]_{A_{p,q}}&=\sup_{Q}\ave{u^{q}w^{q}}^{\frac 1q }_{Q}\ave{u^{-p'}w^{-p'}}^{\frac 1p '}_{Q}
\leq \sup_{Q}\ave{u^{q}w^{q}}^{\frac 1q }_{Q} \ave{w^{-p'}}^{\frac 1p '}_{Q} \big\|u^{-1}\big\|_{L^{\infty}(Q)}=[w]_{A_{p, q}^{u}}.
\end{align*}
An analogous argument holds for $p=1$.

$(ii)$ $w\in A_{p,q}$ implies that $w^{-1}\in A_{q',p'}$ holds and there exists a constant $\epsilon>0$ such that for any cube $Q$,
$\big\langle w^{-p'(1+\epsilon)}\big\rangle^{\frac{1}{p'(1+\epsilon)}}_{Q}\lesssim \ave{w^{-p'}}^{\frac{1}{p'}}_{Q}.$
Set $s=\frac{(r_{u}+p')(1+\epsilon)}{r_{u}}$. Then $\frac{p'r_{u}}{r_{u}+p'}s'<r_{u}$ and
\begin{equation}\label{Apq-Apqu-1}
\begin{aligned}
&\big\langle(u^{-1}w)^{-\frac{p'r_{u}}{r_{u}+p'}}\big\rangle^{\frac{r_{u}+p'}{p'r_{u}}}_{Q}\lesssim \langle w^{-p'(1+\epsilon)}\rangle_{Q}^{\frac{1}{p'(1+\epsilon)}} \ave{u^{\frac{p'r_{u}}{r_{u}+p'}s'}}^{\frac{r_{u}+p'}{p'r_{u}s'}}_{Q}\lesssim \ave{w^{-p'}}^{\frac{1}{p'}}_{Q}\ave{u}_{Q}.
\end{aligned}
\end{equation}
It follows from \eqref{Apq-Apqu-1} that
\begin{align*}
[u^{-1}w]_{A_{\frac{pr_{u}}{r_{u}-p},q}^{u}}
&=\sup_{Q}\ave{w^{q}}_{Q}^{\frac{1}{q}}\big\langle(u^{-1}w)^{-\frac{p'r_{u}}{r_{u}+p'}}
\big\rangle^{\frac{p'r_{u}}{r_{u}+p'}}_{Q}\|u^{-1}\|_{L^{\infty}(Q)}\\
&\lesssim \sup_{Q}\ave{w^{q}}_{Q}^{\frac{1}{q}}\ave{w^{-p'}}^{\frac{1}{p'}}_{Q}\ave{u}_{Q}\|u^{-1}\|_{L^{\infty}(Q)}\lesssim [w]_{A_{p,q}}.
\end{align*}
Hence, $w\in A_{p,q}$ implies $u^{-1}w\in A_{\frac{pr_{u}}{r_{u}-p},q}^{u}$.

Next we prove that for any $\bar{p}\in \big(p, p_{u}\big)$ with $p_{u}=\frac{pr_{u}}{r_{u}-p}$, there is a function $w$ such that $w\in A_{p,q}$ and $u^{-1}w\notin A_{\bar{p}, q}^{u}$.

Note that
$\frac{1}{2}(\frac{1}{p'_{u}}+\frac{1}{\bar{p}'})>\frac{1}{p}>\frac{1}{r_{u}}$ and $\frac{1}{2}(\frac{1}{p'_{u}}+\frac{1}{\bar{p}'})< \frac{1}{p'_{u}}=\frac{1}{p}+\frac{1}{r_{u}}$.
Then $1-\frac{r_{u}}{2}(\frac{1}{p'_{u}}+\frac{1}{\bar{p}'})<0$ and $(\frac{r_{u}}{2}(\frac{1}{p'_{u}}+\frac{1}{\bar{p}'}-1)p'<r_{u}$
hold. This yields $w\in A_{p,q}$ for $w=u^{1-\frac{r_{u}}{2}(\frac{1}{p'_{u}}+\frac{1}{\bar{p}'})}$.

On the other hand, by $u^{-1}w=u^{-\frac{r_{u}}{2}(\frac{1}{p'_{u}}+\frac{1}{\bar{p}'})}$ and
$\frac{r_{u}}{2}(\frac{1}{p'_{u}}+\frac{1}{\bar{p}'})\bar{p}'>r_{u},$
we have that for any cube $Q$,
$\ave{(u^{-1}w)^{-\bar{p}'}}_{Q}=\ave{u^{\frac{r_{u}}{2}(\frac{1}{p'_{u}}+\frac{1}{\bar{p}'})\bar{p}'}}_{Q}>\ave{u^{r_{u}}}_{Q}$.
Thus, $w\notin A_{\bar{p},q}^{u}$ by the definition of $r_{u}$.

$(iii)$ By direct computation and $r'<r_{u}$, one has
\begin{align*}
[u^{-1}w]_{A_{r,q}^{u}}
&=\sup_{Q}\ave{w^{q}}_{Q}^{\frac{1}{q}}\ave{(u^{-1}w)^{-r'}}^{\frac{1}{r'}}_{Q}\|u^{-1}\|_{L^{\infty}(Q)}\\
&\lesssim \sup_{Q}\ave{w^{q}}_{Q}^{\frac{1}{q}}\ave{u^{r'}}^{\frac{1}{r'}}_{Q}\|w^{-1}\|_{L^{\infty}(Q)}\|u^{-1}\|_{L^{\infty}(Q)}\lesssim [w]_{A_{1,q}}.
\end{align*}
Then it is easy to know $w\equiv 1\in A_{1,q}$ and $u^{-1}w=u^{-1}\notin A_{r'_{u},q}^{u}$.
\end{proof}

\section{Extrapolation theorems for partial multiple weights}\label{extra}

There have been extensive studies on the generalized  Rubio de Francia extrapolation theorems
(see \cite{AM2007}, \cite{HMS1988}, \cite{CMP2011} and \cite{D2011}). We next focus on the
extrapolation for partial multiple weights.

Let $f$ be a nonnegative measurable function, $u\in A_{1}$ and $\gamma\leq 1$. Define
\begin{equation}\label{def-Mur}
M_{u,\gamma}(f)(x)= u(x)^{1/\gamma}\sup_{Q\ni x} \ave{f}_{Q}\ave{u^{-1/\gamma}}_{Q}.
\end{equation}

\begin{lem}\label{ex-lem1}
Suppose that $1< p<\infty$, $0<q<\infty$ and $\gamma=\frac 1q +\frac 1{p'} \leq 1$ and $w\in A_{p,q}^{u}$. If $\gamma\leq 1$, then $(w^{1/\gamma},u^{1/\gamma})\in A_{(q\gamma,\infty),q\gamma}$ and
$\|M_{u,\gamma}(f)\|_{L^{q\gamma}(w^{q})}\lesssim \|f\|_{L^{q\gamma}(w^{q})}.$
\end{lem}
\begin{proof}
Note that $(w^{1/\gamma},u^{1/\gamma})\in A_{(q\gamma,\infty),q\gamma}$ holds since
\begin{equation*}
\begin{aligned}
\big[(w^{1/\gamma},u^{1/\gamma})\big]_{A_{(q\gamma,\infty),q\gamma}}
&=\sup_{Q}\ave{u^{q}w^{q}}^{\frac{1}{q\gamma}}_{Q} \ave{w^{\frac{-q}{q\gamma-1}}}^{\frac{q\gamma-1}{q\gamma}}_{Q}\ave{u^{-\frac{1}{\gamma}}}_{Q}\\
&\leq \sup_{Q}\ave{u^{q}w^{q}}^{\frac{1}{q\gamma}}_{Q} \ave{w^{-p'}}^{-p'/\gamma}_{Q}\ave{u^{-1}}_{Q}\|u^{-1}\|^{1/\gamma-1}_{L^{\infty}(Q)}\\
&\leq [u]^{1/\gamma-1}_{A_{1}}[w]^{\frac{1}{\gamma}}_{A^{u}_{p,q}}.
\end{aligned}
\end{equation*}
In addition,
\begin{equation*}
\begin{aligned}
\|M_{u,\gamma}(f)\|_{L^{q\gamma}(w^{q})}&=\|\mathcal{M}(f,u^{-1/\gamma})\|_{L^{q\gamma}(u^{q}w^{q})}
\lesssim \|f\|_{L^{q\gamma}(w^{q})}\|u^{-1/\gamma}\cdot u^{1/\gamma}\|_{L^{\infty}(\R^d)}=\|f\|_{L^{q\gamma}(w^{q})}.
\end{aligned}
\end{equation*}
Then the proof of Lemma \ref{ex-lem1} is completed.
\end{proof}

The author in \cite{RDF} introduced  a construction of $A_1$ weight, which is now known as the Rubio de Francia algorithm.
Similarly, we intend to establish a construction of $A^u_{1,1/\gamma}$ weight for $u \in A_1$ and $\gamma\leq 1$.

\begin{lem}\label{ex-lem2}
Assume that $u\in A_{1}$, $1< p<\infty$, $0<t<\infty$, $\gamma=\frac 1t+\frac 1{p'} \leq 1$,
$w\in A_{p,t}^{u}$, and $f$ is a nonnegative function in $L^{t\gamma}(w^{t})$. Let
$M^{k}_{u,\gamma}$ be the $k$-th iteration of $M_{u,\gamma}$, $M_{u,\gamma}^{0}f=f$ and $\|M_{u,\gamma}\|_{L^{t\gamma}(w^{t})}$
be the norm of $M_{u,\gamma}$ as a bounded operator on $L^{t}(w^{t})$. Define
\begin{equation*}
R_{\gamma}(f)(x)=\sum_{k=0}^{\infty}\frac{M_{u,\gamma}^{k}f(x)}{(2\|M_{u,\gamma}\|_{L^{t\gamma}(w^{t})})^{k}}.
\end{equation*}
Then $f(x)\leq R_{\gamma}(f)(x)$ a.e., $\|R_{\gamma}(f)\|_{L^{t\gamma}(w^{r})}\leq 2\|f\|_{L^{t\gamma}(w^{t})}$,
$u^{-1}R_{\gamma}(f)^{\gamma}$ is an $A^{u}_{1,1/\gamma}$ weight and $[u^{-1}R_{\gamma}(f)^{\gamma}]_{A^{u}_{1,1/\gamma}}\leq 2\|M_{u,\gamma}\|_{L^{t\gamma}(w^{t})}[u]_{A_{1}}$ holds.
\end{lem}
\begin{proof}
The first property is immediate. For the second one, it comes from $\|M_{u,\gamma}^k\|_{L^{t\gamma}(w^{t})}\leq \|M_{u,\gamma}\|^{k}_{L^{t\gamma}(w^{t})}$ and the direct sum of a geometric series.

Due to $\gamma\leq 1$, then for any cube $Q$, one has
\begin{equation}
\begin{aligned}
\ave{R_{\gamma}(f)}_{Q}
&\leq \sum_{k=0}^{\infty}\frac{\ave{M_{u,\gamma}^{k}(f)}_{Q}}{(2\|M_{u,\gamma}\|_{L^{t\gamma}(w^{t})})^{k}}\\
&\leq \inf_{Q\ni x}\sum_{k=0}^{\infty}\frac{M_{u,\gamma}^{k+1}f(x)}{(2\|M_{u,\gamma}\|_{L^{t\gamma}(w^{t})})^{k}}u(x)^{-1/\gamma}\ave{u^{-1/\gamma}}^{-1}_{Q}\\
&\leq 2\|M_{u,\gamma}\|_{L^{t\gamma}(w^{t})} \ave{u}^{1/\gamma}_{Q}\inf_{Q\ni x}u(x)^{-1/\gamma}R_{\gamma}(f)(x)
\end{aligned}
\end{equation}
and $[u^{-1}R_{\gamma}(f)^{\gamma}]_{A^{u}_{1,1/\gamma}}\leq 2\|M_{u,\gamma}\|^{\gamma}_{L^{t\gamma}(w^{t})}[u]_{A_{1}}$.
\end{proof}

\subsection{The extrapolation theorem of Rubio de Francia with $A_{p,q}^u$ weights}

Based on the crucial Lemma \ref{ex-lem2}, Proposition \ref{FAC-1} and the properties of multiple weights, we now
give the proof of Theorem \ref{ex-Apqu}.
\medskip

\noindent
{\bf Proof of Theorem \ref{ex-Apqu}.}
Set $\gamma:=\frac{1}{t_{0}}+\frac{1}{p'_{0}}=\frac{1}{t}+\frac{1}{p'}$.

\vskip 0.1 true cm

{\bf Case 1. $q<q_{0}$}

\vskip 0.1 true cm
Let $f\in L^{p}(w^{p})$. Define the non-negative function $H$ by $H^{t\gamma}w^{t}=f^{p}w^{p}$. Then $H\in L^{t\gamma}(w^{t})$.
By Lemma \ref{ex-lem2}, set
\begin{equation*}
R_{\gamma}(H)(x)= \left\{
\begin{array}{ll}\vspace{1ex}
\sum_{k=0}^{\infty}\frac{M_{u,\gamma}^{k}(H)(x)}{(2\|M_{u,\gamma}\|_{L^{t}(w^{t})})^{k}}, & \gamma\leq 1,\\
R_{1}(H)(x), & \gamma > 1,
\end{array}
\right.
\end{equation*}
we then have
\begin{equation*}
\left\{
   \begin{array}{ll}\vspace{1ex}
u^{-1}R_{\gamma}(H)^{\gamma}\in A^{u}_{1,1/\gamma}, & \gamma\leq 1,\\
u^{-1}R_{\gamma}(H)\in A^{u}_{1,1}, & \gamma > 1.
   \end{array}
 \right.
\end{equation*}
Note that
\begin{equation}\label{ex-Apqu-eq1}
\begin{aligned}
\int_{\R^d}g^{q}w^{q}&=\int_{\R^d}g^{q}w^{q}R_{\gamma}(H)^{\frac{q\gamma(t-t_{0})}{t_{0}}}R_{\gamma}(H)^{\frac{q\gamma(t_{0}-t)}{t_{0}}}\\
&\leq \Big(\int_{\R^d}g^{q_{0}}\big(R_{\gamma}(H)^{\gamma(t-t_{0})}w^{t}\big)^{\frac{q_{0}}{t_{0}}}\Big)^{\frac q{q_{0}} }
\Big(\int_{\R^d}R_{\gamma}(H)^{t\gamma}w^{t}\Big)^{1-\frac q{q_{0}}}.
\end{aligned}
\end{equation}
It follows from $(i)$-$(ii)$ of Proposition \ref{FAC-1} with $\mu=R_{\gamma}(H)$ that
$R_{\gamma}(H)^{\gamma(t-t_{0})/r_{0}}w^{t/t_{0}}\in A_{p_{0},t_{0}}^{u}$ and
$$\big[R_{\gamma}(H)^{\gamma\frac{t-t_{0}}{t_{0}}}w^{\frac{t}{t_{0}}}\big]_{A^{u}_{p_{0},t_{0}}}\leq [u]_{A_{1}}^{\frac{t_{0}-t}{t_{0}}}[u^{-1}R_{\gamma}(H)^{\gamma}]_{A_{1,\max\{1,1/\gamma\}}^{u}}^{\frac{t_{0}-t}{t_{0}}} [w]_{A^{u}_{p,t}}^{\frac{t}{t_{0}}}.$$
By $(g,f)\in\mathcal{F}$ and $R_{\gamma}(H)^{-1}\leq H^{-1}$, we arrive at
$$
\Big(\int_{\R^d}g^{q_{0}}\big(R_{\gamma}(H)^{\gamma(t-t_{0})}w^{t}\big)^{\frac{q_{0}}{t_{0}}}\Big)^{\frac 1{q_{0}} }
\lesssim \Big(\int_{\R^d}f^{p_{0}}\big(R_{\gamma}(H)^{\gamma(t-t_{0})}w^{t}\big)^{\frac{p_{0}}{t_{0}}}\Big)^{\frac 1{p_{0}} }
= \Big(\int_{\R^d}f^{p}w^{p}\Big)^{\frac 1{p_{0}}}.$$
On the other hand,
\begin{equation*}
\begin{aligned}
\big(\int_{\R^d}R_{\gamma}(H)^{t\gamma}w^{t}\big)^{1-\frac q{q_{0}} }&\lesssim \big(\int_{\R^d}H^{t\gamma}w^{t}\big)^{1-\frac q{q_{0}} }
=\big(\int_{\R^d}f^{p}w^{p}\big)^{1-\frac q{q_{0}} }.
\end{aligned}
\end{equation*}
Inserting this into \eqref{ex-Apqu-eq1} yields
$\big(\int_{\R^d}g^{q}w^{q}\big)^{\frac 1q }\lesssim \big(\int_{\R^d}f^{p}w^{p}\big)^{\frac 1p }$.

\vskip 0.1 true cm

{\bf Case 2. $q>q_{0}$}
\vskip 0.1 true cm

We will use the duality method to show Theorem  \ref{ex-Apqu}.
To this end, choosing a nonnegative function $h\in L^{\frac{q}{q-q_{0}}}(w^{q})$ with unit norm.
For $w\in A_{p,t}^{u}$, define $H\in L^{p'\gamma}(w^{-p'})$ by $H^{p'\gamma}w^{-p'}=h^{q/(q-q_{0})}w^{q}$. According to
$(iii)$ and $(iv)$ of Proposition \ref{FAC-1}, we obtain $R_{\gamma}(H)^{\frac{\gamma(t-t_{0})}{(t\gamma-1)tr_{0}}}w^{\frac{t(t_{0}\gamma-1)}{(t\gamma-1)t_{0}}} \in {A^{u}_{p_{0},t_{0}}}$ and
\begin{equation*}
\begin{aligned}
\int_{\R^d}g^{q_{0}}hw^{q}&=\int_{\R^d}g^{q_{0}}(H^{p'\gamma}w^{-p'+q})^{\frac{q-q_{0}}{q}}w^{q}
\leq \int_{\R^d}g^{q_{0}}w^{q_{0}p'/p'_{0}}R_{\gamma}(H)^{\frac{t\gamma(q-q_{0})}{q(t\gamma-1)}}\\
&\lesssim \big(\int_{\R^d}f^{p_{0}}w^{p'(p_{0}-1)}R_{\gamma}(H)^{\frac{p_{0}\gamma(t-t_{0})}{t_{0}(t\gamma-1)}}\big)^{\frac{q_{0}}{p_{0}}}
\lesssim \big(\int_{\R^d}f^{p}w^{p}\big)^{\frac{q_{0}}{p_{0}}}.
\end{aligned}
\end{equation*}\qed

There is a variant version of Theorem \ref{ex-Apqu} in which the related strong type assumption is replaced
by a weak type estimate.

\begin{cor}\label{ex-Apqu-w}
Let $u\in A_{1}$, $1\leq p_{0}<\infty$ and $0<q_{0}, t_{0}<\infty$. Assume that for any $w_{0}\in A^{u}_{p_{0},t_{0}}$,
\begin{equation}\label{1}
\begin{aligned}
\|g\|_{L^{q_{0},\infty}(w_{0}^{q_{0}})}\lesssim \|f\|_{L^{p_{0}}(w_{0}^{p_{0}})}, \qquad (g,f)\in\mathcal{F}.
\end{aligned}
\end{equation}
Then for all $1<p<\infty$ and $0<q,t<\infty$ such that $\frac 1p -\frac 1p _{0}=\frac 1q -\frac 1q _{0}=1/t-1/t_{0}$,
and for all $A_{p,t}^{u}$ we have
\begin{equation}\label{2}
\begin{aligned}
\|g\|_{L^{q,\infty}(w^{q})}\lesssim \|f\|_{L^{p}(w^{p})}, \qquad (g,f)\in\mathcal{F}.
\end{aligned}
\end{equation}
\end{cor}

\begin{proof}
For fixed $\lambda>0$, define $h=\lambda \chi_{g>\lambda}.$
It follows from \eqref{1} that
\begin{equation*}
\begin{aligned}
\|h\|_{L^{q_{0}}(w_{0}^{q_{0}})}&=\lambda w_{0}^{q_{0}}\big(\{x: g(x)>\lambda\}\big)^{\frac{1}{q_{0}}}\lesssim \|f\|_{L^{p_{0}}(w_{0}^{p_{0}})}.
\end{aligned}
\end{equation*}
This yields that for the nonnegative couple $(h,f)$ and any $w_{0}\in A^{u}_{p_{0},t_{0}}$,
\begin{equation*}
\begin{aligned}
\big(\int_{\R^d}h(x)^{q_{0}}w^{q_{0}}dx\Big)^{\frac 1q _{0}}\lesssim \big(\int_{\R^d}f(x)^{p_{0}}w^{p_{0}}dx\Big)^{\frac 1p _{0}},
\quad  (h,f)\in\mathcal{F}.
\end{aligned}
\end{equation*}
Applying Theorem \ref{ex-Apqu}, we have that for $1<p<\infty$, $0<q,t<\infty$ and $w\in A_{p,t}^{u}$,
$\|h\|_{L^{q}(w^{q})}\lesssim \|f\|_{L^{p}(w^{p})}$.
Thus, \eqref{2} is shown by $\|g\|_{L^{q,\infty}(w^{q})}=\sup_{\lambda>0}\|h\|_{L^{q}(w^{q})}.$
\end{proof}

In addition, we have the following stronger result.
\begin{cor}
Suppose that $T$ is defined on $\bigcup_{1\leq \rho<\infty}\bigcup_{w\in A_{\rho}}L^{\rho}(w)$ and takes values in the space of measurable complex-valued functions. Let $u\in A_{1}$. For some fixed $1\leq p_{0}<\infty$ and $0<q_{0}<\infty$,
assume that for any $w_{0}\in A^{u}_{p_{0},q_{0}}$, one has
\begin{equation*}
\begin{aligned}
\|T(f)\|_{L^{q_{0},\infty}(w_{0}^{q_{0}})}\lesssim \|f\|_{L^{p_{0}}(w_{0}^{p_{0}})}.
\end{aligned}
\end{equation*}
Then for all $1<p<\infty$ and $0<q<\infty$ such that $\frac 1p -\frac 1q =\frac 1{p_{0}}-\frac 1{q_{0}}$, and for all $f\in A_{p,q}^{u}$
we have
\begin{equation}\label{YH-7}
\begin{aligned}
\|T(f)\|_{L^{q}(w^{q})}\lesssim \|f\|_{L^{p}(w^{p})}.
\end{aligned}
\end{equation}
\end{cor}
\begin{proof}
By Proposition \ref{Apqu-q+}, there exists a constant $\epsilon\in (0,1)$ such that for any 
$w\in A_{p,q}^{u}$, $w\in A^{u}_{p_{1},q_{1}}\bigcap A^{u}_{p_{2},q_{2}}$ holds, where
$p_{1}=p(1+\epsilon)$, $q_{1}=q(1+\epsilon)$, $p_{2}=p(1-\epsilon)$ and $q_{2}=q(1-\epsilon).$
Then the proof of \eqref{YH-7} follows from Corollary \ref{ex-Apqu-w} and the Marcinkiewicz 
interpolation theorem.
\end{proof}

\subsection{Extrapolation theorem for two types of off-diagonal estimates}

Let us recall some fundamental embedding properties and duality principle of Lorentz spaces
as follows (see \cite{L1950}, \cite{L1951}, or \cite[Theorems 6.3. and 6.22]{CR2016}).
\begin{lem}\label{Lorentz-emb}
Let $0< p\leq \infty$ and $0<r<q\leq \infty$. Then $L^{p,r}(\R^d)\subset L^{p,q}(\R^d)$.
\end{lem}

\begin{lem}\label{Lorentz-dua}
Let $1< p<\infty$ and $1\leq q<\infty$. Then $(L^{p,q}(\R^d))^{*}= L^{p',q'}(\R^d)$.
\end{lem}

Given $1<s<\infty$  and a nonnegative locally integrable function $f$, it is known that $M(f)^{\frac1s}$ belongs 
to the Muckenhoupt class $A_1$. Below we present a generalization of this property for the systems of two functions.

\begin{lem}\label{A1}
Given $1< s_1,s_2<\infty$ with $\frac1{s_1}+\frac1{s_2}<1$ and two nonnegative locally integrable functions $f_1, f_2$, we have $M(f_{1})^{\frac{1}{s_1}}M(f_2)^{\frac1{s_2}}\in A_1$.
\end{lem}
\begin{proof}
For any cube $Q$, it suffices to show
\begin{equation}\label{M1M2}
\frac{1}{|Q|}\int_{Q}M(f_{1})(y)^{\frac{1}{s_1}}M(f_{2})(y)^{\frac{1}{s_2}}dy\lesssim M(f_1)(x)^{\frac1{s_1}}M(f_2)(x)^{\frac1{s_2}} 
\quad \text{for almost all~} x\in Q.
\end{equation}
To prove \eqref{M1M2}, we write
$f_i=f_i\chi_{3Q}+f_i\chi_{(3Q)^c}, i=1,2.$
Note that there exits a constant $\epsilon\in (0, \min\{s_1-1,s_2-1\})$ such that $\frac1{s_1-\epsilon}+\frac1{s_2-\epsilon}=1$.
Then it holds that
\begin{align*}
&\frac{1}{|Q|}\int_{Q}M(f_{1}\chi_{3Q})(y)^{\frac{1}{s_1}}M(f_{2}\chi_{3Q})(y)^{\frac{1}{s_2}}dy
\leq \prod_{i=1}^2\Big(\frac{1}{|Q|}\int_{Q}M(f_{i}\chi_{3Q})(y)^{\frac{s_i-\epsilon}{s_i}}dy\Big)^{\frac1{s_i-\epsilon}}\\
&\lesssim \prod_{i=1}^2 \Big(\frac{1}{|Q|}\int_{\R^d}(f_{i}\chi_{3Q})(y)dy\Big)^{\frac1{s_i}}
\lesssim M(f_{1})(x)^{\frac{1}{s_1}}M(f_{2})(x)^{\frac{1}{s_2}}
\end{align*}
in view of Kolmogorov's inequality. And for $f_i\chi_{(3Q)^c}$, one has
$$M(f_i\chi_{(3Q)^c})(y)\lesssim M(f_{i})(x)$$
since any ball $B$ centered at $x$ that gives a nonzero average for $f_{i}\chi_{(3Q)^c}$ must have radius at least the side length of $Q$,
and then $\sqrt{d}B$ contains $x$. Hence \eqref{M1M2} also holds when $f_{i}$ is replaced by $f_i\chi_{(3Q)^c}$ on the left
hand side. In addition, it holds that for almost all $x\in Q$,
\begin{equation*}
\begin{aligned}
&\frac{1}{|Q|}\int_{Q}M(f_{1})(y)^{\frac{1}{s_1}}M(f_{2}\chi_{(3Q)^c})(y)^{\frac{1}{s_2}}dy\\
&\lesssim \frac{1}{|Q|}\int_{Q}M(f_{1})(y)^{\frac{1}{s_1}}dyM(f_2)(x)^{\frac1{s_2}}\lesssim M(f_1)(x)^{\frac1{s_1}}M(f_2)(x)^{\frac1{s_2}}
\end{aligned}
\end{equation*}
and
\begin{equation*}
\frac{1}{|Q|}\int_{Q}M(f_{1}\chi_{(3Q)^c})(y)^{\frac{1}{s_1}}M(f_{2})(y)^{\frac{1}{s_2}}dy\lesssim M(f_1)(x)^{\frac1{s_1}}M(f_2)(x)^{\frac1{s_2}}.
\end{equation*}
Collecting all the estimates above and using the subadditivity property
$M(f+g)^{\frac{1}{s_i}}\leq M(f)^{\frac1{s_i}}+M(g)^{\frac1{s_i}}$, we obtain \eqref{M1M2}.
\end{proof}

The following inclusion holds for partial multiple weight classes.
\begin{lem}\label{Apqu1-Appu2}
Let $u, u_2\in A_{1}$, $1<p\leq q_1<q_2<r_{u}$ and $\frac{1}{q_1}=\frac{1}{q_2}+\frac{1}{r_{u}}$. Then  the inclusion $A^{u_2}_{p,q_{2}}\subset A_{p,q_1}^{u_1}$ holds when $u_1:=uu_2\in A_{1}$.
\end{lem}
\begin{proof}
For any $w\in A^{u_2}_{p,q_2}$, we have $u_2^{q_2}w^{q_2}\in A_{\infty}$. There is a constant $\epsilon>0$ such that $w\in A^{u_2}_{p,q_2(1+\epsilon)}$
and $[w]_{A^{u_2}_{p,q_2(1+\epsilon)}}\lesssim [w]_{A^{u_2}_{p,q_2}}$. By H\"{o}lder inequality
with exponent $\frac{q(1+\epsilon)}{q_1}$ and $(\frac{q(1+\epsilon)}{q_1})'$,
one has that due to $q_1(\frac{q(1+\epsilon)}{q_1})'<r_{u}$,
\begin{align*}
[w]_{A^{u_1}_{p,q_1}}&=\sup_{Q}\ave{u_1^{q_1}w^{q_1}}^{\frac{1}{q_1}}_{Q} \ave{w^{-p'}}^{\frac{1}{p'}}_{Q}   
\big\|u_1^{-1}\big\|_{L^{\infty}(Q)}\\
&\leq \sup_{Q}\ave{u_2^{q_2(1+\epsilon)}w^{q_2(1+\epsilon)}}^{\frac{1}{q_2(1+\epsilon)}}_{Q}
\ave{u^{q_1(\frac{q_2(1+\epsilon)}{q_1})'}}^{\frac{1}{(\frac{q_2(1+\epsilon)}{q_1})'q_1}}_{Q}
\ave{w^{-p'}}^{\frac{1}{p'}}_{Q} \big\|u_1^{-1}\big\|_{L^{\infty}(Q)}\\
&\lesssim \sup_{Q}\ave{u_2^{q(1+\epsilon)}w^{q(1+\epsilon)}}^{\frac{1}{q(1+\epsilon))}}_{Q}
\ave{w^{-p'}}^{\frac{1}{p'}}_{Q}\big\|u_2^{-1}\big\|_{L^{\infty}(Q)}
\ave{u}_{Q} \big\|u^{-1}\big\|_{L^{\infty}(Q)}\\
&=[w]_{A_{p, q_2(1+\epsilon)}^{u_2}}\lesssim [w]_{A_{p, q_2}^{u_2}}.
\end{align*}
This completes the proof.
\end{proof}

Next, for any $v\in A^{u_2}_{p,q_2}$ and $U\in L^{\frac{d}{\beta_2-\beta_1},\infty}(\R^d)$,  we construct a function $u_1\in A_{1}\bigcap M^{\frac{d}{\alpha-\beta_1}}_{s_{1}}(\R^{d})$ such that $u_1\geq u_2U$ and $v\in A^{u_1}_{p,q_1}$.

\begin{lem}\label{D-OD-2-lem}
Let $0\leq \beta_{1}<\beta_2<\alpha<d$, $1<p\leq q_1< q_2<\frac{d}{\beta_2-\beta_1}$, $\frac{1}{p}=\frac{1}{q_1}+\frac{\beta_1}{d}=\frac{1}{q_2}+\frac{\beta_2}{d}$, $u_{2}\in A_{1}\bigcap M^{\frac{d}{\alpha-\beta_2}}_{s_2}(\R^d)$, $v\in A^{u_2}_{p,q_2}$ and $U\in L^{\frac{d}{\beta_2-\beta_1},\infty}(\R^d)$. There exists a constant $s_{1}>1$ with $\frac1{s_1}>\frac1{s_2}+\frac{\beta_2-\beta_1}{d}$,
and a function $u_1\geq u_2U$ such that $u_1\in A_{1}\bigcap M^{\frac{d}{\alpha-\beta_1}}_{s_{1}}(\mathbb{R}^{d})$, $v\in A_{p,q_1}^{u_1}$ and
\begin{align*}
\|u_1\|_{M^{\frac{d}{\alpha-\beta_1}}_{s_{1}}(\mathbb{R}^{d})}\lesssim \|U\|_{L^{\frac{d}{\beta_2-\beta_1},\infty}(\R^d)}\|u_2\|_{M^{\frac d{\alpha-\beta_2}}_{s_2}(\R^d)}.
\end{align*}
\end{lem}
\begin{proof}
For any weight $v\in A^{u_2}_{p_2,q_2}$, there exist a sufficiently small constant $\epsilon>0$ such that $v\in A^{u_2}_{p,q_2(1+\epsilon)}$ (see Proposition \ref{Apqu-q+}). Choosing $\epsilon_{0}\in (0, \frac{d}{\beta_2-\beta_1}-q_2)$ such that
$\frac{1}{q_1}=\frac{1}{q_2(1+\epsilon)}+\frac{1}{\frac{d}{\beta_2-\beta_1}-\epsilon_{0}}$.
Define the functions $u$ and $u_1$ by
$u= M(U^{\frac{d}{\beta_2-\beta_1}-\epsilon_{0}})^{\frac{1}{{\frac{d}{\beta_2-\beta_1}-\epsilon_{0}}}}$
and $u_1= uM(u_{2}^{s_2-\epsilon_0})^{\frac1{s_2-\epsilon_0}}$.
By the Kolmogorov inequality, we obtain such an embedding relation
$$\|U\|_{M^{\frac{d}{\beta_2-\beta_1}}_{s}(\R^d)}\lesssim \|U\|_{L^{\frac{d}{\beta_2-\beta_1},\infty}(\R^d)}, 1\leq s< \frac{d}{\beta_2-\beta_1},$$
which ensures that the definition of $u_1$ is well-posed. Moreover, $u_1\in A_1$ follows by Lemma \ref{A1}.
By the boundedness of the Hardy-Littlewood maximal operator $M$ on Morrey spaces
and $\frac1{s_1}>\frac1{s_2}+\frac{\beta_2-\beta_1}d$, one has
\begin{align*}
\|u_1\|_{M^{\frac{d}{\alpha-\beta_1}}_{s_{1}}(\mathbb{R}^{d})}
&\leq \|M(U^{\frac{d}{\beta_2-\beta_1}-\epsilon_{0}})^{\frac{1}{{\frac{d}{\beta_2-\beta_1}-\epsilon_{0}}}}\|
_{M^{\frac{d}{\beta_2-\beta_1}}_{\frac{d}{\beta_2-\beta_1}-\frac{\epsilon_{0}}{2}}(\mathbb{R}^{d})}
\| M(u_{2}^{s_2-\epsilon_0})^{\frac1{s_2-\epsilon_0}}\|_{M^{\frac d{\alpha-\beta_2}}_{s_2}(\R^d)}\\
&\lesssim \|U\|_{M^{\frac{d}{\beta_2-\beta_1}}_{\frac{d}{\beta_2-\beta_1}-\frac{\epsilon_{0}}{2}}(\mathbb{R}^{d})}
\|u_2\|_{M^{\frac d{\alpha-\beta_2}}_{s_2}(\R^d)}\lesssim \|U\|_{L^{\frac{d}{\beta_2-\beta_1},\infty}(\R^d)}\|u_2\|_{M^{\frac d{\alpha-\beta_2}}_{s_2}(\R^d)}.
\end{align*}
In addition, $u^{r}=M(U^{\frac{d}{\beta_2-\beta_1}-\epsilon_{0}})^{\frac{r}{\frac{d}{\beta_2-\beta_1}-\epsilon_{0}}}\in A_{1}$ for any $1<r<\frac{d}{\beta_2-\beta_1}-\epsilon_{0}$, then for any cube $Q$,
$$\ave{u^{r}}_{Q}\leq [u^{r}]_{A_{1}}\inf_{Q\ni x}u(x)^{r}\leq [u^{r}]_{A_{1}}\ave{u}^{r}_{Q}.$$
Hence, $r_{u}=\frac{d}{\beta_2-\beta_1}-\epsilon_{0},$ and $v\in A^{u_1}_{p,q_1}$ by Proposition \ref{Apqu1-Appu2}.
\end{proof}

For the case of $\beta_2=\alpha$ in Lemma \ref{D-OD-2-lem}, one has a similar result as follows.
\begin{lem}\label{D-OD-3-lem}
Let $0\leq \beta<\alpha<d$, $2\alpha-\beta<d$, $1<p< q<\frac{d}{\alpha-\beta}$, $\frac{1}{p}=\frac{1}{q}+\frac{\alpha}{d}=\frac{1}{\tilde{q}}+\frac{\beta}{d}$, $v\in A_{p,q}$ and $U\in L^{\frac{d}{\alpha-\beta},\infty}(\R^d)$. There exists a constant $s_{0}\in (1,\frac{d}{\alpha-\beta})$ with $\frac{d}{\alpha-\beta}-s_{0}$ sufficiently small, 
and a function $u\geq U$ such that $u\in A_{1}\bigcap M^{\frac{d}{\alpha-\beta}}_{s_{0}}(\mathbb{R}^{d})$, $v\in A_{p,\tilde{q}}^{u}$ and
$\|u\|_{M^{\frac{d}{\alpha-\beta}}_{s_{0}}(\mathbb{R}^{d})}\lesssim
\|U\|_{L^{\frac{d}{\alpha-\beta},\infty}(\mathbb{R}^{d})}$.
\end{lem}
\begin{proof}
The proof procedure is analogous to that of Lemma \ref{D-OD-2-lem}.
For any weight $v\in A_{p,q}$, there exist a sufficiently small constant $\epsilon>0$ such that $v\in A_{p,q(1+\epsilon)}$.
Choosing $\epsilon_{0}\in (0, \frac{d}{\alpha-\beta}-q)$ such that
$\frac{1}{\tilde{q}}=\frac{1}{q(1+\epsilon)}+\frac{1}{\frac{d}{\alpha-\beta}-\epsilon_{0}}$.
Define the function $u$ and the exponent $s_{0}$ by
$u=M(U^{\frac{d}{\alpha-\beta}-\epsilon_{0}})^{\frac{1}{{\frac{d}{\alpha-\beta}-\epsilon_{0}}}}$
and $s_{0}=\frac{d}{\alpha-\beta}-\frac{\epsilon_{0}}{2}$.
It follows from the Kolmogorov inequality that
$\|U\|_{M^{\frac{d}{\alpha-\beta}}_{s}(\R^d)}\lesssim \|U\|_{L^{\frac{d}{\alpha-\beta},\infty}(\R^d)}$
when $1\leq s< \frac{d}{\alpha-\beta}$,
which ensures that the definition of $u$ is well-posed.
Due to the boundedness of the Hardy-Littlewood maximal operator $M$ on Morrey spaces and
$s_{0}>\frac{d}{\alpha-\beta}-\epsilon_{0}$, then
$$\|u\|_{M^{\frac{d}{\alpha-\beta}}_{s_{0}}(\mathbb{R}^{d})}
=\|M(U^{\frac{d}{\alpha-\beta}-\epsilon_{0}})\|^{\frac{1}{\frac{d}{\alpha-\beta}-\epsilon_{0}}}
_{M^{\frac{d}{(\alpha-\beta)(\frac{d}{\alpha-\beta}-\epsilon_{0})}}_{\frac{s_{0}}{{\frac{d}{\alpha-\beta}-\epsilon_{0}}}}(\mathbb{R}^{d})}
\lesssim \|U\|_{M^{\frac{d}{\alpha-\beta}}_{s_{0}}(\mathbb{R}^{d})}$$
and $v\in A^{u}_{p,\tilde{q}}$ by Proposition \ref{Apq-Appu}.
\end{proof}

\vspace{0.3cm}
\noindent
{\bf Proof of Theorem \ref{D-OD-2}.}
For any $w_2\in A^{u_2}_{p, q_2}$ with $1<p<q_2<\frac{d}{\alpha-\beta_2}$ and $\frac1{p}-\frac1{q_2}=\frac{\beta_2}{d}$, applying Lemmas \ref{Lorentz-emb} and \ref{Lorentz-dua} with $1<r<\frac{q_2}{q_1}$ yields
\begin{equation*}
\begin{aligned}
\|g\|^{q_1}_{L^{q_2}(u_2^{q_2}w_2^{q_2})}&=\|g^{q_1}u_{2}^{q_1}w_2^{q_1}\|_{L^{\frac{q_2}{q_1}}(\R^d)}\lesssim \|g^{q_1}u_{2}^{q_1}w_2^{q_1}\|_{L^{\frac{q_2}{q_1}, r}(\R^d)}\\
&=\sup_{\|W\|_{L^{(\frac{q_2}{q_1})',r'}(\R^d)}\neq 0} \Big|\int_{\R^d}g(x)^{q_1}u_2^{q_1}w_2(x)^{q_1}W(x)dx
\Big|\|W\|^{-1}_{L^{(\frac{q_2}{q_1})',r'}(\R^d)}.
\end{aligned}
\end{equation*}
If $W\in L^{(\frac{q_2}{q_1})',r'}(\R^d)$ (hence $W\in L^{(\frac{q_2}{q_1})',\infty}(\R^d)$), then $U:=W^{\frac1{q_1}}\in L^{\frac{d}{\beta_2-\beta_1},\infty}(\R^d)$ and
\begin{equation}\label{D-OD-2-eq1}
\begin{aligned}
\|g\|^{q_1}_{L^{q_2}(u_2^{q_2}w_2^{q_2})}&\leq \sup_{\|U\|_{L^{\frac{d}{\beta_2-\beta_1},\infty}(\R^d)}\neq 0} \big|\int_{\R^d}g(x)^{q_1}u_2^{q_1}w_2(x)^{q_1}U(x)^{q_1}dx\big|\|U\|^{-q_1}_{L^{\frac{d}{\beta_2-\beta_1},\infty}(\R^d)}.
\end{aligned}
\end{equation}
By Lemma \ref{D-OD-2-lem}, there exists a function $u_1$ such that $u_1\geq u_2 U$, $u_1\in M^{\frac d{\alpha-\beta_1}}_{s_1}(\R^d)$ 
and $w_2\in A^{u_1}_{p,q_1}$. Therefore,
\begin{equation}\label{D-OD-2-eq2}
\begin{aligned}
&\big|\int_{\R^d}g(x)^{q_1}u_2(x)^{q_1}w_2(x)^{q_1}U(x)^{q_1}dx\big|\leq \int_{\R^d}|g(x)|^{q_1}u_1(x)^{q_1}w_2(x)^{q_1}dx\\
&\lesssim \|u_{1}\|^{q_1}_{M^{\frac{d}{\alpha-\beta_1}}_{s_1}(\R^d)}\|f\|^{q_1}_{L^{p}(w_2^{p})}\lesssim \|u_{2}\|^{q_1}_{M^{\frac{d}{\alpha-\beta_2}}_{s_2}(\R^d)}\|U\|_{L^{\frac{d}{\beta_2-\beta_1},\infty}(\R^d)}\|f\|^{q_1}_{L^{p}(w_2^{p})},
\end{aligned}
\end{equation}
where $\frac1{s_1}>\frac1{s_2}+\frac{\beta_2-\beta_1}{d}$. Collecting \eqref{D-OD-2-eq1} and \eqref{D-OD-2-eq2} shows that
for any $w_2\in A^{u_2}_{p,q_2}(\R^d)$,
$\|g\|_{L^{q_2}(u_2^{q_2}w_2^{q_2})}\lesssim \|u_{2}\|_{M^{\frac{d}{\alpha-\beta_2}}_{s_2}(\R^d)}\|f\|_{L^{p}(w_2^{p})}.$
\qed

\vspace{0.3cm}
\noindent
{\bf Proof of Theorem \ref{D-OD-3}.}
The proof is similar to that of Theorem \ref{D-OD-2}, with only a few minor differences in the details.
Indeed, for any $v\in A_{p, q_2}$ with $1<p<q_2<\frac{d}{\alpha-\beta}$ and $\frac1p-\frac1{q_1}=\frac{\beta}{d}$, we have
\begin{equation}\label{D-OD-3-eq1}
\begin{aligned}
\|g\|^{q_1}_{L^{q_2}(v^{q_2})}&\leq \sup_{\|U\|_{L^{\frac{d}{\alpha-\beta},\infty}(\R^d)}\neq 0} \big|\int_{\R^d}g(x)^{q_1}v(x)^{q_1}U(x)^{q_1}dx\big|\|U\|^{-q_1}_{L^{\frac{d}{\alpha-\beta},\infty}(\R^d)}.
\end{aligned}
\end{equation}
In addition, it follows from Lemma \ref{D-OD-3-lem} that
\begin{equation}\label{D-OD-3-eq2}
\begin{aligned}
&\big|\int_{\R^d}g(x)^{q_1}v(x)^{q_1}U(x)^{q_1}dx\big|\leq \int_{\R^d}|g(x)|^{q_1}v(x)^{q_1}u(x)^{q_1}dx\\
&\lesssim \|u\|^{q_1}_{M^{\frac{d}{\alpha-\beta}}_{s}(\R^d)}\|f\|^{q_1}_{L^{p}(v^{p})}\lesssim \|u\|^{q_1}_{M^{\frac{d}{\alpha-\beta}}_{s_{0}}(\R^d)}\|f\|^{p}_{L^{p}(v^{p})}
\lesssim \|U\|^{p}_{L^{\frac{d}{\alpha},\infty}(\R^d)}\|f\|^{p}_{L^{p}(v^{p})},
\end{aligned}
\end{equation}
where $1<s<s_{0}<\frac{d}{\alpha-\beta}$.
Combining \eqref{D-OD-3-eq1} with \eqref{D-OD-3-eq2} yields
$\|g\|_{L^{q}(v^{q})}\lesssim \|f\|_{L^{p}(v^{p})}.$ \qed

\vspace{0.3cm}
\noindent
{\bf Proof of Theorem \ref{D-OD}.} As mentioned in subsection \ref{1.3-C}, according to Theorems \ref{ex-Apqu}, \ref{D-OD-2} and \ref{D-OD-3},
one can get the required conclusion. \qed

\vspace{0.5cm}

At the end of this subsection, we present the converse result of Theorem \ref{D-OD}.
By applying the H\"{o}lder inequality in weak Lebesgue spaces, one knows that for $u\in L^{\frac{d}{\alpha},\infty}(\R^d)$ and $\frac{1}{p}=\frac{1}{q}+\frac{\alpha}{d}$,
\begin{equation}
\|f\|_{L^{p,\infty}_{u^{p}w^{p}}(\R^d)}\leq \|u\|_{L^{\frac{d}{\alpha},\infty}(\R^d)}\|f\|_{L^{q,\infty}_{w^{q}}(\R^d)},
\end{equation}
where $\|f\|_{L^{q,\infty}_{w^{q}}(\R^d)}:=\|fw\|_{L^{q,\infty}(\R^d)}$.

\begin{thm}\label{OD-D}
Suppose that $T$ is defined on $\bigcup_{1\leq \rho<\infty}\bigcup_{w\in A_{\rho}}L^{\rho}(w)$ and takes values in measurable
complex-valued function space. Let $1\leq p_{0}<q_{0}<\infty$, $0<\alpha<d$ and $\frac {1}{p_{0}} -\frac {1}{q_{0}} =\frac \alpha d$. If it holds that for any weight $v\in A_{p_{0},q_{0}}$,
$\|T(f)\|_{L^{q_{0}}(v^{q_{0}})}\lesssim \|f\|_{L^{p_{0}}(v^{p_{0}})},$
then for any $u\in A_{1}\bigcap L^{\frac{d}{\alpha-\beta},\infty}(\R^d)$ and any weight $w\in A_{p,q}$ with $1<p\leq t<q<\infty$,
$\frac 1{p} -\frac 1{q} =\frac \alpha d$ and $\frac 1{p} -\frac 1{t} =\frac \beta d$, one has
$\|T(f)\|_{L^{q}(u^{q}w^{q})}\lesssim \|f\|_{L^{p}(w^{p})}$.
\end{thm}

The proof of Theorem \ref{OD-D} is based on the combinations of extrapolation arguments, embedding theorems between
the different function spaces
and Marcinkiewicz interpolation theorem.
For readers' convenience,  we present the proof structure in the following picture.
In addition, without loss of generality and for conciseness, it suffices to treat the case
of $\beta=0$.

\begin{figure}[H]
    \centering
    \begin{subfigure}[b]{1.0\textwidth}
        \centering
        \includegraphics[width=17cm]{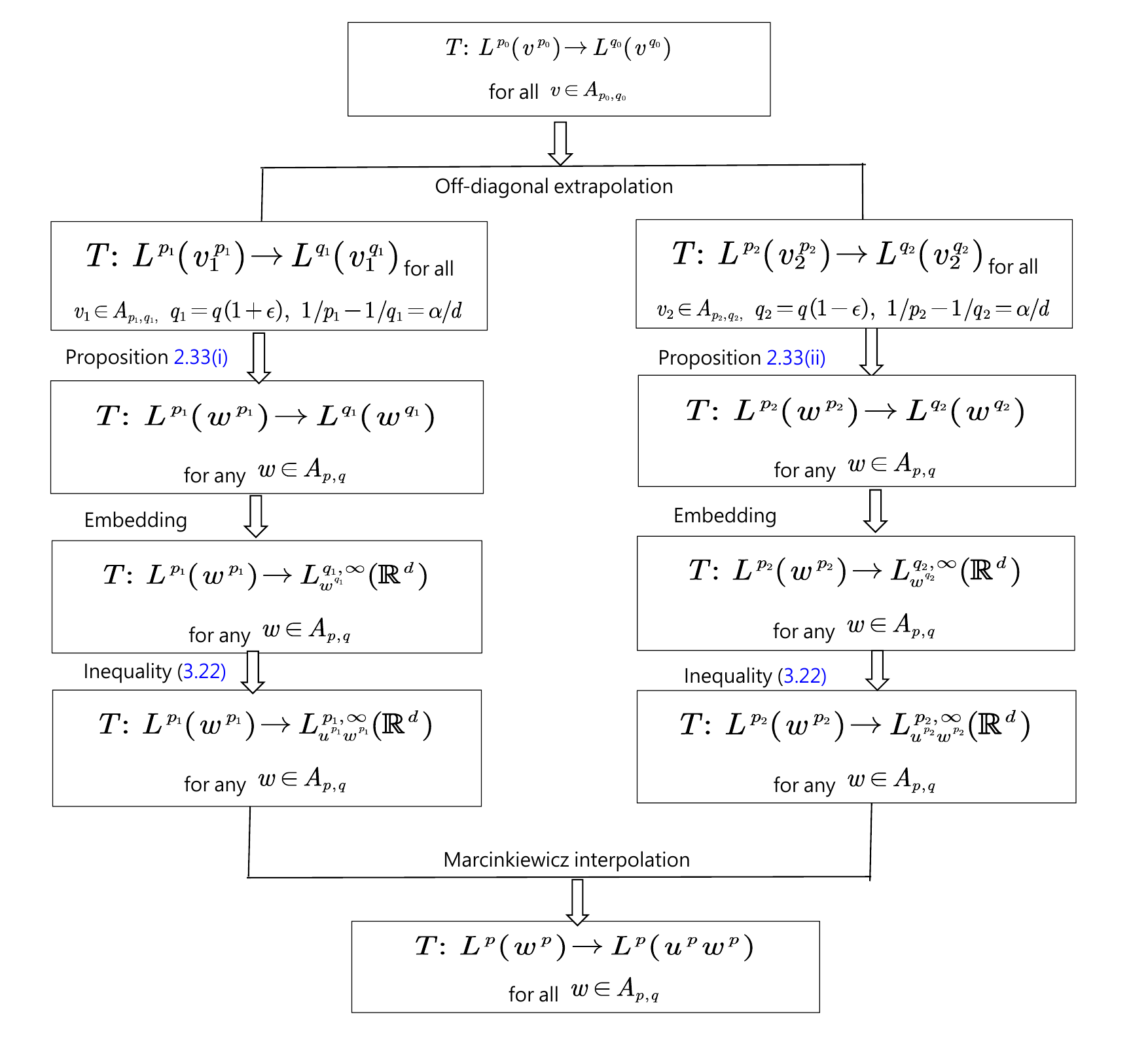}
        \caption{Proof of Theorem \ref{OD-D}}
        \label{fig:sub1}
    \end{subfigure}
    \label{fig:main}
\end{figure}

\subsection{Extrapolation for partial multilinear Muckenhoupt weights}

Given $0<q <\infty$, $\vec{p}=(p_{1},\ldots, p_{m})$ with $1\leq p_{1},\ldots, p_{m}\leq \infty$ and $\vec{r}=(r_{1},\ldots,r_{m+1})$ with $1\leq r_{1},\ldots,r_{m+1}\leq \infty$, when $\vec r\preceq (\vec p,q )$, we set
$$
\frac{1}{r}:=\sum_{i=1}^{m+1}\frac{1}{r_{i}},\qquad \frac{1}{p_{m+1}}:=1-\frac{1}{q }
\qquad\text{and} \qquad \frac{1}{\delta_{i}}=\frac{1}{r_{i}}-\frac{1}{p_{i}}, \quad i=1,\ldots,m+1.
$$
Note that $\frac{1}{r}>1+\frac{1}{p}-\frac{1}{q}>0$ and formally $\frac{1}{p_{m+1}}=\frac{1}{q'}$ could be negative or zero
for $q \leq 1$. A direct computation yields
$$
\sum_{i=1}^{m+1}\frac{1}{p_{i}}=1+\frac{1}{p}-\frac{1}{q} \qquad\text{and} \qquad \sum_{i=1}^{m+1}\frac{1}{\delta_{i}}=\frac{1}{r}-\frac{1}{p}-\frac{1}{q'}.
$$
In addition, $\vec r\preceq (\vec p,q )$ means $r_{i}\leq p_{i}$, $\delta_{i}\geq 0$, $1\leq i\leq m$,
$r_{m+1}<p_{m+1}$ and $\delta^{-1}_{m+1}>0$.
On the other hand, $\vec r\prec (\vec p,p_{0})$ implies $r_{i}<p_{i}$ or $\delta_{i}^{-1}>0$ for $1\leq i\leq m+1$.
In this case, $\vec w\in A_{(\vec p,q), \vec r}^{u}$ can be written as
\begin{equation}\label{YH-4}
[\vec w]_{A_{(\vec p,q), \vec r}}=\sup_{Q}\Big\langle w^{\delta_{m+1}} \Big\rangle^{\frac{1}{\delta_{m+1}}}_{Q}
\prod_{i=1}^{m}\ave{w_{i}^{-\delta_{i}}}^{\frac{1}{\delta_{i}}}_{Q}<\infty,
\end{equation}
when $p_{i}=r_{i}$ $(i.e. \ \delta_{i}^{-1}=0),$ the corresponding term in \eqref{YH-4} will
be replaced by $\esssup_{Q}w_{i}^{-1}$.

\begin{lem}\label{ex-lem-main}
Assume that $0<q<\infty$, $\vec{p}=(p_{1},\ldots,p_{m})$ with $1\leq p_{1},\ldots,p_{m}\leq \infty$,
$\vec{r}=(r_{1},\ldots,r_{m+1})$ with $1\leq r_{1},\ldots,r_{m+1}\leq \infty$
and $\vec{r}\preceq (\vec{p},q )$. Set
\begin{equation}\label{def-rho-kappa}
\frac{1}{\varrho}:=\frac{1}{r}-\frac{1}{r'_{m+1}}+\sum_{i=1}^{m-1}\frac{1}{p_{i}}=\frac{1}{\delta_{m}}+\frac{1}{\delta_{m+1}}>0, \qquad
\kappa :=\sum_{i=1}^{m+1}\frac{1}{\delta_{i}}=\frac{1}{r}-\frac{1}{p}-\frac{1}{q'}>0,
\end{equation}
and for each $1\leq i\leq m-1$,
$\frac{1}{\theta_{i}}:=\kappa -\frac{1}{\delta_{i}}=\Big(\sum_{j=1}^{m+1}\frac{1}{\delta_{j}}\Big)-\frac{1}{\delta_{i}}>0$.
Then it holds that
\begin{list}{$(\theenumi)$}{\usecounter{enumi}\leftmargin=.8cm
	\labelwidth=.8cm\itemsep=0.2cm\topsep=.1cm
	\renewcommand{\theenumi}{\roman{enumi}}}

\item Given  $\vec{w}=(w_1,\dots,w_m)\in A_{(\vec p,q), \vec r}$, write $w:=\prod_{i=1}^mw_i$,
\begin{equation}\label{formula-1}
\widehat{w}:=\big(\prod_{i=1}^{m-1}w_i\big)^{\varrho}
\qquad
\mbox{and}
\qquad
W
:=
w^{r_m} \widehat{w}^{- \frac{r_m}{\delta_{m+1}}}
=
w_m^{r_m}\widehat{w}^{\frac{r_m}{\delta_m}},
\end{equation}
one has
\begin{list}{$(\theenumi.\theenumii)$}{\usecounter{enumii}\leftmargin=.4cm
\labelwidth=.8cm\itemsep=0.2cm\topsep=.1cm
\renewcommand{\theenumii}{\arabic{enumii}}}

\item $w_i^{\theta_i}\in A_{\kappa \theta_i}$ and $\big[w_i^{\theta_i}\big]_{A_{\kappa \theta_i}}
\le [\vec w]_{A_{(\vec p,q), \vec r}}^{\theta_i}$ for $1\le i\le m-1$.

\item $\widehat{w}\in A_{\kappa \varrho}$ and
$[\widehat{w}]_{A_{\kappa \varrho}}\le [\vec w]_{A_{(\vec p,q), \vec r}}^{\varrho}$.

\item $W\in A_{\frac{p_m}{r_m},\frac{\delta_{m+1}}{r_m}}(\widehat{w})$
and $[W]_{A_{\frac{p_m}{r_nm},\frac{\delta_{m+1}}{r_m}} (\widehat{w})}
\le
[\vec w]_{A_{(\vec p,q), \vec r}}^{\delta_{m+1}}$.
\end{list}

\item Given $w_i^{\theta_i}\in A_{\kappa \theta_i}$ for $1\le i\le m-1$,
which leads to
\begin{equation}\label{formula2-1}
\widehat{w}=
\big(\prod_{i=1}^{m-1}w_i\big)^{\varrho}
\in A_{\kappa \varrho}
\end{equation}
and
$W\in A_{\frac{p_m}{r_m},\frac{\delta_{m+1}}{r_m}} (\widehat{w})$.  Set
\begin{equation}\label{formula-2}
w_m
:=
W^{\frac{1}{r_m}} \widehat{w}^{-\frac{1}{\delta_m}}.
\end{equation}
Then $\vec{w}=(w_1,\dots, w_m)\in A_{(\vec p,q), \vec r}$ and
$[\vec w]_{A_{(\vec p,q), \vec r}}
\le [W]_{A_{\frac{p_m}{r_m},\frac{\delta_{m+1}}{r_m}} (\widehat{w})}^{\frac1{\delta_{m+1}}}
[\widehat{w}]_{A_{\kappa \varrho}}^{\frac1 {\varrho}}
\prod_{i=1}^{m-1}\big[w_i^{\theta_i}\big]_{A_{\kappa \theta_i}}^{\frac1{\theta_i}}.$

\item For any measurable function $f\ge0 $ and in the context of $(i)$ or $(ii)$, there hold
\begin{equation}\label{LHS-rew}
\|f\|_{L^q(w^{q})}
=
\big\|\big(f\widehat{w}^{-\frac1{r_{m+1}'}}\big)^{r_m}\big\|_{L^\frac{q}{r_m}(W^{\frac{q}{r_m}} d\widehat{w})}^{\frac1{r_m}}
\end{equation}
and
\begin{equation}\label{RHS-rew}
\|f\|_{L^{p_m}(w_m^{p_{m}})}
=
\big\|\big(f\widehat{w}^{-\frac1{r_m}}\big)^{r_m}\big\|_{L^\frac{p_m}{r_m}(W^{\frac{p_m}{r_m}} d\widehat{w})}^{\frac1{r_m}}.
\end{equation}
\end{list}
\end{lem}
\begin{proof}
It follows from \eqref{formula-1}, \eqref{formula2-1} and \eqref{formula-2} that
\begin{equation}\label{qrfera}
\big\langle  W^{\frac{\delta_{m+1}}{r_m}} \big\rangle_{Q,\widehat{w}}
=
\big\langle  \widehat{w}\big\rangle_{Q}^{-1}\big\langle  w^{\delta_{m+1}}
\big\rangle_{Q}
\end{equation}
and $\big\langle   W^{-(\frac{p_m}{r_m})'}\big\rangle_{Q,\widehat{w}}
=\big\langle   w_m^{-r_m(\frac {p_m}{r_m})'}\widehat{w}^{\frac{r_m}{\delta_m}(\frac {p_m}{r_m})'}
\big\rangle_{Q,\widehat{w}}
=\big\langle  \widehat{w}\big\rangle_{Q} ^{-1}
\big\langle   w_m^{-\delta_m} \big\rangle_{Q}$ when $\delta_m^{-1}\neq 0$ (i.e., $r_m<p_m$).

This yields that for $\delta_m^{-1} \neq 0$,
\begin{equation}\label{aaa:1}
\big\langle  W^{\frac{\delta_{m+1}}{r_m}}   \big\rangle_{Q,\widehat{w}}^{\frac{r_m}{\delta_{m+1}}}
\big\langle   W^{-(\frac{p_m}{r_m})'}\big\rangle_{Q,\widehat{w}}^{\frac{1}{(\frac{p_m}{r_m})'}}
=\big\langle  \widehat{w} \big\rangle_{Q}^{-\frac{r_{m}}{\varrho}}
\big\langle  w^{\delta_{m+1}} \big\rangle_{Q} ^{\frac{r_m}{\delta_{m+1}}}
\big\langle   w_m^{-\delta_m} \big\rangle_{Q}^{\frac{r_m}{\delta_m}}.
\end{equation}
Thus, when $\delta_m^{-1}\neq0$  and $\sum_{i=1}^{m-1}\frac1{\delta_i}>0$, we have
\begin{equation}\label{aaa:0}
\begin{aligned}
&\big\langle  W^{\frac{\delta_{m+1}}{r_m}}   \big\rangle_{Q,\widehat{w}}^{\frac{r_m}{\delta_{m+1}}}
\big\langle   W^{-(\frac{p_m}{r_m})'}\big\rangle_{Q,\widehat{w}}^{\frac{1}{(\frac{p_m}{r_m})'}}
\\
&\quad =
\big[
\big\langle  \widehat{w}^{1-(\kappa \varrho)'}\big\rangle_{Q} ^{\frac1{\varrho}(\kappa \varrho-1)}
\big\langle  w^{\delta_{m+1}}  \big\rangle_{Q}  ^{\frac 1{\delta_{m+1}}}
\big\langle    w_m^{-\delta_m}\big\rangle_{Q}  ^{\frac1{\delta_m}}
\big\langle  \widehat{w}\big\rangle_{Q} ^{-\frac1{\varrho}}
\big\langle  \widehat{w}^{1-(\kappa \varrho)'}\big\rangle_{Q}^{-\frac1{\varrho}(\kappa \varrho-1)}
\big]^{r_{m}}.
\end{aligned}
\end{equation}
On the other hand, when $\sum_{i=1}^{m-1}\frac1{\delta_i}>0$ and $\delta_m^{-1}=0$,  due to $\varrho=\delta_{m+1}$,
then it holds that
\begin{equation}\label{aaa:2}
\begin{aligned}
&\big\langle  W^{\frac{\delta_{m+1}}{r_m}}   \big\rangle_{Q,\widehat{w}}^{\frac{r_m}{\delta_{m+1}}}
=\big\langle  \widehat{w}\big\rangle_{Q} ^{-\frac{r_{m}}{\varrho}}\big\langle  w^{\delta_{m+1}}
\big\rangle_{Q} ^{\frac{r_m}{\delta_{m+1}}}
\\
&\ =
\big[
\big\langle  \widehat{w}^{1-(\kappa \varrho)'}\big\rangle_{Q} ^{\frac1{\varrho}(\kappa \varrho-1)}
\big\langle  w^{\delta_{m+1}}  \big\rangle_{Q} ^{\frac 1{\delta_{m+1}}}
\big\langle  \widehat{w}\big\rangle_{Q} ^{-\frac1{\varrho}}
\big\langle  \widehat{w}^{1-(\kappa \varrho)'}\big\rangle_{Q} ^{-\frac1{\varrho}(\kappa \varrho-1)}
\big]^{r_{m}}.
\end{aligned}
\end{equation}

We now show part $(i)$. At first, we demonstrate $(i.1)$. Define
$$\mathcal{I} = \{j : 1 \leq j \leq m, \delta_j^{-1} \neq 0\}\quad \text{and}\quad \mathcal{I}' 
= \{1, \ldots, m\} \setminus \mathcal{I}.$$
Set $\frac{1}{\eta_i} = \frac{\theta_i}{\delta_{m+1}}$ for  $1 \leq i \leq m-1$
and $\frac{1}{\eta_j} = \frac{\theta_i}{\delta_j}$ for $j \in \mathcal{I}$ with $j \neq i$,
which yields $\frac{1}{\eta_i} + \sum_{j \neq i, j \in \mathcal{I}} \frac{1}{\eta_j} = 1.$
Then it follows from H\"{o}lder's inequality that
\begin{equation}\label{YH-5}
\begin{aligned}
\big\langle  w_i^{\theta_i} \big\rangle_{Q}
=
\big\langle
 w^{\theta_i}\prod_{\substack{1\le j \le m\\j\neq i} }w_j^{-\theta_i }\big\rangle_{Q}
&\le
\big\langle  w^{\theta_i\eta_i}\big\rangle_{Q}^{\frac{1}{\eta_i}}
\prod_{i\neq j\in\mathcal{I}}  \big\langle    w_j^{-\theta_i \eta_j}  \big\rangle_{Q}^{\frac{1}{\eta_j}}
\prod_{i\neq j \in\mathcal{I}'} \esssup_{Q} w_j^{-\theta_i }
\\
&=
\big\langle  w^{\delta_{m+1}}\big\rangle_{Q}^{\frac{\theta_i}{\delta_{m+1}}}
\prod_{i\neq j\in\mathcal{I}} \big\langle    w_j^{-\delta_j}  \big\rangle_{Q}^{\frac{\theta_i}{\delta_j}}
\big(\prod_{i\neq j \in\mathcal{I}'} \esssup_{Q} w_j^{-1} \big)^{\theta_i}.
\end{aligned}
\end{equation}
When $p_i = r_i$, one has $\theta_i = \kappa$. Then \eqref{YH-5} implies that $w_i^{\theta_i} \in A_1$ and
$\big[w_i^{\theta_i}\big]_{A_1} \leq [\vec{w}]_{A_{(\vec{p}, q), \vec{r}}}^{\theta_i}.$
When $p_i > r_i$, due to $\theta_i(1 - (\kappa \theta_i)') = -\delta_i$, then it holds that
$w_i^{\theta_i} \in A_{\kappa \theta_i}$ and $\big[w_i^{\theta_i}\big]_{A_{\kappa \theta_i}} 
\leq [\vec{w}]_{A_{(\vec{p}, q), \vec{r}}}^{\theta_i}.$
\medskip

To obtain $(i.2)$,  the following three cases are investigated.

\vskip 0.2 true cm
\noindent\textbf{Case 1. $\sum_{i=1}^{m}\frac1{\delta_i}=0$}

\vskip 0.1 true cm
In this case, one has $p_j=r_j$ for $1\le j\le m$
and $\frac1{\varrho}=\frac1{\delta_{m+1}}=\kappa$.
Then $\widehat{w}\in A_1$ and
\begin{equation*}
\begin{aligned}
\big\langle  \widehat{w} \big\rangle_{Q}
&=
\big\langle  w^{\delta_{m+1}}w_m^{-\delta_{m+1}} \big\rangle_{Q}
\le
\big\langle  w^{\delta_{m+1}}\big\rangle_{Q}  \esssup_Q w_m^{-\delta_{m+1}}
\\
&\le
[w]_{A_{\vec{p},\vec{r}, q }}^{\delta_{m+1}}\,\big(\prod_{i=1}^{m-1}\essinf_{Q} w_i\big)^{\delta_{m+1}}
\le
[w]_{A_{\vec{p},\vec{r}, q }}^{\varrho}\, \essinf_{Q}\widehat{w}.
\end{aligned}
\end{equation*}

\vskip 0.1 true cm

\noindent\textbf{Case 2. $\sum_{i=1}^{m-1}\frac1{\delta_i}=0$}

\vskip 0.1 true cm

In this case, it holds that $p_j=r_j$ for $1\le j\le m-1$, $\delta_m^{-1}\neq 0$
and $\frac{1}{\varrho} = \frac{1}{\delta_{m+1}} + \frac{1}{\delta_m} = \kappa.$
We now show $\widehat{w} \in A_1$. In fact, it follows from H\"{o}lder's inequality with $\frac{\delta_{m+1}}{\varrho} = 1 + \frac{\delta_{m+1}}{\delta_m} > 1$ that
\begin{equation*}
\begin{aligned}
\big\langle  \widehat{w} \big\rangle_{Q}
&=
\big\langle  w^{\varrho}w_m^{-\varrho } \big\rangle_{Q}
\le
\big\langle  w^{\delta_{m+1}}\big\rangle_{Q} ^{\frac{\varrho}{\delta_{m+1}}}
\big\langle  w_m^{-\delta_m}\big\rangle_{Q}  ^{\frac{\varrho}{\delta_m}}
\\
&\le
[w]_{A_{\vec{p},\vec{r}, q }}^{\varrho}\,\big(\prod_{i=1}^{m-1}\essinf_{Q} w_i\big)^{\varrho}
\le
[w]_{A_{\vec{p},\vec{r}, q }}^{\varrho}\, \essinf_{Q}\widehat{w}.
\end{aligned}
\end{equation*}

\vskip 0.1 true cm

\noindent\textbf{Case 3. $\sum_{i=1}^{m-1}\frac1{\delta_i}>0$}

\vskip 0.1 true cm

At this time, it holds that $\kappa=\sum_{i=1}^{m+1}\frac1{\delta_i}
=\sum_{i=1}^{m-1}\frac1{\delta_i}+\frac1{\varrho}>\frac1{\varrho}.$
Let
$\mathcal{I}=\{i: 1\le i\le m-1, \delta_i^{-1}\neq 0\}\neq\emptyset$
and $\mathcal{I}'=\{1,\dots,m-1\}\setminus\mathcal{I}$.
For $i\in\mathcal{I}$, set
\[\frac1{\eta_i}
:=\frac1{\delta_i}
\big(\sum_{j=1}^{m-1} \frac1{\delta_j}\big)^{-1}
=\frac1{\delta_i}
\big(\sum_{j=1}^{m+1} \frac1{\delta_j}-\frac1{\delta_m}-\frac1{\delta_{m+1}}\big)^{-1}
=\frac1{\delta_i}\big(\kappa -\frac1{\varrho}\big)^{-1}.
\]
Due to
$\sum_{i\in\mathcal{I}}\frac1{\eta_i}=1$, then H\"{o}lder's inequality leads to
\begin{multline}\label{q35t5}
\big\langle  \widehat{w}^{1-(\kappa \varrho)'}\big\rangle_{Q} ^{\kappa \varrho-1}
=
\big\langle  \prod_{i\in\mathcal{I}} w_i^{-\frac{\delta_i}{\eta_i}} \prod_{i\in\mathcal{I}'} w_i^{-\varrho((\kappa \varrho)'-1)} \big\rangle_{Q} ^{\kappa \varrho-1}
\le
\prod_{i\in\mathcal{I}} \big\langle  w_i^{-\delta_i} \big\rangle_{Q} ^{\frac{\varrho}{\delta_i}}
\big(\prod_{i\in\mathcal{I}'} \esssup_{Q} w_i^{-1} \big)^{\varrho}.
\end{multline}
On the other hand, if $\delta_m^{-1}\neq 0$, then we use H\"{o}lder's inequality with
$\frac{\delta_{m+1}}{\varrho} =1+\frac{\delta_{m+1}}{\delta_m}>1$ to obtain
\begin{equation}\label{qt53ae}
\big\langle  \widehat{w} \big\rangle_{Q}
=
\big\langle  w^{\varrho}w_m^{-\varrho } \big\rangle_{Q}
\le
\big\langle  w^{\delta_{m+1}}\big\rangle_{Q}  ^{\frac{\varrho}{\delta_{m+1}}}
\big\langle  w_m^{-\delta_m}\big\rangle_{Q} ^{\frac{\varrho}{\delta_m}}.
\end{equation}
If  $\delta_m^{-1}=0$, then $\varrho=\delta_{m+1}$ and
\begin{equation}\label{sgtgt}
\big\langle  \widehat{w} \big\rangle_{Q}
=
\big\langle  w^{\varrho}w_m^{-\varrho} \big\rangle_{Q}
\le
\big\langle  w^{\delta_{m+1}}\big\rangle_{Q} ^{\frac{\varrho}{\delta_{m+1}}}
\big(\esssup_{Q} w_m^{-1}\big)^{\varrho}.
\end{equation}
Combining \eqref{q35t5} with either \eqref{qt53ae} or \eqref{sgtgt} yields $\widehat{w} \in A_{\kappa \varrho}$ and
$[\widehat{w}]_{A_{\kappa \varrho}} \leq [\vec{w}]_{A_{(\vec{p}, q), \vec{r}}}^{\varrho}.$
This completes the proof of $(i.2)$.

\vskip 0.2 true cm

To show $(i.3)$, we discuss the following three cases.

\vskip 0.2 true cm

\noindent\textbf{Case 1. $\sum_{i=1}^{m}\frac1{\delta_i}=0$}

\vskip 0.1 true cm
In this case, we have $p_j=r_j$ for $1\le j\le m$,
$\varrho = \delta_{m+1}$ and $W = w_m^{\frac{r_m}{p_m}} = w_m$ by $p_m = r_m$. Note that
\begin{equation}\label{holder-trivial-case1}
1
=
\big\langle  \widehat{w}\widehat{w}^{-1}\big\rangle_{Q}
\le
\big\langle  \widehat{w}\big\rangle_{Q}  \esssup_{Q}\widehat{w}^{-1}
\le
\big\langle  \widehat{w}\big\rangle_{Q}
\prod_{i=1}^{m-1} \esssup_{Q} w_i^{-\varrho}.
\end{equation}
This and \eqref{qrfera} imply
\begin{multline*}
\big\langle  W^{\frac{\delta_{m+1}}{r_m}} \big\rangle_{Q,\widehat{w}}
=
\big\langle  \widehat{w}\big\rangle_{Q} ^{-1}\big\langle  w^{\delta_{m+1}}
\big\rangle_{Q}
\le
[\vec w]_{A_{(\vec p,q), \vec r}}^{\delta_{m+1}}
\big(\prod_{i=1}^{m-1} \esssup_{Q} w_i^{-\delta_{m+1}}\big)
\big(\prod_{i=1}^m \esssup_{Q} w_i^{-1}\big)^{-\delta_{m+1}}
\\
=
[\vec w]_{A_{(\vec p,q), \vec r}}^{\delta_{m+1}}
\big(\esssup_{Q} w_m^{-1}\big)^{-\delta_{m+1}}
=
[\vec w]_{A_{(\vec p,q), \vec r}}^{\delta_{m+1}}
\big(\esssup_{Q} W^{-\frac{\delta_{m+1}}{r_m}}  \big)^{-1}.
\end{multline*}
Thus, $W \in A_{1,\frac{\delta_{m+1}}{r_m}}(\widehat{w})$
and $[W]_{A_{1,\frac{\delta_{m+1}}{r_m}}(\widehat{w})} \leq [\vec{w}]_{A_{(\vec{p}, q), \vec{r}}}^{\delta_{m+1}}.$

\vskip 0.2 true cm
\noindent\textbf{Case 2. $\sum_{i=1}^{m-1}\frac1{\delta_i}=0$}

\vskip 0.1 true cm

In this case, it holds $p_j=r_j$ for $1\le j\le m-1$) and $\delta_m^{-1}\neq 0$.
By \eqref{aaa:1} and \eqref{holder-trivial-case1}, we have
\begin{align*}
\big\langle  W^{\frac{\delta_{m+1}}{r_m}}& \big\rangle_{Q,\widehat{w}}
\big\langle   W^{-(\frac{p_m}{r_m})'}\big\rangle_{Q,\widehat{w}}^{\frac{\frac{\delta_{m+1}}{r_m}}{(\frac{p_m}{r_m})'}}
=
\big\langle  \widehat{w} \big\rangle_{Q} ^{-1-\frac{\delta_{m+1}}{\delta_m}}
\big\langle  w^{\delta_{m+1}} \big\rangle_{Q}
\big\langle   w_m^{-\delta_m} \big\rangle_{Q} ^{\frac{\delta_{m+1}}{\delta_m}}
\\
&
\qquad\le
\big\langle  w^{\delta_{m+1}} \big\rangle_{Q}
\big\langle   w_m^{-\delta_m} \big\rangle_{Q} ^{\frac{\delta_{m+1}}{\delta_m}}
\big(\prod_{i=1}^{m-1} \esssup_{Q} w_i^{-1}\big)^{\delta_{m+1}}
\le
[\vec w]_{A_{(\vec p,q), \vec r}}^{\delta_{m+1}}.
\end{align*}

\vskip 0.2 true cm

\noindent\textbf{Case 3. $\sum_{i=1}^{m-1}\frac1{\delta_i}>0$}

\vskip 0.1 true cm

At this time, one has $\kappa \varrho > 1$. Therefore, applying H\"{o}lder's inequality derives
\begin{align*}
1=
\big\langle \widehat{w}^{\frac{\varrho}{\kappa }} \widehat{w}^{-\frac{\varrho}{\kappa }}\big\rangle_{Q}^{\kappa \varrho}
	\le
	\big\langle \widehat{w}\big\rangle_{Q}
\big\langle  \widehat{w}^{1-(\kappa \varrho)'}\big\rangle_{Q} ^{\kappa \varrho-1}.
\end{align*}
This, together with \eqref{aaa:0} and \eqref{q35t5}, yields that  when  $\delta_m^{-1}\neq 0$ (that is $r_m<p_m$),
\begin{align*}
&
\big[\big\langle  W^{\frac{\delta_{m+1}}{r_m}}   \big\rangle_{Q,\widehat{w}}
\big\langle   W^{-(\frac{p_m}{r_m})'}\big\rangle_{Q,\widehat{w}}
^{\frac{\frac{\delta_{m+1}}{r_m}}{(\frac{p_m}{r_m})'}}\big]^{\frac1{\delta_{m+1}}}
\le
\big\langle  \widehat{w}^{1-\big(\kappa \varrho\big)'}\big\rangle_{Q} ^{\frac1{\varrho}(\kappa \varrho-1)}
\big\langle  w^{\delta_{m+1}}  \big\rangle_{Q}  ^{\frac 1{\delta_{m+1}}}
\big\langle    w_m^{-\delta_m}\big\rangle_{Q}  ^{\frac1{\delta_m}}
\\
&\le
\big\langle  w^{\delta_{m+1}}  \big\rangle_{Q}  ^{\frac 1{\delta_{m+1}}}
\big\langle    w_m^{-\delta_m}\big\rangle_{Q}  ^{\frac1{\delta_m}}
\big(\prod_{i\in\mathcal{I}}\! \big\langle  w_i^{-\delta_i} \big\rangle_{Q} ^{\frac{1}{\delta_i}}\big)
\big(\!\prod_{i\in\mathcal{I}'}\! \esssup_{Q} w_i^{-1} \big)\!
\\
&\le
[\vec w]_{A_{(\vec p,q), \vec r}}.
\end{align*}
Thus, $W\in A_{\frac{p_m}{r_m},\frac{\delta_{m+1}}{r_m}}(\widehat{w})$
and $[W]_{A_{\frac{p_m}{r_m},\frac{\delta_{m+1}}{r_m}} (\widehat{w})}
\le [\vec w]_{A_{(\vec p,q), \vec r}}^{\delta_{m+1}}$.

On the other hand, when $\delta_m^{-1}=0$, i.e., $r_m=p_m$,  due to $\varrho=\delta_{m+1}$ and $W=w_m^{\frac{r_m}{p_m}}=w_m$,
then collecting \eqref{aaa:2} and \eqref{q35t5} yields
\begin{multline*}
\big\langle  W^{\frac{\delta_{m+1}}{r_m}} \big\rangle_{Q,\widehat{w}}
\le
\big[
\big\langle  \widehat{w}^{1-(\kappa \varrho)'}\big\rangle_{Q} ^{\frac1{\varrho}(\kappa \varrho-1)}
\big\langle  w^{\frac{\delta_{m+1}}{p}}  \big\rangle_{Q}  ^{\frac 1{\delta_{m+1}}}
\big]^{\delta_{m+1}}
\\
\le
\big[
\big\langle  w^{\delta_{m+1}}  \big\rangle_{Q}  ^{\frac 1{\delta_{m+1}}}
\big(\prod_{i\in\mathcal{I}} \big\langle  w_i^{-\delta_i} \big\rangle_{Q} ^{\frac{1}{\delta_i}}\big)
\big(\prod_{i\in\mathcal{I}'} \esssup_{Q} w_i^{-1} \big)
\big]^{\delta_{m+1}}
\\
\le
[\vec w]_{A_{(\vec p,q), \vec r}}^{\delta_{m+1}}\essinf_{Q} w_m^{\delta_{m+1}}
=
[\vec w]_{A_{(\vec p,q), \vec r}}^{\delta_{m+1}}\essinf_{Q}  W^{\frac{\delta_{m+1}}{r_m}}.
\end{multline*}
This completes the proof of $(i.3)$ and hence that of $(i)$.

\vskip 0.1 true cm

We now turn our attention to $(ii)$. For $w_i^{\theta_i}\in A_{\kappa \theta_i}$ ($1\le i\le m-1$),
by \eqref{formula2-1}, one has $\widehat{w}\in A_{\kappa \varrho}$ and $W\in A_{\frac{p_m}{r_m},\frac{\delta_{m+1}}{r_m}} (\widehat{w})$.
Let $w_m$ be as in \eqref{formula-2}. Next we show
$\vec{w}\in A_{(\vec p,q), \vec r}$.

\vskip 0.2 true cm

\noindent\textbf{Case 1. $\sum_{i=1}^{m}\frac1{\delta_i}=0$}
\vskip 0.1 true cm

In this case,  one has $p_j=r_j$ for $1\le j\le m$.
Note that $\theta_i=\frac 1{\kappa }$ holds for $1\le i\le m-1$. This and H\"{o}lder's inequality yield
\begin{equation}\label{6awf}
\begin{aligned}
\essinf_{Q} \Big(\prod_{i=1}^{m-1} w_i\Big)
&\le
\big(\dashint_{Q} \prod_{i=1}^{m-1} w_i^{\frac{\theta_i}{m-1}}dx\big)^{(m-1)\kappa }
\le
\prod_{i=1}^{m-1} \big(\dashint_{Q} w_i^{\theta_i}dx\big)^{\frac1{\theta_i}}\le
\prod_{i=1}^{m-1} \big[w_i^{\theta_i}\big]_{ A_1}^{\frac1{\theta_i}}
\essinf_{Q} w_i,
\end{aligned}
\end{equation}
where $\kappa \theta_i=1$ is used in the last estimate. It follows from \eqref{6awf}
and \eqref{qrfera} that
\begin{align*}
\big\langle  w^{\delta_{m+1}}  \big\rangle_{Q}  ^{\frac 1{\delta_{m+1}}}
&=
\big\langle  W^{\frac{\delta_{m+1}}{r_m}} \big\rangle_{Q,\widehat{w}}^{\frac 1{\delta_{m+1}}}
\big\langle  \widehat{w}\big\rangle_{Q} ^{\frac 1{\delta_{m+1}}}
\le
[W]_{A_{1,\frac{\delta_{m+1}}{r_m}} (\widehat{w})}^{\frac1{\delta_{m+1}}}
[\widehat{w}]_{A_1}^{^{\frac1{\delta_{m+1}}}}
\essinf_Q W^{\frac{1}{r_m}}
\essinf_Q \widehat{w}^{\frac{1}{\delta_{m+1}}}
\\
&=
[W]_{A_{1,\frac{\delta_{m+1}}{r_m}} (\widehat{w})}^{\frac1{\delta_{m+1}}}
[\widehat{w}]_{A_1}^{^{\frac1{\delta_{m+1}}}}
\essinf_Q w_m
\essinf_{Q} \big(\prod_{i=1}^{m-1} w_i\big)
\\
&\le
[W]_{A_{1,\frac{\delta_{m+1}}{r_m}} (\widehat{w})}^{\frac1{\delta_{m+1}}}
[\widehat{w}]_{A_1}^{^{\frac1{\delta_{m+1}}}}
\big(
\prod_{i=1}^{m-1} \big[w_i^{\theta_i }\big]_{ A_1}^{\frac1{\theta_i}}
\big)
\big(\prod_{i=1}^{m} \essinf_{Q} w_i\big),
\end{align*}
where $p_m=r_m$, $w_m=W^{\frac{1}{r_m}}=W$ and $\varrho=\delta_{m+1}$ are used.

\vskip 0.2 true cm

\noindent\textbf{Case 2. $\sum_{i=1}^{m-1}\frac1{\delta_i}=0$}
\vskip 0.1 true cm

This means  $p_j=r_j$ for $1\le j\le m-1$) and $\delta_m^{-1}\neq 0$.
It follows from \eqref{aaa:1} and \eqref{6awf} that
\begin{align*}
&\big\langle  w^{\delta_{m+1}} \big\rangle_{Q} ^{\frac1{\delta_{m+1}}}
\big\langle   w_m^{-\delta_m} \big\rangle_{Q} ^{\frac1{\delta_m}}
=\big\langle  W^{\frac{\delta_{m+1}}{r_m}}   \big\rangle_{Q,\widehat{w}}^{\frac1{\delta_{m+1}}}
\big\langle   W^{-(\frac{p_m}{r_m})'}\big\rangle_{Q,\widehat{w}}^{\frac{1}{r_m(\frac{p_m}{r_m})'}}
\big\langle  \widehat{w} \big\rangle_{Q} ^{\frac1{\varrho}}
\\
&\qquad\le
[W]_{A_{\frac{p_m}{r_m},\frac{\delta_{m+1}}{r_m}} (\widehat{w})}^{\frac1{\delta_{m+1}}}
[\widehat{w}]_{A_1}^{\frac1 {\varrho}}\essinf_{Q} \big(\prod_{i=1}^{m-1} w_i\big)
\\
&\qquad\le
[W]_{A_{\frac{p_m}{r_m},\frac{\delta_{m+1}}{r_m}} (\widehat{w})}^{\frac1{\delta_{m+1}}}
[\widehat{w}]_{A_1}^{^{\frac1{\delta_{m+1}}}}
\big(
\prod_{i=1}^{m-1} \Big[w_i^{\theta_i}\big]_{ A_1}^{\frac1{\theta_i}}
\big)
\big(\prod_{i=1}^{m-1} \essinf_{Q} w_i\big),
\end{align*}
which yields $\vec{w}\in A_{\vec{p},\vec{r}, q}$.

\vskip 0.2 true cm

\noindent\textbf{Case 3. $\sum_{i=1}^{m-1}\frac1{\delta_i}>0$}

\vskip 0.1 true cm

Set $\eta_i=\kappa (m-1)\theta_i$ for $1\le i\le m-1$ and $\eta_m=(m-1)(\kappa \varrho)'$.
Note that
\begin{align*}
\sum_{i=1}^m\frac1{\eta_i}
&=
\frac{1}{(m-1)\kappa }\sum_{i=1}^{m-1}\frac1{\theta_i}+ \frac1{(m-1)(\kappa \varrho)'}
\\
&=
1-\frac{1}{(m-1)\kappa }\big(\kappa -\frac1{\varrho}\big)
+
\frac1{m-1}-\frac1{m-1}\frac{1}{\kappa \varrho}
=1.
\end{align*}
Thus, it follows from  Holder's inequality with the exponents $\eta_i$ for $1\le i\le m$ that
\begin{align*}
1&=
\big\langle  \widehat{w}^{-\frac{1}{(m-1)\kappa \varrho}} \widehat{w}^{\frac{1}{(m-1)\kappa \varrho}}\big\rangle_{Q}^{\kappa (m-1)}\\
&=
\big\langle  \widehat{w}^{(1-(\kappa \varrho)')\frac1{\eta_m}} \prod_{i=1}^{m-1} w_i^{\theta_i\frac1{\eta_i}}\big\rangle_{Q}^{\kappa (m-1)}
\le
\big\langle  \widehat{w}^{1-(\kappa \varrho)'}\big\rangle_{Q}^{\frac1{\varrho}(\kappa \varrho-1)}
\prod_{i=1}^{m-1} \big\langle  w_i^{\theta_i}\big\rangle_{Q}^{\frac1{\theta_i}}.
\end{align*}
On the other hand, if we define $\mathcal{I} = \{j : 1 \leq j \leq m-1, \delta_j^{-1} \neq 0\} \neq \emptyset$
and $\mathcal{I}' = \{1, \dots, m-1\} \setminus \mathcal{I}$, due to
$\frac{\theta_i}{p_i} \left( (\kappa \theta_i)' - 1 \right) = \frac{\delta_i}{p_i}$, then
$$
\big(\prod_{i\in\mathcal{I}}  \big\langle  w_i^{-\delta_i} \big\rangle_{Q} ^{\frac{1}{\delta_i}}\big)
\big(\prod_{i\in\mathcal{I}'}  \esssup_{Q} w_i^{-1}\big)
\le
\prod_{i=1}^{m-1}  \big[w_i^{\theta_i}\big]_{A_{\kappa \theta_i}}^{\frac1{\theta_i}} \big\langle  w_i^{\theta_i} \big\rangle_{Q}^{-\frac1{\theta_i}}
\le
\big(\prod_{i=1}^{m-1}  \big[w_i^{\theta_i}\big]_{A_{\kappa \theta_i}}^{\frac1{\theta_i}} \big)
\big\langle  \widehat{w}^{1-(\kappa \varrho)'}\big\rangle_{Q}^{\frac1{\varrho}(\kappa \varrho-1)}.
$$
This and \eqref{aaa:0} derive that when $\frac1{\delta_m}\neq 0$,
\begin{align*}
&\big\langle  w^{\delta_{m+1}}  \big\rangle_{Q} ^{\frac 1{\delta_{m+1}}}
\big\langle    w_m^{-\delta_m}\big\rangle_{Q} ^{\frac1{\delta_m}}
\big(\prod_{i\in\mathcal{I}}  \big\langle  w_i^{-\delta_i} \big\rangle_{Q} ^{\frac{1}{\delta_i}}\big)
\big(\prod_{i\in\mathcal{I}'}  \esssup_{Q} w_i^{-1}\big)
\\
&\qquad\le
\big(\prod_{i=1}^{m-1}  \big[w_i^{\theta_i}\big]_{A_{\kappa \theta_i}}^{\frac1{\theta_i}} \big)
\big[\big\langle  W^{\frac{\delta_{m+1}}{r_m}}   \big\rangle_{Q,\widehat{w}}
\big\langle   W^{-(\frac{p_m}{r_m})'}\big\rangle_{Q,\widehat{w}}^{\frac{\frac{\delta_{m+1}}{r_m}}{(\frac{p_m}{r_m})'}}\big]^{\frac1{\delta_{m+1}}}
\big\langle  \widehat{w}\big\rangle_{Q} ^{\frac1{\varrho}}
\big\langle  \widehat{w}^{1-(\kappa \varrho)'}\big\rangle_{Q}^{\frac1{\varrho}(\kappa \varrho-1)}.
\\
&\qquad\le
[W]_{A_{\frac{p_m}{r_m},\frac{\delta_{m+1}}{r_m}} (\widehat{w})}^{\frac1{\delta_{m+1}}}
[\widehat{w}]_{A_{\kappa \varrho}}^{\frac1 {\varrho}}
\prod_{i=1}^{m-1}\big[w_i^{\frac {\theta_i}{ p_i}}\big]_{A_{\kappa \theta_i}}^{\frac1{\theta_i}}.
\end{align*}
When $\frac{1}{\delta_m} = 0$, by $p_m=r_m$, $\varrho=\delta_{m+1}$, $W=w_m^{\frac{r_m}{p_m}}=w_m$  and \eqref{aaa:2},
we have
\begin{align*}
&\big\langle  w^{\delta_{m+1}}  \big\rangle_{Q}  ^{\frac 1{\delta_{m+1}}}
\big( \esssup_{Q} w_m^{-1}\big)
\big(\prod_{i\in\mathcal{I}}  \big\langle  w_i^{-\delta_i} \big\rangle_{Q} ^{\frac{1}{\delta_i}}\big)
\big(\prod_{i\in\mathcal{I}'}  \esssup_{Q} w_i^{-1}\big)
\\
&\qquad\le
\big(\prod_{i=1}^{m-1}  \big[w_i^{\theta_i}\big]_{A_{\kappa \theta_i}}^{\frac1{\theta_i}} \big)
\big[\big\langle  W^{\frac{\delta_{m+1}}{r_m}}   \big\rangle_{Q,\widehat{w}} \big( \esssup_{Q} W^{-\frac{\delta_{m+1}}{r_m}}\big)
\big]^{\frac1{\delta_{m+1}}}
\big\langle  \widehat{w}\big\rangle_{Q} ^{\frac1{\varrho}}
\big\langle  \widehat{w}^{1-(\kappa \varrho)'}\big\rangle_{Q}^{\frac1{\varrho}(\kappa \varrho-1)}
\\
&\qquad\le
[W]_{A_{1,\frac{\delta_{m+1}}{r_m}} (\widehat{w})}^{\frac1{\delta_{m+1}}}
[\widehat{w}]_{A_{\kappa \varrho}}^{\frac1 {\varrho}}
\prod_{i=1}^{m-1}\big[w_i^{\theta_i}\big]_{A_{\kappa \theta_i}}^{\frac1{\theta_i}}.
\end{align*}
This completes the proof of $(ii)$.

\medskip

Finally, \eqref{LHS-rew} and \eqref{RHS-rew} follow directly from the definition of $\varrho$ and either \eqref{formula-1} for case $(i)$ or \eqref{formula-2} for case $(ii)$.
\end{proof}

\noindent
{\bf Proof of Theorem \ref{ex-Appqu}.}
For simplicity and without loss of generality, we fix $p_i$ for $1\le i\le m-1$ and vary $p_m$.
Otherwise, we can rearrange the $f_i$'s. Fix $0<q^*<\infty$ and $\vec p^*=(p^*_1,\dots, p^*_{m-1}, p^*_m)$
with  $\vec{r}\preceq(\vec{p^*},q^*)$, $r_m<p^*_m$ and $p^*_i=p_i$ ($1\le i\le m-1$).
Assume $\vec v \in A^{u}_{(\vec p^*, q^*), \vec{r}}$. Set $\frac1{p^*}:=\sum_{i=1}^{m}\frac1{p^{*}_i}$  and $v:=\prod_{i=1}^mv_i$.
Define
\[
\frac1r:=\sum_{i=1}^{m+1} \frac1{r_i};
\quad
\frac1{p_{m+1}}:=1-\frac1p;
\quad
\frac1{p^*_{m+1}}:=1-\frac1{p^*};
\quad
\kappa:=\frac{1}{r}-\frac{1}{p}-\frac{1}{q'}=\frac{1}{r}-\frac{1}{p^*}-\frac{1}{(q^*)'},
\]
and for $i=1,\dots, m+1,$
$\frac1{\delta_i}=\frac1{r_i}-\frac1{p_i}$,
$\frac1{\widetilde{\delta}_i}=\frac1{r_i}-\frac1{p^*_i }$.

Observe that $\delta_i=\widetilde{\delta}_i$ for $1\le i\le m-1$. This means that in view of \eqref{def-rho-kappa},
\[
\frac1{\varrho}
:=
\frac 1{r_m}-\frac 1{r_{m+1}'}+\sum_{i=1}^{m-1}\frac1{p_i}
=
\frac 1{r_m}-\frac 1{r_{m+1}'}+\sum_{i=1}^{m-1}\frac1{p^*_i }
=
\frac1{\delta_m}+\frac1{\delta_{m+1}}
=
\frac1{\widetilde{\delta}_m}+\frac1{\widetilde{\delta}_{m+1}}.
\]
For $1\le i \le m-1$, set $w_i:=v_i$. Then by applying Lemma \ref{ex-lem-main} $(i)$ to $\vec v \in A^{u}_{(\vec p^*, q^*),\vec{r}}$ and  $(i.1)$,
we have that for $1\le i\le m-1$,
$w_i^{\theta_i}=w_i^{\theta_i}\in A_{\kappa \theta_i }$ with $\frac1{\theta_i}
:=\kappa-\frac1{\widetilde{\delta}_i}$ holds;
while $(i.2)$ gives $\widehat{w}:=\big(\prod_{i=1}^{m-1}w_i\big)^{\varrho}\in A_{\kappa \varrho}$;
and finally $(i.3)$ implies
\begin{equation}\label{aferer}
V
:=
v^{r_m} \widehat{w}^{- \frac{r_m}{\widetilde{\delta}_{m+1}}}
\in
A^{u}_{\frac{p^*_m}{r_m},\frac{\widetilde{\delta}_{m+1}}{r_m}}(\widehat{w}).
\end{equation}
In particular, $\widehat{w}\in A_\infty$ holds.

Let $W\in A^{u}_{\frac{p_m}{r_m},\frac{\delta_{m+1}}{r_m}}(\widehat{w})$ and $w_m
:=W^{\frac{1}{r_m}} \widehat{w}^{\frac{1}{\delta_m}}$ as in \eqref{formula-2}.
Since $w_i^{\theta_i }\in A_{\kappa \theta_i }$ for $1\le i\le m-1$ and $\widehat{w}\in A_{\kappa \varrho}$, one can 
apply Lemma \ref{ex-lem-main} $(ii)$ with $\vec{p}$ and $\vec{r}$ to derive $\vec{w}=(w_1,\dots, w_m)\in A^{u}_{(\vec p, q), \vec{r}}$
(notice that $\varrho$ is fixed and depends on  $p_i=p^*_i $ for $1\le i \le m-1$, $r_m$ and $r_{m+1}$).
By Lemma \ref{ex-lem-main} $(iii)$, we see that for $(f,f_1,\dots,f_m)\in\mathcal{F}$,
\begin{multline}\label{faefe}
\big\|\big(f\widehat{w}^{-\frac1{r_{m+1}'}}\big)^{r_m}\big\|_{L^\frac{q}{r_m}(W^{\frac{q}{r_m}}\d\widehat{w})}^{\frac1{r_m}}
=
\|f\|_{L^{q}(w^{q})}
\lesssim
\prod_{i=1}^m\|f_i\|_{L^{p_i}(w_i^{p_{i}})}
\\
=
\big(\prod_{i=1}^{m-1}\|f_i\|_{L^{p_i}(w_i^{p_{i}})}\big)
\big\|\big(f_m\widehat{w}^{-\frac1{r_m}}\big)^{r_m}\big\|_{L^\frac{p_m}{r_m}(W^{\frac{p_m}{r_m}}\d\widehat{w})}^{\frac1{r_m}}.
\end{multline}
Let
\[
\mathcal{G}
:=
\big\{
\big(\big(f\widehat{w}^{-\frac1{r_{m+1}'}}\big)^{r_m}, \big(\prod_{i=1}^{m-1}\|f_i\|_{L^{p_i}(w_i^{p_{i}})}f_m\widehat{w}^{-\frac1{r_m}}\big)^{r_m}\big):
(f,f_1,\dots,f_m)\in\mathcal{F}
\big\}.
\]
Then \eqref{faefe} can be written as
\[
\|F\|_{L^\frac{q}{r_m}(W^{\frac{q}{r_m}}\d\widehat{w})}
\lesssim
\|G \|_{L^\frac{p_m}{r_m}(W^{\frac{p_m}{r_m}}\d\widehat{w})},
\qquad
(F,G)\in\mathcal{G},
\]
which holds for $W \in A^{u}_{\frac{p_m}{r_m}, \frac{\delta_{m+1}}{r_m}}(\widehat{w})$. Consequently, by Theorem \ref{ex-Apqu},
for any $s_m$ satisfying $r_m < s_m < \infty$, $0 < s, \tau < \infty$ and
\begin{equation}\label{fraefwera}
\frac 1s-\frac 1{p_{0}}=\frac1{\tau}-\frac 1{\delta_{m+1}}=\frac 1{s_m}-\frac 1{p_m},
\end{equation}
and for any $U\in A^{u}_{\frac{s_m}{r_m},\frac{\tau}{r_m}} (\widehat{w})$, we have
\begin{equation}\label{wegsrt54}
\|F\|_{L^\frac{s}{r_m}(U^{\frac{s}{r_m}}\d\widehat{w})}
\lesssim
\|G \|_{L^\frac{s_m}{r_m}(U^{\frac{s_m}{r_m}}\d\widehat{w})},
\qquad
(F,G)\in\mathcal{G}.
\end{equation}

Next, we set $s := q^*$, $s_m := p^*_m$, and $\tau = \widetilde{\delta}_{m+1}$.
By assumption, $r_m < p^*_m = s_m$ holds. Given $p^*_i = p_i$ for $1 \leq i \leq m-1$, it follows that
\[
\frac 1s-\frac 1{q}
=
\frac 1{q^*}-\frac 1{q}
=
\sum_{i=1}^m \Big(\frac1{p^*_i }-\frac1{p_i}\Big)
=
\frac1{p^*_m}-\frac1{p_m}
=
\frac 1{s_m}-\frac 1{p_m}
\]
and
\[
\frac1{\tau}-\frac1{\delta_{m+1}}
=
\frac1{\widetilde{\delta}_{m+1}}-\frac1{\delta_{m+1}}
=
\frac 1{p_{m+1}}-\frac 1{p^*_{m+1}}
=
\frac 1{q^*}-\frac 1{q}
=
\frac 1{s}-\frac 1{q}.
\]
Thus, \eqref{fraefwera} is satisfied. On the other hand, \eqref{aferer} implies
$V\in
A^{u}_{\frac{p^*_m}{r_m},\frac{\widetilde{\delta}_{m+1}}{r_m}}(\widehat{w})=A^{u}_{\frac{s_m}{r_m},\frac{\tau}{r_m}} (\widehat{w})
.$
Therefore, \eqref{wegsrt54} holds for $U = V$ and some suitable choices of related parameters.
Consequently, applying Lemma \ref{ex-lem-main} $(iii)$ with $\vec{p^*}$
and $\vec{r}$ gives that for $(f, f_1, \dots, f_m) \in \mathcal{F}$,
\begin{multline*}
\|f\|_{L^{q^*}(v^{q^*})}
=
\big\|\big(f\widehat{w}^{-\frac1{r_{m+1}'}}\big)^{r_m}\big\|_{L^\frac{q}{r_m}(V^{\frac{q^*}{r_m}}\d\widehat{w})}^{\frac1{r_m}}
=
\|F\|_{L^\frac{q^*}{r_m}(V^{\frac{q^*}{r_m}}\d\widehat{w})}^{\frac1{r_m}}
\lesssim
\|G \|_{L^\frac{p^*_m}{r_m}(V^{\frac{p^*_m}{r_m}}\d\widehat{w})}^{\frac1{r_m}}
\\
=
\big( \prod_{i=1}^{m-1}\|f_i\|_{L^{p_i}(w_i^{p_{i}})}\big)
\big\|\big(f_m\widehat{w}^{-\frac1{r_m}}\big)^{r_m}\big\|_{L^\frac{q_m}{r_m}(W^{\frac{p^*_m}{r_m}}\d\widehat{w})}^{\frac1{r_m}}
=\prod_{i=1}^{m}\|f_i\|_{L^{p^*_i }(w_i^{p^*_i })}.\qed
\end{multline*}

\section{Partial weight estimates for fractional integrals and applications}\label{estimates}

\subsection{Partial weight estimates for fractional integrals in Lebesgue space}

We first establish the weighted boundedness characterization of the multilinear fractional maximal function in endpoint spaces.

\begin{thm}\label{Mam}
Let $0\leq \alpha<md$, $1< p_{1},\ldots, p_{m}\leq \infty$, $0<q< \infty$ and $\frac{1}{p}:=\sum_{i=1}^m\frac{1}{p_{i}}=\frac{1}{q}+\frac{\alpha}{d}>0$. Then
\begin{equation*}\label{YH-6}
\Big(\int_{\R^d}\big(\mathcal{M}_{\alpha,m}(\vec f)(x) w(x)\big)^{q}\Big)^{\frac{1}{q}}\lesssim \prod_{i=1}^{m}\Big(\int_{\R^d}|f_{i}(x)w_{i}(x)|^{p_{i}}dx\Big)^{\frac{1}{p_{i}}}
\end{equation*}
holds if and only if $\vec w\in A_{\vec p, q}$.
\end{thm}

\begin{proof}
When $1< p_{1},\ldots, p_{m}< \infty$, Theorem \ref{YH-6} has been proved by \cite{LOPTT2009} and \cite{M2009}. Combining this with
Corollary \ref{ex-Appq} for $r_{1}=\cdots=r_{m+1}=1$, one can conclude the boundedness for $1< p_{1},\ldots, p_{m}\leq \infty$.
The necessity is similar to that in \cite{LOPTT2009} and \cite{M2009}, it only suffices to take $f_{i} = w^{-p'_{i}}_{i}\chi_{Q}$.
\end{proof}

Consequently, we can obtain the following partial multiple weight estimates for the fractional maximal function.

\begin{thm}\label{Lp-Lq-M}
Let $0\leq \beta< \alpha< d$ and $1<p\leq q<\infty$ with $\frac{1}{p}-\frac{1}{q}=\frac{\beta}{d}$. If $u\in A_{1}\bigcap L^{\frac{d}{\alpha-\beta},1}(\R^d)$ and $w\in A_{p,q}^u$, then
$\|M_{\alpha}(f)\|_{L^{q}(u^{q}w^{q})}\lesssim \|u\|_{L^{\frac{d}{\alpha-\beta},1}(\R^d)}\|f\|_{L^{p}(w^p)}.$
\end{thm}
\begin{proof}
For $u\in A_{1}\bigcap L^{\frac{d}{\alpha-\beta},1}(\R^d)$ and $w\in A_{p,q}^u$,
it follows from \eqref{Ma-Mb} and Theorem \ref{Mam} that
\begin{equation*}
\begin{aligned}
&\|M_{\alpha}(f)\|_{L^{q}(u^{q}w^{q})}\lesssim \|u\|_{L^{\frac{d}{\alpha-\beta},1}(\R^d)} \|\mathcal{M}_{\beta,2}(f,u^{-1})\|_{L^{q}(u^{q}w^{q})}\\
&\lesssim \|u\|_{L^{\frac{d}{\alpha-\beta},1}(\R^d)}\|u^{-1}\cdot u\|_{L^{\infty}(\R^d)}\|f\|_{L^{p}(w^{p})}\lesssim \|u\|_{L^{\frac{d}{\alpha-\beta},1}(\R^d)}\|f\|_{L^{p}(w^{p})}.
\end{aligned}
\end{equation*}
\end{proof}

For $U\in L^{\frac{d}{\alpha-\beta},s}(\R^d)$ and $1<r<s<\infty$, it is known for
$M(U^{r})^{1/r}\in A_{1}\bigcap  L^{\frac{d}{\alpha-\beta},s}(\R^d)$.
\begin{cor}\label{Lp-Lq-U}
Let $0\leq  \beta< \alpha< d$ and $1<r<s<p\leq q<\infty$ with $\frac{1}{p}-\frac{1}{q}=\frac{\beta}{d}$. If $U\in L^{\frac{d}{\alpha-\beta},s}(\R^d)$, then
$\|M_{\alpha}(f)U\|_{L^{q}(w^{q})}\lesssim \|f\|_{L^{p}(w^p)},$
where $u=M(U^{r})^{1/r}$ and $w\in A_{p,q}^u$.
\end{cor}

The next lemma illustrates the relation between $I_{\alpha}$ and $M_{\alpha}$. For its proof, one can be referred to \cite{MW1974}.
\begin{lem}\label{Ia-Ma}
Suppose that $0<\alpha<d, 0<q<\infty$ and $w\in A_{\infty}$. Then there exists a constant $C>0$ independent of $f$ such that
$$\int_{\R^d}|I_{\alpha}(f)(x)|^{q}w(x)dx\leq C\int_{\R^d}|M_{\alpha}(f)(x)|^{q}w(x)dx.$$
\end{lem}

\medskip

\noindent
{\bf Proof of Theorem \ref{Lp-Lq-Ia}.} Applying Theorems \ref{Lp-Lq-M} and Lemma \ref{Ia-Ma}, we can get the weighted boundedness of $I_{\alpha}$. \qed

\subsection{Fefferman-Phong inequality related to partial Muckenhoupt weights}

As mentioned in the introduction, the author in \cite{FP1983} established that if $U \in L^{s,1}(\R^d)$ with $2<s<\infty$,
then the following inequality holds:
\begin{equation}\label{YH-11}
\|f U\|_{L^2(\R^d)} \lesssim \|\nabla f\|_{L^2(\R^d)}.
\end{equation}
In particular, if $U(x) = |x|^{-1}$, then \eqref{YH-11} corresponds to the Hardy inequality. In addition,
\eqref{YH-11} has been extended to $L^p(\R^d)$ case in \cite{CF1990}.

It is pointed out that the connection between the Riesz potential and the gradient has been established
through the following two inequalities (see \cite{CF2013, GT1998}).
The first inequality is that for $f \in C_{c}^{\infty}(\R^d)$,
\begin{equation}\label{eqfIa-1}
|f(x)| \lesssim I_{1}(|\nabla f|)(x).
\end{equation}
The second inequality is that for any bounded, convex domain $\Omega$ and for all $x \in \Omega$,
\begin{equation}\label{eqfIa-2}
|f(x) - \ave{f}_{\Omega}| \leq I_{1}(|\nabla f| \chi_{\Omega})(x).
\end{equation}
Note that by the definition of the fractional Laplacian and the properties of the Fourier transformation,
one has that for $\alpha\in [0,1]$,
$$I_{2\alpha}((-\Delta)^{\alpha}f)(x)=f(x).$$
We next give some extension of \eqref{YH-11}.
\begin{thm}\label{FP}
Let $1 < p \leq q < \infty$, $0\leq \beta < d$,  $\frac{1}{p} - \frac{1}{q}  = \frac{\beta}{d}$
and $U \in L^{\frac{d}{1-\beta},q_{0}}(\R^d)$ with $1<q_{0}\leq q$. For any $w \in A_{p,q}^u$
with $u = M(U^r)^{\frac{1}{r}}$, for some $1 < r < q_{0}$, we have that for $f \in C^{\infty}_{c}(\R^d)$,
$$\|f U\|_{L^{q}(w^{q})} \lesssim \|\nabla f\|_{L^{p}(w^{p})}.$$
\end{thm}
\begin{proof}
By \eqref{eqfIa-1}, one has
$$\|f U\|_{L^{q}(w^{q})} \lesssim \|I_{1}(\nabla f)\|_{L^{q}(u^{q} w^{q})} \lesssim \|\nabla f\|_{L^{p}(w^{p})} \qquad \text{for } \quad f\in C^{\infty}_{c}(\R^d),$$
where $U \leq u$ and Theorem \ref{Lp-Lq-Ia} are utilized.
\end{proof}

Our next result is on the fractional Laplacian.
\begin{thm}\label{FP-a}
Let $1 < p \leq q < \infty$, $0\leq \beta\leq \alpha$,  $\frac{1}{p} - \frac{1}{q}  = \frac{\beta}{d}$ and $U \in L^{\frac{d}{\alpha-\beta},q_{0}}(\R^d)$ with $1<q_{0}\leq q$.  For any $w \in A_{p,q}^u$
with $u = M(U^r)^{\frac{1}{r}}$, for some $1 < r < q_{0}$, we have that for $f \in C^{\infty}_{c}(\R^d)$,
$$\|f U\|_{L^{q}(w^{q})} \lesssim \|(-\Delta)^{\alpha} f\|_{L^{p}(w^{p})}.$$
\end{thm}

The following Hardy-Leray inequality with partial Muckenhoupt weights is a special case of Theorem \ref{FP}.
\begin{cor}(Hardy-Leray inequality)
Let $d>1$ and $1 < p \leq q < \infty$. For any $w \in A_{p,p}^u$
with $u(x)= |x|^{-1}$, we have that for $f \in C^{\infty}_{c}(\R^d)$,
\begin{equation}\label{HL}
\big\|\frac{f}{|\cdot|}\big\|_{L^{p}(w^{p})} \lesssim \|\nabla f\|_{L^{p}(w^{p})}.
\end{equation}
\end{cor}
In particular, when $w(x)=\big(\frac{|x'|}{|x|}\big)^{\frac{1}{2}}$, $d\geq 3, p=2$, inequality \eqref{HL} takes the form
\begin{equation}\label{HL2}
\int_{\R^d}\frac{|u|^{2}}{|x||x'|}dx\leq C\int_{\R^d}|\nabla u|^{2}\frac{|x'|}{|x|}dx,
\end{equation}
which strengthens the Hardy inequality. Note that \eqref{HL2} was early appeared in \cite[Theorem 1.3]{LY2021}.

We next give the Poincar\'{e}-type inequality as in Theorem \ref{FP}.

\begin{thm}\label{P}
Assume that $\Omega$ is a bounded and convex set. Let $1 < p \leq q < \infty$, $0\leq \beta < d$,  $\frac{1}{p} - \frac{1}{q}  = \frac{\beta}{d}$ and $U \in L^{\frac{d}{1-\beta},q_{0}}(\R^d)$ with $1<q_{0}\leq q$. For any $w \in A_{p,q}^u$ with $u = M(U^r)^{\frac{1}{r}}$,
for some $1 < r < q_{0}$, we have that for $f \in C^{\infty}_{c}(\R^d)$,
$$\big\|(f-\ave{f}_{\Omega})\chi_{\Omega} U\big\|_{L^{q}(w^{q})} \lesssim \big\||\nabla f|\chi_{\Omega}\big\|_{L^{p}(w^{p})}.$$
\end{thm}
\begin{proof}
By \eqref{eqfIa-2} and Theorem \ref{Lp-Lq-Ia}, one has that for $f\in C^{\infty}_{c}(\R^d)$,
$$\|(f-\ave{f}_{\Omega})\chi_{\Omega}  U\|_{L^{q}(w^{q})}  \lesssim \|I_{1}(\nabla f\chi_{\Omega})\|_{L^{q}(U^{q} w^{q})} \lesssim \||\nabla f|\chi_{\Omega}\|_{L^{p}(w^{p})}.$$
\end{proof}

\subsection{Caffarelli-Kohn-Nirenberg inequalities related to partial Muckenhoupt weights}

Let $d\geq 1$, $s, p, q, \alpha, \mu, \beta, a$ be real numbers satisfying
\begin{equation}\label{CKN-1}
s>0, p,q\geq 1, 0\leq a\leq 1,
\end{equation}
\begin{equation}\label{CKN-2}
\frac{1}{s}+\frac{\gamma_{1}}{d}>0, \frac{1}{p}+\frac{\gamma_{2}}{d}>0, \frac{1}{q}+\frac{\gamma_{3}}{d}>0,
\end{equation}
\begin{equation}\label{CKN-3}
\frac{1}{s}+\frac{\gamma_{1}}{d}=a\Big(\frac{1}{p}+\frac{\gamma_{2}-1}{d}\Big)+(1-a)\Big(\frac{1}{q}+\frac{\gamma_{3}}{d}\Big),
\end{equation}
\begin{equation}\label{CKN-4}
\gamma_{1}\leq a\gamma_{2}+(1-a)\gamma_{3},
\end{equation}
\begin{equation}\label{CKN-5}
\frac{1}{s}\leq \frac{a}{p}+\frac{1-a}{q} \quad \text{if} \quad a=0 \quad \text{or} \quad a=1 \quad \text{or} \quad
\frac{1}{s}+\frac{\gamma_{1}}{d}=\frac{1}{p}+\frac{\gamma_{2}-1}{d}=\frac{1}{q}+\frac{\gamma_{3}}{d}.
\end{equation}

The following classical interpolation inequalities were first introduced in \cite{CKN1984}.

{\bf Theorem A.} For $d\geq 1$, let $s, p,q, \gamma_{1},\gamma_{2},\gamma_{3}$ and $a$ satisfy \eqref{CKN-1} and \eqref{CKN-2}. Then
there exists a positive constant $C$ such that
\begin{equation}\label{CKN-6}
\||x|^{\gamma_{1}}f\|_{L^{s}(\R^d)}\leq C\||x|^{\gamma_{2}}|\nabla f|\|^{a}_{L^{p}(\R^d)}\||x|^{\gamma_{3}}f\|^{1-a}_{L^{q}(\R^d)}
\end{equation}
holds for all $f\in C^{\infty}_{c}(\R^d)$ if and only if \eqref{CKN-3}-\eqref{CKN-5} hold.
So far, there are extensive generalizations on \eqref{CKN-6}
(see \cite{L1986}, \cite{BT2002},\cite{BCG2005, BCG2006}, \cite{NS2018, NS2019}, \cite{Xu-Yin} and the references therein).
Next, we present the Caffarelli-Kohn-Nirenberg inequality with partial Muckenhoupt weights.

For $d\geq 1$, let $s, p,p_{0}, q, q_{0}, r, u, w, w_{1},w_{2}$ and $a$ satisfy
\begin{equation}\label{CKN-W-1}
s>0, p_{0}\geq p>1, q>0, q_{0}>r>1, 0<a\leq 1,
\end{equation}
\begin{equation}\label{CKN-W-2}
\frac{1}{s}= \frac{a}{p_{0}}+\frac{1-a}{q}\leq \frac{a}{p}+\frac{1-a}{q},
\end{equation}
\begin{equation}\label{CKN-W-3}
w=U^{a}w_{1}^{a} w_{2}^{1-a}, u=M(U^{r})^{1/r}, w_{1}\in A_{p,p_{0}}^{u},
\end{equation}
we then have
\begin{thm}\label{FP-W}
For $d\geq 1$, let $U\in L^{\frac{1}{\frac 1d-\frac{1}{p}+\frac{1}{p_{0}}},q_{0}}(\R^d)$ and $s, p, p_{0}, q, q_{0}, w, w_{1},w_{2}$ and $a$ satisfy \eqref{CKN-W-1}, \eqref{CKN-W-2} and \eqref{CKN-W-3}. Then
there exists a positive constant $C$ such that for any $f \in C^{\infty}_{c}(\R^d)$,
\begin{equation}\label{YH-13}
\|fw\|_{L^{s}(\R^d)}\leq C\|\nabla f w_{1}\|^{a}_{L^{p}(\R^d)}\|fw_{2}\|^{1-a}_{L^{q}(\R^d)}.
\end{equation}
\end{thm}
\begin{proof}
Theorem \ref{FP} is the special case of $a=1$ in \eqref{YH-13}. For $0<a<1$,
set $s_{1} = \frac{p_{0}}{a}$ and $s_{2} = \frac{q}{1-a}$. By applying Young's inequality with $\frac{1}{s}=\frac{1}{s_{1}}+\frac{1}{s_{2}}$,
we have
$$\|fw\|_{L^{s}(\R^d)} = \|f^{a}U^{a}w^{a}_{1} \cdot f^{1-a}w^{1-a}_{2}\|_{L^{s}(\R^d)} \leq \|f^{a}u^{a}w^{a}_{1}\|_{L^{s_{1}}(\R^d)} \|f^{1-a}w^{1-a}_{2}\|_{L^{s_{2}}(\R^d)}.$$
This leads to
$$\|fw\|_{L^{s}(\R^d)} \leq \|fuw_{1}\|^{a}_{L^{p_{0}}(\R^d)} \|fw_{2}\|^{1-a}_{L^{q}(\R^d)}.$$
This, together with Theorem \ref{FP} and $w_1 \in A_{p,p_{0}}^u$, yields
$\|fw\|_{L^{s}(\R^d)} \leq C \|\nabla f w_{1}\|^{a}_{L^{p}(\R^d)} \|fw_{2}\|^{1-a}_{L^{q}(\R^d)}.$
\end{proof}

\begin{rem}
Let $w(x)=|x|^{\gamma_{1}}$, $w_{1}(x)=|x|^{\gamma_{2}}$, $w_{2}(x)=|x|^{\gamma_{3}}$ and $U(x)=|x|^{1-\frac{d}{p_{0}}+\frac{d}{p}}$ with $\frac{1}{p}+\frac{\gamma_{2}}{d}<1$. Then $u\in A_{1}\cap L^{\frac{1}{\frac 1 d-\frac{1}{p_{0}}+\frac{1}{p}},q_{0}}(\R^d)$ and $w_{1}\in A^{u}_{p,p_{0}}$ hold by Remark \ref{remark-power}. Therefore, in this case, Theorem \ref{FP-W}
coincides with Theorem A under the conditions of $\frac{1}{p}+\frac{\gamma_{2}}{d}<1$ and $p\neq 1$.
\end{rem}

\begin{rem}
Let $w(x)=|x|^{\gamma_{1}}|x'|^{\alpha}$, $w_{1}(x)=|x|^{\gamma_{2}}|x'|^{\mu}$ and $w_{2}(x)=|x|^{\gamma_{3}}|x'|^{\beta}$ with some numbers $\alpha,\mu,\beta, \gamma_{1}, \gamma_{2}$ and $\gamma_{3}$ being determined in Theorem \ref{FP-W}.
We can obtain such an anisotropic Caffarelli-Kohn-Nirenberg inequality for some appropriate indices:
\begin{equation}\label{LY2023-eq-1}
\||x|^{\gamma_{1}}|x'|^{\alpha}f\|_{L^{s}(\R^d)}\leq C\||x|^{\gamma_{2}}|x'|^{\mu}|\nabla f|\|^{a}_{L^{p}(\R^d)}\||x|^{\gamma_{3}}|x'|^{\beta}f\|^{1-a}_{L^{q}(\R^d)}.
\end{equation}
Note that \eqref{LY2023-eq-1} was recently established in \cite{LY2023} for studying the asymptotic stability of solutions
to the Navier-Stokes equations.
\end{rem}

\section{Proof of Theorem \ref{main}}\label{Thm1.1}

Given $1<s<\infty$, the $s$-fractional maximal function is defined by
$M_{\alpha,s}(f)(x)=\big(M_{\alpha}(|f|^{s})(x)\big)^{\frac{1}{s}}.$
The auxiliary operator is defined by
$$\mathcal{C}(b,f)(x)=\sup_{\epsilon>0}\big|\int_{|x-y|>\epsilon}\frac{b(x)-\ave{b}_{Q(x,\epsilon)}}{|x-y|^{d-\alpha}}f(y)dy\big|.$$
In \cite{C1982}, the author proved the pointwise estimate $[b,I_{\alpha}](f)(x)\leq \mathcal{C}(b,f)(x)$ for $b\in {\rm BMO}$, and established the following result.
\begin{lem}\label{main-lem}
Let $\gamma_1>0$ be sufficiently small and $\gamma_2>0$ sufficiently large. Then for any $p$ such that $1<p<\frac{d}{\alpha}$
and for any $\lambda>0$, we have
\begin{align*}
\big|
\{x: \mathcal{C}(b,f)(x)>\gamma_2\lambda, \|b\|_{\rm BMO}\big(I_{\alpha}(|f|)(x)+M_{\alpha,s}(f)(x)\big)\leq \gamma_1 \lambda\}
\big|
\leq C\gamma_1|\{x:\mathcal{C}(b,f)(x)>\lambda\}|.
\end{align*}
\end{lem}

Next, we establish the sufficiency and necessity of the related estimate for the commutator $[b, I_{\alpha}]$ as follows.

\noindent
{\bf Proof of Theorem \ref{main}.}
By Lemma \ref{main-lem}, we have that for any $b\in {\rm BMO}$, $1<s_0<q$ and $w\in A_{p,q}^u$,
\begin{align*}
&\int_{\R^d}|\mathcal{C}(b,f)(x)|^{q}u(x)^{q}w(x)^qdx\\
&=C\int_{0}^{\infty}\lambda^{q-1}(uw)^q\big(\{x\in \R^d:\mathcal{C}(b,f)(x)>\gamma_2\lambda\}\big)d\lambda\\
&\leq C\gamma\int_{0}^{\infty}\lambda^{q-1}(uw)^q\big(\{x\in \R^d:\mathcal{C}(b,f)(x)>\lambda\}\big)d\lambda\\
&\quad+ C\int_{0}^{\infty}\lambda^{q-1}(uw)^q\big(\{x\in \R^d:\|b\|_{\rm BMO}\big(I_{\alpha}(|f|)(x)+M_{\alpha,s_0}(f)(x)\big)>\gamma_1\lambda\}\big)d\lambda\\
&=C\gamma_1\int_{\R^d}\mathcal{C}(b,f)(x)u(x)^{q}w(x)^{q}dx+C\|b\|^q_{\rm BMO}\int_{\R^d}\big(I_{\alpha}(|f|)(x)+M_{\alpha,s_0}(f)(x)\big)^{q}u(x)^{q}w(x)^qdx.
\end{align*}
Taking $\gamma_1>0$ so small that $C\gamma_1\leq \frac12$, we obtain
$$\|\mathcal{C}(b,f)\|_{L^{q}(u^{q}w^q)}\lesssim \|b\|_{\rm BMO}\big(\|I_{\alpha}(|f|)\|_{L^{q}(u^{q}w^q)}+\|M_{\alpha,s_0}(|f|)\|_{L^{q}(u^{q}w^q)}\big).$$
Choosing $\epsilon$ small enough such that $0<\epsilon<RH_{w^{-q'}}-1$,
$\frac{1}{s_0}=\frac{1}{q}+\frac{1}{q'(1+\epsilon)}<1$
and $1<s_0<\min\{s,RH_{u}\}$. It follows that
$s_0\cdot (\frac{q}{s_0})'<q'(1+\epsilon)$, $u^{s_0}\in A_{1}$ and $w^{s_0}\in A^{u^{s_0}}_{p/s_0, q/s_0}.$
Thus, $$\|[b,I_{\alpha}](f)\|_{L^{q}(u^{q}w^q)}\lesssim\|\mathcal{C}(b,f)\|_{L^{q}(u^{q}w^q)} \lesssim \|u\|_{M^{\frac{d}{\alpha-\beta}}_{s}(\R^d)}\|f\|_{L^{p}(w^p)}$$ holds by Theorems \ref{Lp-Lq-Ia} and \ref{Lp-Lq-M}. This completes the proof of sufficiency in Theorem \ref{main}.

On the other hand, it follows from \cite[Corollary 2.4]{CR2018} that $b\in {\rm BMO}$ is a necessary condition for the boundedness of $[b,I_{\alpha}]$ from $L^{p}(v^{p})$  to $L^{q}(v^{q})$, where $0<\alpha<d$, $1<p<q<\infty$, $\frac \alpha d=\frac{1}{p}-\frac{1}{q}$ and $v\in A_{p,q}$.
Together with Theorem \ref{D-OD}, this yields the necessity result in Theorem \ref{main}. \qed

\bigskip

{\bf Acknowledgements}. The authors would like to thank Professor David Cruz-Uribe and Professor Li Kangwei for
many fruitful discussions.

\vskip 0.2 true cm
{\bf \color{blue}{Conflict of Interest Statement:}}

\vskip 0.2 true cm

{\bf The authors declare that there is no conflict of interest in relation to this article.}

\vskip 0.2 true cm
{\bf \color{blue}{Data availability statement:}}

\vskip 0.2 true cm

{\bf  Data sharing is not applicable to this article as no data sets are generated
during the current study.}

\vskip 0.2 true cm


\begin{thebibliography}{88I}

\parskip=0.1cm
\bibitem{ALM2023}
E. Airta,  Li Kangwei, H. Martikainen,
Two-weight inequalities for multilinear commutators in product spaces.
Potential Anal., 59 (2023), 1745--1792.


\bibitem{AM2007}
P. Auscher, J. M. Martell,
Weighted norm inequalities, off-diagonal estimates and elliptic operators. {P}art {I}. {G}eneral operator theory and weights.
Adv. Math., 212 (2007), 225--276.

\bibitem{BT2002}
M. Badiale, G. Tarantello, A Sobolev-Hardy inequality with applications to
a nonlinear elliptic equation arising in astrophysics.
Arch. Ration. Mech. Anal., 163 (2002), 259--293.

\bibitem{BCG2005}
H. Bahouri, J. Y. Chemin, I. Gallagher, Sharper Hardy inequalities.
C. R. Math. Acad. Sci. Paris., 341 (2005), 89--92.

\bibitem{BCG2006}
H. Bahouri, J. Y. Chemin, I. Gallagher, Refined Hardy inequalities. Ann.
Sc Norm. Super. Pisa Cl. Sci., 5 (2006), 375-391.

\bibitem{CKN1984}
L. Caffarelli, R. Kohn, L. Nirenberg,
First order interpolation inequalities with weights. Composition Math. 53 (1984), 259--275.

\bibitem{CR2016}
R. E. Castillo and H. Rafeiro, An introductory course in Lebesgue spaces. Cham: Springer, 2016.

\bibitem{C2016}
L. Chaffee,
Characterizations of bounded mean oscillation through commutators of bilinear singular integral operators.
Proc. Roy. Soc. Edinburgh Sect. A, 146 (2016), 1159--1166.

\bibitem{CR2018}
L. Chaffee, D. Cruz-Uribe, Necessary conditions for the boundedness of linear and bilinear commutators on Banach function spaces.
Math. Inequal. Appl., 21 (2018), 1-16.

\bibitem{C1982}
S. Chanillo, A note on commutators. Indiana Univ. Math. J., 31 (1982), 7--16.

\bibitem{CF1990}
F. Chiarenza, M. Frasca, A remark on a paper by C. Fefferman.
Proc. Amer. Math. Soc., 108 (1990), 407--409.

\bibitem{CF1974}
R. Coifman, C. Fefferman, Weighted norm inequalities for maximal functions and singular integrals.
Studia Math., 51 (1974), 241--250.

\bibitem{CRW1976}
R. R. Coifman, R. Rochberg, G. Weiss,
Factorization theorems for Hardy spaces in several variables.
Ann. of Math., 103 (1976), 611--635.


\bibitem{CU2016}
D. Cruz-Uribe, Two weight inequalities for fractional integral operators and commutators,
World Scientific Publishing Co. Pte. Ltd., Hackensack, NJ, 2017, 25--85.

\bibitem{CF2013}
D. Cruz-Uribe, A. Fiorenza, Variable Lebesgue spaces, in: Foundations and Harmonic Analysis, in: Applied and Numerical
Harmonic Analysis, Birkh\"{a}user/Springer, Heidelberg, 2013.


\bibitem{CMP2011}
D. Cruz-Uribe, J. M. Martell, C. P\'{e}rez, Weights, extrapolation and the theory of Rubio de Francia,
operator theory: Advances and Applications, vol. 215, Birkh\"{a}user/Springer Basel AG, Basel, 2011.

\bibitem{CM2012}
D. Cruz-Uribe, K. Moen,
Sharp norm inequalities for commutators of classical operators.
Publ. Math., 56 (2012), 147--190.


\bibitem{D2011}
J. Duoandikoetxea, Extrapolation of weights revisited: new proofs and sharp bounds.
J. Funct. Anal., 260 (2011), 1886--1901.


\bibitem{FP1983}
C. Fefferman, The uncertainty principle.
Bull. Amer. Math. Soc., (N.S.) 9 (1983), 129--206.

\bibitem{FM1974}
C. Fefferman, B. Muckenhoupt,
Two nonequivalent conditions for weight functions.
Proc. Amer. Math. Soc., 45 (1974), 99--104.


\bibitem{FS1972}
C. Fefferman, E. M. Stein, $H^p$ spaces of several variables.
Acta Math., 129 (1972), 137--193.


\bibitem{GT1998}
D. Gilbarg, N. S. Trudinger, Elliptic partial differential equations of second order. Classics in Mathematics,
Springer-Verlag, Berlin, 2001, Reprint of the 1998 edition.

\bibitem{G2004}
L. Grafakos, Classical and modern Fourier analysis. Prentice Hall, 2004.


\bibitem{GLW2020}
Guo Weichao,  Lian Jiali,  Wu Huoxiong.
The unified theory for the necessity of bounded commutators and applications.
J. Geom. Anal., 30 (2020), 3995--4035.


\bibitem{HMS1988}
E. Harboure, R. A. Mac\'ias, C. Segovia, Extrapolation results for classes of weights.
Am. J. Math., 110 (1988), 383--397.


\bibitem{JN1961}
F. John, L. Nirenberg, On functions of bounded mean oscillation.
Comm. Pure Appl. Math., 2 (1961), 415--426.

\bibitem{J1980}
P. W. Jones, Factorizaiton of $A_{p}$ weights.
Ann. of Math., 111 (1980), 511-530.


\bibitem{LMPT2010}
M. T. Lacey, K. Moen, C. P\'{e}rez, R. H. Torres,
Sharp weighted bounds for fractional integral operators.
J. Funct. Anal., 259 (2010), 1073--1097.


\bibitem{LOPTT2009}
A. K. Lerner, S. Ombrosi, C. P\'{e}rez, R. H. Torres, R. Trujillo-Gonz\'{a}lez,
New maximal functions and multiple weights for the multilinear Calder\'{o}n-Zygmund theory.
Adv. Math., 220 (2009), 1222--1264.

\bibitem{LMO2020}
Li Kangwei, J. M. Martell, S. Ombrosi,
Extrapolation for multilinear Muckenhoupt classes and applications.
Adv. Math., 373 (2020), 107286.


\bibitem{LMV2021}
Li Kangwei, H. Martikainen, E. Vuorinen,
Genuinely multilinear weighted estimates for singular integrals in product spaces.
Adv. Math., 393 (2021), 108099.


\bibitem{LY2021}
Li Yanyan,  Yan Xukai,
Asymptotic stability of homogeneous solutions to the Navier-Stokes equations in $\R^3$.
J. Differential Equations,  297 (2021), 226--245.

\bibitem{LY2023}
Li Yanyan,   Yan Xukai,
Anisotropic Caffarelli-Kohn-Nirenberg type inequalities.
Adv. Math., 419 (2023), 108958.

\bibitem{L1986}
Lin Changshou, Interpolation inequalities with weights.
Comm. Partial Differential Equations,  11 (1986), 1515--1538.


\bibitem{L1950}
G. G. Lorentz, Some new functional spaces. Ann. Math. (2), 51(1950), 37--55.

\bibitem{L1951}
G. G. Lorentz, On the theory of spaces $\Lambda$. Pac. J. Math., 1(1951), 411--429.


\bibitem{M2009}
K. Moen, Weighted inequalities for multilinear fractional integrals.
Collect. Math., 60 (2009), 213--238.


\bibitem{MW1974}
B. Muckenhoupt, R. L. Wheeden,
Weighted norm inequalities for fractional integrals.
Trans. Amer. Math. Soc., 192 (1974), 261--274.


\bibitem{N1983}
C. J. Neugebauer, Inserting $A_p$-weights.
Proc. Amer. Math. Soc., 87 (1983), 644--648.


\bibitem{NS2018}
H. Nguyen, M. Squassina, Fractional Caffarelli-Kohn-Nirenberg inequalities. J. Funct. Anal., 274 (2018), 2661--2672.

\bibitem{NS2019}
H. Nguyen, M. Squassina, On Hardy and Caffarelli-Kohn-Nirenberg inequalities. J. Anal. Math., 139 (2019), 773--797.


\bibitem{P1990}
C. P\'{e}rez, Two weight norm inequalities for Riesz potentials and uniform $L^p$-weighted Sobolev inequalities.
Indiana Univ. Math. J., 39 (1990), 31--44.

\bibitem{P1994}
C. P\'{e}rez,
Two weighted inequalities for potential and fractional type maximal operators.
Indiana Univ. Math. J., 43 (1994), 663--683.

\bibitem{RDF}
J. L. Rubio de Francia, Factorization theory and $A_p$ weights.
Amer. J. Math., 106 (1984), 533--547.

\bibitem{S1984}
E. T. Sawyer,
A two weight weak type inequality for fractional integrals.
Trans. Amer. Math. Soc., 281 (1984), 339--345.

\bibitem{S1988}
E. T. Sawyer,
A characterization of two weight norm inequalities for fractional and Poisson integrals.
Trans. Amer. Math. Soc., 308 (1988), 533--545.

\bibitem{ST1991}
C. Segovia, J. L. Torrea,
Weighted inequalities for commutators of fractional and singular integral.
Publ. Math., 35 (1991), 209--235.

\bibitem{Stein1970}
E. M. Stein,
Singular integrals and differentiability properties of functions.
Princeton Mathematical Series, Vol. 30, pp. xiv+290
(Princeton University Press, Princeton, N.J., 1970).

\bibitem{W2023}
Wang Dinghuai,
The necessity theory for commutators of multilinear singular
integral operators: the weighted case.
Studia Math., 272 (2023), 1--33.

\bibitem{Xu-Yin}
Xu Gang, Yin Huicheng,
On global multidimensional supersonic flows with vacuum states at infinity.
Arch. Ration. Mech. Anal., 218 (2015), No. 3, 1189--1238.


\end{thebibliography}
\end{document}